\documentclass{article}
\usepackage{amssymb,latexsym,amsmath,amsthm,wasysym}
\usepackage[dvips]{graphics}
\usepackage{pstricks}
\usepackage{enumerate}

\newtheorem{theorem}{Theorem}[section]
\newtheorem{lemma}[theorem]{Lemma}
\newtheorem{corollary}[theorem]{Corollary}
\newtheorem{proposition}[theorem]{Proposition}

\theoremstyle{definition}
\newtheorem{definition}[theorem]{Definition}
\newtheorem{remark}[theorem]{Remark}

\newcommand{\Cok}{{ \rm Cok }}
\newcommand{\Coker}{{ \rm Coker }}
\newcommand{\End}{{ \rm End }}

\newcommand{\Ext}{{ \rm Ext }}
\newcommand{\Hom}{{ \rm Hom }}

\newcommand{\Res}{{ \rm Res }}

\newcommand{\add}{{ \rm add }}

\newcommand{\Ker}{{ \rm Ker }}
\newcommand{\Mod}{{ \rm Mod }}

\newcommand{\Tor}{{ \rm Tor }}

\newcommand{\rad}{{ \rm rad }}

\newcommand{\g}{\hbox{-}}
\newcommand{\Endol}[1]{\hbox{\rm endol}\,(#1)}

\newcommand{\hueca}[1]{\mathbb{#1}}
\renewcommand{\mod}{{\rm mod }}

\newcommand{\rightmap}[1]{\smash{\mathop{\hbox to
20pt{\rightarrowfill}}\limits^{#1}}}
\newcommand{\leftmap}[1]{\smash{\mathop{\hbox to
20pt{\leftarrowfill}}\limits^{#1}}}

\newcommand{\longrightmap}[1]{\smash{\mathop{\hbox to
4cm{\rightarrowfill}}\limits^{#1}}}
\newcommand{\longleftmap}[1]{\smash{\mathop{\hbox to
4cm{\leftarrowfill}}\limits^{#1}}}
\newcommand{\medrightmap}[1]{\smash{\mathop{\hbox to
2cm{\rightarrowfill}}\limits^{#1}}}
\newcommand{\medleftmap}[1]{\smash{\mathop{\hbox to
2cm{\leftarrowfill}}\limits^{#1}}}
  
\newcommand{\shortlmapup}[1]
{\uparrow\rlap{$\vcenter{\hbox{$\scriptstyle#1$}}$}}

\newcommand{\shortlmapdown}[1]
{\downarrow\rlap{$\vcenter{\hbox{$\scriptstyle#1$}}$}}

\newcommand{\shortrmapdown}[1]
{\big\downarrow\rlap{$\vcenter{\hbox{$\scriptstyle#1$}}$}}

\newcommand{\shortrmapup}[1]
{\big\uparrow\rlap{$\vcenter{\hbox{$\scriptstyle#1$}}$}}

\begin{document}
   \setcounter{page}{1}

\date{}
\title{\bf Generic modules for the category of
modules filtered by standard modules}
\vskip-1cm
\author{R. Bautista, E. P\'erez and L. Salmer\'on}
\vskip-1cm
\maketitle

\renewcommand{\thefootnote}{}

\footnote{2010 \emph{Mathematics Subject Classification}:
 16G60, 16G70, 16G20.}

\footnote{\emph{Keywords and phrases}: differential tensor
algebras, ditalgebras,  quasi-hereditary algebras, standardly stratified algebras, homological systems,
  tame algebras, reduction functors, generic modules.}
  
\begin{abstract}
 Here we show that, given a finite homological system $({\cal P},\leq,\{\Delta_u\}_{u\in {\cal P}})$ for a finite-dimensional algebra $\Lambda$ over an algebraically closed field, the category ${\cal F}(\Delta)$ of $\Delta$-filtered modules is tame if and only if, for any $d\in \hueca{N}$, there are only finitely many isomorphism classes of generic $\Lambda$-modules adapted to ${\cal F}(\Delta)$ with endolength $d$. We study the relationship between these generic modules and one-parameter families of indecomposables in ${\cal F}(\Delta)$. This study applies in particular to the category of modules filtered by standard modules for  standardly stratified algebras. This article includes a correction of an error in \cite{bpsqh}.
\end{abstract}

\section{Introduction}
Throughout this work $k$ denotes a fixed algebraically closed field. For any $k$-algebra $\Lambda$, we will denote by $\Lambda\g \Mod$ the category of left $\Lambda$-modules and by $\Lambda\g\mod$ its full subcategory of finite-dimensional modules.

  Recall that given a finite-dimensional $k$-algebra $\Lambda$ and a $\Lambda$-module $G$,   the
\emph{endolength} of $G$ is its length as a right
$\End_\Lambda(G)^{op}$-module. The $\Lambda$-module $G$ is called
\emph{generic} if it is indecomposable, of infinite length
as a $\Lambda$-module, but with finite endolength. The
algebra $\Lambda$ is called \emph{generically tame} if, for
each $d\in \hueca{N}$, there are only   finitely many 
isoclasses of generic $\Lambda$-modules with endolength $d$.
This notion was introduced by W. W. Crawley-Boevey in
\cite{CB2}, where he proved  that the class of generically tame algebras
coincides 
with  the class of tame algebras. Recall that an algebra
$\Lambda$ is \emph{tame} 
if  the pairwise
non-isomorphic indecomposable modules in each dimension can
be parametrized by a finite number of one-parameter
families. Thus, the
notion of generic tameness, permits to  generalize the notion of
tameness for arbitrary fields. 

In order to be more precise, we recall  some definitions
from \cite{BPS2}. 

\begin{definition} For any $k$-algebra $R$ and any  $M\in   R\g\Mod$,
    the \emph{endolength} of
$M$, denoted by $\Endol {M}$,  is its
length as an $\End_R
(M)^{op}$-module.  The
$R$-module $M$ is called \emph{pregeneric} iff it is an
 indecomposable with finite endolength but with  infinite dimension over $k$.
\end{definition}

If $R$ is a finite-dimensional $k$-algebra, the notion of
pregeneric $R$-module coincides with the usual notion of
generic $R$-module.

In this article we study generic $\Lambda$-modules adapted to the category of $\Delta$-filtered modules ${\cal F}(\Delta)$ in the case where $\Lambda$ is a finite-dimensional algebra and $({\cal P},\leq,\{\Delta_u\}_{u\in {\cal P}})$ is a homological system for $\Lambda$. In the following, we recall the basic terminology from \cite{MSX} and \cite{hsb}.

A \emph{preordered set} $({\cal P},\leq)$ is a non-empty set ${\cal P}$ together with a reflexive transitive relation $\leq$. We have the equivalence relation in ${\cal P}$ defined by $u\sim v$ iff $u\leq v$ and $v\leq u$. The set $\overline{\cal P}:={\cal P}/\sim$  of equivalences classes is partially ordered by the relation $\overline{u}\leq \overline{v}$ iff $u\leq v$, where $\overline{u}$ and $\overline{v}$ denote the equivalence classes of $u$ and $v$, respectively.

 \begin{definition}\label{D: homological system y F(Delta)}
  Given a finite-dimensional $k$-algebra $\Lambda$, a \emph{(finite) homological system $({\cal P},\leq,\{\Delta_u\}_{u\in {\cal P}})$ for $\Lambda$} consists of
  a finite preordered set $({\cal P},\leq)$ and
  a family of pairwise non-isomorphic indecomposable finite-dimensional $\Lambda$-modules $\{\Delta_u\}_{u\in {\cal P}}$ satisfying the following two conditions:
  \begin{enumerate}
  \item $\Hom_\Lambda(\Delta_u,\Delta_v)\not=0$ implies $u\leq v$;
  \item $\Ext^1_\Lambda(\Delta_u,\Delta_v)\not=0$ implies $u\leq v$ and $u\not\sim v$.
  \end{enumerate}

  Given such a homological system, we set   $\Delta:=\{\Delta_u\mid u\in {\cal P}\}$ and denote by $\widetilde{\cal F}(\Delta)$  the full subcategory of $\Lambda\g\Mod$ consisting of the trivial module $0$ and all those $M\in \Lambda\g\Mod$ which admit a \emph{$\Delta$-filtration},  that is a filtration of submodules
$$0=M_t\subseteq M_{t-1}\subseteq \cdots\subseteq M_1\subseteq M_0=M$$
such that $M_j/M_{j+1}$ is isomorphic to a (possibly infinite) direct sum of modules in $\Delta$, for each $j\in [0,t-1]$. We denote by ${\cal F}(\Delta)$ the full  subcategory of $\widetilde{\cal F}(\Delta)$ formed by its finite-dimensional objects. 

 A finite homological system ${\cal H}=({\cal P},\leq,\{\Delta_u\}_{u\in {\cal P}})$ for a finite-dimensional $k$-algebra $\Lambda$ is called \emph{admissible} if $\Lambda\in {\cal F}(\Delta)$ and the number of isoclasses of indecomposable projective $\Lambda$-modules coincides with the cardinality of ${\cal P}$.
 In this case, the modules in  $\Delta$ are determined up to isomorphism and are known as \emph{the standard modules};  $\Lambda$ is called a \emph{pre-standardly stratified algebra}, see \cite[(2.11)]{hsb}.

A pre-standardly stratified algebra $\Lambda$, with preordered index set $({\cal P},\leq)$, is called \emph{standardly stratified} if $({\cal P},\leq)$ is a  partial order. A standardly stratified algebra  $\Lambda$, equipped with the poset $({\cal P},\leq)$, is \emph{quasi-hereditary} iff  $\End_\Lambda(\Delta_u)\cong k$, for each $u\in {\cal P}$.
  \end{definition}

\begin{definition}\label{D: generico para subcat cal(C) de Lambda-mod}
Let $({\cal P},\leq,\{\Delta_u\}_{u\in {\cal P}})$ be a homological system for a finite-dimensional algebra $\Lambda$.
 Then, 
 \begin{enumerate}
  \item A $\Lambda$-module $G$ is called a \emph{generic module for ${\cal F}(\Delta)$} iff  $G$ is  a generic $\Lambda$-module with $G\in \widetilde{\cal F}(\Delta)$. 
  \item The category ${\cal F}(\Delta)$ is called \emph{generically tame} iff, for each $d\in \hueca{N}$, there are only finitely many non-isomorphic generic modules for ${\cal F}(\Delta)$ with endolength $d$.
 \end{enumerate}
\end{definition}

In 2021, we uploaded to arXiv a preprint with the results for ${\cal F}(\Delta)$ of this article, in  the particular case of the homological system of standard modules over a quasi-hereditary algebra.
 This article is a revised version of that one, where some of the arguments are simplified and the results are  generalized  for
arbitrary homological systems over finite-dimensional algebras.

This article can be seen as a continuation of \cite{bpsqh}. There we prove a tame-wild  dichotomy theorem for ${\cal F}(\Delta)$, namely: ${\cal F}(\Delta)$ is tame or wild, but not both, where the notions of tameness and wildness for ${\cal F}(\Delta)$ are made explicit in \cite[(1.2)]{bpsqh}. 
The main results of this article are the following two statements, where by a \emph{ rational algebra} we mean any $k$-algebra  of the form $k[x]_f$, for some $f\in k[x]\setminus{\{0\}}$.

\begin{theorem}\label{T:propiedades de gens de F(Delta)} Assume that  $\Lambda$ is a finite-dimensional algebra and consider any homological system 
 $({\cal P},\leq,\{\Delta_u\}_{u\in {\cal P}})$  for  $\Lambda$. Then, if  ${\cal F}(\Delta)$ is not wild and  $G$ is a generic $\Lambda$-module for  ${\cal F}(\Delta)$, the following statements hold:
\begin{enumerate}
 \item There is a rational algebra $\Gamma_G$ and a $\Lambda\g \Gamma_G$-bimodule $Z_G$ such that as a right $\Gamma_G$-module $Z_G$ is free of rank equal to the endolength of $G$. If $Q_G$ is the field of fractions of $\Gamma_G$, then $G\cong Z_G\otimes_{\Gamma_G}Q_G$.
 \item The functor $Z_G\otimes_{\Gamma_G}-:\Gamma_G\g\Mod\rightmap{}\widetilde{\cal F}(\Delta)$ preserves isoclasses and indecomposability.
 \item  The functor $Z_G\otimes_{\Gamma_G}-:\Gamma_G\g\mod\rightmap{}{\cal F}(\Delta)$ preserves Auslander-Reiten sequences. 
\end{enumerate}
\end{theorem}

In the following, when we say that \emph{almost all objects in a class ${\cal C}$ of modules satisfy some property}, we mean that every object in this class has this property, with the exception of those lying in a finite union of isoclasses in ${\cal C}$. 

\begin{theorem}\label{T: main thm} Assume that  $\Lambda$ is a finite-dimensional algebra and 
consider any homological system $({\cal P},\leq,\{\Delta_u\}_{u\in {\cal P}})$   for  $\Lambda$.  Then, ${\cal F}(\Delta)$ is generically tame if and only if  it is tame.

Moreover, if ${\cal F}(\Delta)$ is generically tame, for any $d\in \hueca{N}$,  
almost every indecomposable $\Lambda$-module $M\in {\cal F}(\Delta)$ with  $\dim_kM\leq d$ is of the form $M\cong Z_G\otimes_{\Gamma_G}N$, for some generic $\Lambda$-module  $G$ for ${\cal F}(\Delta)$ with $\Endol{G} \leq d$ and some indecomposable $N\in \Gamma_G\g\mod$, with the notation of (\ref{T:propiedades de gens de F(Delta)}).
\end{theorem}

The proofs of the preceding results rely, in case the given homological system is admissible, on the construction of an interlaced  weak ditalgebra $\underline{\cal A}(\Delta)$ associated to the
category ${\cal F}(\Delta)$, in such a way that ${\cal F}(\Delta)$ is generically tame iff $\underline{\cal A}(\Delta)$ is pregenerically tame in the  sense of the following definition. The construction of $\underline{\cal A}(\Delta)$ in the case of quasi-hereditary algebras is due to Koenig, K\"ulshammer, and
Ovsienko, see \cite{KKO}. The construction for a general pre-standardly stratified algebra can be found in \cite{hsb}. The proofs of our main results in the case of a general homological system are reduced to the admissible case using a construction due to Mendoza, S\'aenz, and Xi, which associates to a general homological system $({\cal P},\leq,\{\Theta_v\}_{v\in {\cal P}})$ a pre-standardly stratified algebra with an admissible homological system $({\cal P},\leq,\{\Delta_v\}_{v\in {\cal P}})$ such that ${\cal F}(\Theta)$ and ${\cal F}(\Delta)$ are equivalent as exact categories, see \cite{MSX} and \cite[\S11]{bpsqh}.

We will assume some familiarity with the language introduced
in the first sections of \cite{bpsqh}. So $\underline{\cal A}\g\Mod$
(resp. $\underline{\cal A}\g\mod$) denotes the category of modules
(resp. finite-dimensional modules) over an  interlaced weak ditalgebra $\underline{\cal
A}$ and, given $M\in \underline{\cal A}\g\Mod$, we denote by
$\End_{\underline{\cal A}}(M)$ its endomorphism algebra in $\underline{\cal
A}\g\Mod$.

\begin{definition}\label{D: endol y genericos}
Let $\underline{\cal A}=({\cal A},I)$ be an interlaced weak ditalgebra, with layer $(R,W)$, see \cite[(2.5) and (4.1)]{bpsqh}.  Given $M\in \underline{\cal A}
\g\Mod$, denote by $E_M:=\End_{\underline{\cal A}}(M)^{op}$ the opposite of its endomorphism algebra. Then, $M$ admits a structure of $R\g E_M$-bimodule, where $m\cdot (f^0,f^1)=f^0(m)$, for $m\in M$ and $(f^0,f^1)\in E_M$. By definition, the \emph{endolength} of $M$, denoted by $\Endol{M}$, is the length of $M$ as a right $E_M$-module.

A module $M\in \underline{\cal A}\g\Mod$ is called \emph{pregeneric} iff $M$ is indecomposable, with finite endolength but with infinite dimension over the ground field $k$. The interlaced weak ditalgebra $\underline{\cal A}$ is called \emph{pregenerically tame} iff, for each natural number $d$, there are only finitely many isoclasses of pregeneric $\underline{\cal A}$-modules of endolength $d$.
\end{definition}

The preceding notion extends the one introduced in \cite[\S2]{BPS2} for layered ditalgebras. As remarked there, 
if $B$ is any $k$-algebra, we have the corresponding regular ditalgebra ${\cal A}$ with layer $(B,0)$ and we can identify canonically the categories ${\cal A}\g\Mod$ with $B\g\Mod$. Thus, if $B$ is a finite-dimensional algebra, the notions of \emph{pregeneric module} and \emph{pregeneric tameness} for the algebra $B$ coincide with the usual notions of \emph{generic module} and \emph{generic tameness}. 

The interlaced weak ditalgebra $\underline{\cal A}(\Delta)$ associated to the category ${\cal F}(\Delta)$, mentioned before, is a ${\cal P}$-oriented interlaced  weak ditalgebra, which means that its biquiver, see \cite[(1.5)]{bpsqh}, has the special geometrical characteristics related to the preordered set ${\cal P}$ explained below.

\begin{definition}\label{D: biquiver P-orientado}
Let ${\cal P}=({\cal P},\leq)$ be a finite preordered set and $\hueca{B}$ a biquiver with set of points ${\cal P}$, then we say that $\hueca{B}$ is \emph{${\cal P}$-oriented} iff
 \begin{enumerate}
  \item  Whenever there is a dashed arrow from $u$ to $v$, we have $u\leq v$;
  \item Whenever there is a solid arrow from $u$ to $v$, we have
  $\overline{u}<\overline{v}$, .
 \end{enumerate}
 A weak ditalgebra ${\cal A}=(T,\delta)$ will be called a \emph{${\cal P}$-oriented weak ditalgebra} iff its underlying tensor algebra $T$ is the tensor algebra of a ${\cal P}$-oriented biquiver. An interlaced weak ditalgebra $\underline{\cal A}=({\cal A}, I)$ is called \emph{${\cal P}$-oriented}  iff ${\cal A}$ is ${\cal P}$-oriented. 
 \end{definition}

A  biquiver $\hueca{B}$ is called \emph{directed} iff it admits no oriented
(non-trivial) cycle (composed by any kind of arrows). In a ${\cal P}$-oriented biquiver $\hueca{B}$, the only (non-trivial) oriented cycles (composed of any kind of arrows) consist of dashed arrows between points in the same $\sim$ class. Such a biquiver $\hueca{B}$ is not necessarily directed.

The proof of the main theorem (\ref{T: main thm}) relies on the following result. 

\begin{theorem}\label{T: cal(A) pregen tame sii tame} Let ${\cal P}$ be a finite preordered set. Suppose that $\underline{\cal A}=({\cal A},I)$ is a ${\cal P}$-oriented triangular interlaced weak ditalgebra, where $I$ is an ideal of $A$ contained in the radical of $A$. Then,  $\underline{\cal A}$ is pregenerically tame iff it is tame.
\end{theorem}

The proof of this last statement will follow the line of reasoning of \cite{bpsqh}, adapting the arguments to handle pregeneric modules. This strategy is similar to the one used by Crawley-Boevey to transit from \cite{CB1} to \cite{CB2}. In order to present a readable argument for the proof of our main results, we could not avoid to recall many definitions and constructions from \cite{CB2} and \cite{bpsqh}. Section 2 includes a correction to \cite{bpsqh}, which required an adjustement of the definition of triangularity for an interlaced weak ditalgebra. This correction does not modify the main arguments and results of \cite{bpsqh}.

\section{Triangularity and a correction to \cite{bpsqh}}  \label{S: endogen}

In the following lemmas, we recall  elementary reduction procedures on triangular interlaced weak ditalgebras studied in \cite[\S6]{bpsqh}, and we describe their effect on the endolength. 

Here, as in \cite{BPS2}, given a layered interlaced weak ditalgebra $\underline{\cal A}=({\cal A},I)$ with underlying weak ditalgebra ${\cal A}=(T,\delta)$ and layer $(R,W)$, we denote by $A=[T]_0=T_R(W_0)$ the subalgebra of $T=T_R(W)$ formed by the zero-degree elements and, by $V=[T]_1=AW_1A$ the $A$-$A$-subbimodule of $T$ formed by the  elements with degree one. Given an ideal $I$ of $A$, we will use the notation $\langle I\rangle_2:=IV^2+VIV+V^2I$ and $\langle I\rangle_3:=IV^3+VIV^2+V^2IV+V^3I$. 

We take this opportunity to correct an error of \cite{bpsqh} in the construction of the reduction by regularization, which requires an adjustment in the definition of a triangular interlaced weak ditalgebra given there. The error is that \cite[(6.1)]{bpsqh} was applied for this construction, but 
 the assumption on the commutativity of the diagrams does not hold in the regularization case. 
 The correct construction, which is almost the same as the one given there, is triangular in the following sense.

\begin{definition}\label{D: triangular interlaced weak ditalg}
Let $({\cal A},I)$ be an interlaced  weak ditalgebra, with underlying weak ditalgebra  
${\cal A}=(T,\delta)$. We say that a layer $(R,W)$ of ${\cal A}$ is a \emph{triangular layer for} $({\cal A},I)$ if
\begin{enumerate}
 \item There is a filtration of $R$-$R$-subbimodules
$0=W_0^0\subseteq W_0^1\subseteq \cdots \subseteq W_0^r=W_0$ such that
$\delta(W_0^{i+1})\subseteq A_iW_1A_i$, for all $i\in [0,r-1]$,
where $A_i$ denotes the subalgebra of $A$ generated by $R$ and $W_0^i$.
\item There is a filtration of $R$-$R$-subbimodules
$0=W_1^0\subseteq W_1^1\subseteq \cdots \subseteq W_1^s=W_1$ such that
$\delta(W_1^{i+1})\subseteq AW_1^iAW_1^iA+\langle I\rangle_2$, for all $i\in [0,s-1]$.
\end{enumerate}
 We say that $({\cal A},I)$ is a \emph{triangular interlaced  weak ditalgebra} whenever  $I$ is an ${\cal A}$-triangular ideal of $A$, see \cite[(3.2)]{bpsqh}, and $({\cal A},I)$ admits a triangular layer.
\end{definition}

We stress the fact that the term triangular interlaced weak ditalgebra is used in this  whole paper in the sense of (\ref{D: triangular interlaced weak ditalg}). When the ideal $I$ is zero, we recover the usual definition of triangularity for ditalgebras, see \cite{BSZ}. 

\begin{remark}\label{R: adjust 1-triangularty} With the preceding notation, given $M,N,L\in ({\cal A},I)\g\Mod$, if $f^1 \in \Hom_{A\g A}(V, \Hom_k(M,N))$,  $g^1 \in \Hom_{A\g A}(V, \Hom_k(N, L))$, and 
 $z \in \langle I\rangle_2$, then 
$m((g^1 \otimes f^1)(z)) = 0$, where $m$ denotes composition. Hence, all the proofs 
 concerning splitting of idempotents in $({\cal A},I)\g\Mod$, nilpotency of endomorphisms of the form $(f^0,f^1)$ with $f^0$ nilpotent, Roiter's condition and its preservation under reduction procedures hold unchanged  with this adjusted definition of triangularity. The changes discussed  have no impact on the results obtained in \cite{bpsqh}, which remain well established as they were stated there, with the new definition of triangular interlaced weak  ditalgebra given above. Notice that the reduction functor $F^r$  associated with the regularization used here is the same as the one used in \cite{bpsqh}, the only difference is that the last one is wrongly presented in \cite{bpsqh} as a restriction functor associated to a morphism of interlaced weak ditalgebras.    
 
 Notice also that the notion of triangularity defined in  \cite{hsb} for an  interlaced weak ditalgebra is consistent with all the content of that article, because the  interlaced weak ditalgebra studied there is triangular in the precise sense specified there. So the  correction we present here has no impact at all on \cite{hsb}. Moreover, the results presented in \cite{hsb} for triangular interlaced weak ditalgebras in sections 11 and 12 hold for triangular interlaced weak ditalgebras as defined in  (\ref{D: triangular interlaced weak ditalg}). 
\end{remark}

In the following, we present a simplified approach to the construction of the basic reductions which arise from certain surjective morphisms (including regularization). Since we want a unified approach to the construction of these basic reductions, we need to replace \cite[(6.1)]{bpsqh} by the next proposition. 
Here, we verify the fact that the reduced interlaced weak ditalgebra $({\cal A}',I')$ is triangular,  according to definition (\ref{D: triangular interlaced weak ditalg}), whenever $({\cal A},I)$ is so.

\begin{proposition}\label{P: la nueva} Let $\underline{\cal A}=({\cal A},I)$ be a triangular interlaced weak  ditalgebra with layer $(R,W)$. Assume that we have a surjective morphism of $k$-algebras  $\phi_\flat:R\rightmap{}R'$, $R'$-$R'$-bimodules $W'_0$ and $W'_1$, and retractions of $R\g R$-bimodules $\phi_0:W_0\rightmap{}W'_0$ and $\phi_1:W_1\rightmap{}W'_1$. Consider the tensor algebras $T=T_R(W_0\oplus W_1)$ and $T'=T_{R'}(W'_0\oplus W'_1)$ and the 
 surjective morphism of graded $k$-algebras $\phi:T\rightmap{}T'$ induced by the maps $\phi_\flat$,  $\phi_0$, and $\phi_1$. Assume furthermore that 
 $$\delta(\Ker\phi_0)\subseteq \Ker\phi \hbox{ \  and \  }
 \delta(\Ker \phi_1)\subseteq \Ker\phi+\langle I\rangle_2 
 $$
 where $\delta$ is the derivation of ${\cal A}$. Consider a right inverse $\sigma:W'\rightmap{}W$ for the retraction $\phi_0\oplus \phi_1$ and the derivation $\delta':T'\rightmap{}T'$ with $\delta'(R')=0$  induced by the morphism 
 $\phi\delta\sigma:W'\rightmap{}T'$.  Then, we have a triangular interlaced weak ditalgebra 
 $\underline{\cal A}':=({\cal A}',I')$, where ${\cal A}'=(T',\delta')$ and $I'=\phi(I)$, and a  faithful functor 
 $$F':\underline{\cal A}'\g\Mod\rightmap{}\underline{\cal A}\g\Mod.$$
 Consider the kernels $K_0:=\Ker\phi_{\vert A}$ and $K_1:=\Ker \phi_{\vert V}$. Then, 
the image of $F'$ consists of the $\underline{\cal A}$-modules annihilated by $K_0$, and the functor $F'$ is full whenever $K_1=K_0V+VK_0+\delta(K_0)$. If this is the case, the functor $F'$  preserves endolength. So, 
$\underline{\cal A}'$ is pregenerically tame whenever $\underline{\cal A}$ is so. 
\end{proposition}

\begin{proof} The map $\phi\delta-\delta'\phi$ may not be zero but, since $\delta(\Ker \phi_0)\subseteq \Ker \phi$, it satisfies 
$$(\phi\delta -\delta'\phi)_{\vert W_0}=0$$
and, since $\delta(\Ker \phi_1)\subseteq \Ker\phi+\langle I\rangle_2$, we get 
$$(\phi\delta-\delta'\phi)(W_1)\subseteq \phi(\Ker \phi +\langle I\rangle_2)=\langle \phi(I)\rangle_2.$$
From the first equality, we derive that $\phi\delta(a)=\delta'\phi(a)$, for all $a\in A=T_R(W_0)$. From both of them, we derive that 
$$(\phi\delta-\delta'\phi)(V)\subseteq \langle \phi(I)\rangle_2\hbox{ \ and \ }(\phi\delta-\delta'\phi)(V^2)\subseteq \langle \phi(I)\rangle_3.$$ 
We have the weak ditalgebra ${\cal A}'=(T',\delta')$ with layer $(R',W')$, we claim that  $({\cal A}',I')$ is an interlaced weak ditalgebra where $I'=\phi(I)$. Indeed, we have 
$$(\delta')^2(A')=\delta'\delta'\phi(A)
=
\delta'\phi\delta(A)\subseteq 
\phi\delta^2(A)+\langle\phi(I)\rangle_2
\subseteq \langle\phi(I)\rangle_2$$
and, since $\delta(I)\subseteq IV+VI$,   so $\delta'(\langle \phi(I)\rangle_2)\subseteq \langle \phi(I)\rangle_3$, we obtain  
$$(\delta')^2(V')=\delta'\delta'\phi(V)
\subseteq 
\delta'\phi\delta(V)+\delta'(\langle\phi(I)\rangle_2)\subseteq 
\phi\delta^2(V)+\langle\phi(I)\rangle_3
\subseteq \langle\phi(I)\rangle_3.$$
If we adopt the notation of (\ref{D: triangular interlaced weak ditalg}), we have  the filtration 
$$0=\phi(W_0^0)\subseteq \phi(W_0^1)\subseteq \cdots \subseteq \phi(W_0^r)=\phi(W_0)=W'_0$$  such that
$\delta'(\phi(W_0^{i+1}))=\phi(\delta(W_0^{i+1}))\subseteq A'_i\phi(W_1)A'_i$, for all $i\in [0,r-1]$. We can also consider the filtration 
$$0=\phi(W_1^0)\subseteq \phi(W_1^1)\subseteq \cdots \subseteq \phi(W_1^s)=\phi(W_1)=W'_1.$$ 
 For  $i\in [0,s-1]$,  we have 
$$\delta'\phi(W_1^{i+1})\subseteq \phi\delta(W_1^{i+1})+\langle \phi(I)\rangle_2\subseteq 
A'\phi(W_1^i)A'\phi(W_1^i)A'+\langle \phi(I) \rangle_2.$$
Assume that $0=H_0\subseteq H_1\subseteq \cdots\subseteq H_t=I$ is the triangular filtration of $I$, as in \cite[(3.2)]{bpsqh}. Since $(\delta'\phi-\phi\delta)_{\vert A}=0$, it is clear that the sequence of subspaces
 $0=\phi(H_0)\subseteq \phi(H_1)\subseteq\cdots\subseteq \phi(H_t)=\phi(I)$
 is such that 
 $\delta'(\phi(H_i))\subseteq A'\phi(H_{i-1})V'+V'\phi(H_{i-1})A'$,  for all $i\in [1,t].$ Thus, the ideal 
 $\phi(I)$ of $A'$ is ${\cal A}'$-triangular. 
 
 The functor $F'$ of the statement of this proposition maps any $A'$-module $M$ with $I'M=0$ onto the $A$-module $F'(M)$ with action obtained by restriction through the surjective morphism of algebras $\phi_A:=\phi_{\vert A}:A\rightmap{}A'$, which satisfies $IF'(M)=0$. It maps each $f=(f^0,f^1)\in \Hom_{({\cal A}',I')}(M,N)$ onto $F'(f)=(F'(f)^0,F'(f)^1)$ where $F'(f)^0=f^0$ and $F'(f)^1=f^1\phi_V$,  where $\phi_V:=\phi_{\vert V}:V\rightmap{}V'$. Since $\delta'\phi(a)=\phi\delta(a)$, for all $a\in A$, we have indeed that $F'(f)\in \Hom_{({\cal A},I)}(F'(M),F'(N))$.  
 
 Let us show that $F'$ preserves the  composition of morphisms.    Consider morphisms $f:M\rightmap{}N$ and $g:N\rightmap{}L$ in $({\cal A}',I')\g\Mod$. 
Then, for each $v\in V$, we have 
  $$\begin{matrix}
  [F'(g)F'(f)]^1(v) &=& 
g^0f^1(\phi(v))+g^1(\phi(v))f^0 +m(g^1\phi\otimes f^1\phi)\delta(v)\hfill\\
& =& g^0f^1(\phi(v))+g^1(\phi(v))f^0 +m(g^1 \otimes f^1)\phi(\delta(v))\hfill\\
&=& g^0f^1(\phi(v)) + g^1(\phi(v))f^0 + m(g^1 \otimes f^1)(\delta'\phi(v)) + m(g^1 \otimes f^1)(z)
\end{matrix}$$
where $z \in \langle \phi(I)\rangle_2$. As observed in (\ref{R: adjust 1-triangularty}), we have   $m(g^1 \otimes f^1)(z)=0$, thus  
$$[F'(g)F'(f)]^1(v) = g^0f^1(\phi(v)) + g^1(\phi(v))f^0 + m(g^1 \otimes f^1)(\delta'\phi(v)) = F'(gf)^1(v).$$
The functor $F'$ is faithful because $\phi_V:=\phi_{\vert V}:V\rightmap{}V'$ is surjective. 

Assuming that $K_1=K_0V+VK_0+\delta(K_0)$ we can show that the functor $F'$ is full with the same arguments used to prove \cite[(6.1)(6)]{bpsqh}. It is clear that, in this case, $F'$ preserves endolength. 
\end{proof}

\begin{remark}\label{R: basic properties for Fd, Fq, Fa} The statements (6.2), (6.4) and (6.5) of \cite{bpsqh} hold with the new definition of triangularity stated in (\ref{D: triangular interlaced weak ditalg}). Moreover, the associated functors $F^z:({\cal A}^z,I^z)\g\Mod\rightmap{}({\cal A},I)\g\Mod$ described there, for $z\in \{d,q,a\}$,  preserve endolength. This can be derived from an application of the last proposition
 to the   situations described in the three mentioned statements, where full and faithful functors preserving endolength are induced by a specified triple of surjective morphisms  
$(\phi_\flat,\phi_0,\phi_1)$.

For the regularization construction in \cite[(6.3)]{bpsqh}, the correct statement, which requires the new definition of triangularity, is the next one. It follows  from (\ref{P: la nueva}) and the proof of \cite[(6.3)]{bpsqh}. Notice that the explicit description of the functor $F^r:({\cal A}^r,I^r)\g\Mod\rightmap{}({\cal A},I)\g\Mod$ coincides with the one given in \cite[(6.3)]{bpsqh}. 
\end{remark}

\begin{proposition}[regularization]\label{P: regularization}
 Let $\underline{\cal A}=({\cal A},I)$ be a triangular interlaced weak ditalgebra with underlying weak ditalgebra ${\cal A}=(T,\delta)$ and 
layer $(R,W)$. 
Assume that we have $R$-$R$-bimodule decompositions $W_0=W'_0\oplus W''_0$ and 
$W_1=\delta(W'_0)\oplus W''_1$, set $W'':=W_0''\oplus W''_1$. Consider the
identity map 
$\phi_{\flat}:R\rightmap{}R$, the canonical projections
$\phi_j:W_j\rightmap{}W''_j$, for $j\in 
\{0,1\}$, and the tensor algebra
 $T^r=T_R(W'')$. So, we have a morphism of graded 
algebras $\phi:T\rightmap{}T^r$ 
and the ideal $I^r=\phi(I)$ of $A^r$. Then, there is a triangular 
interlaced weak ditalgebra 
$\underline{\cal A}^r=({\cal A}^r,I^r)$ with layer $(R^r,W^r)$, where $R^r=R$, $W_0^r=W''_0$,  
and $W^r_1=W''_1$.  The morphism of graded algebras $\phi$ induces a full and faithful functor 
$F^r:\underline{\cal 
A}^r\g\Mod\rightmap{}\underline{\cal A}\g\Mod$ which preserves  endolength. Moreover, 
if $\underline{\cal A}$ is a Roiter interlaced weak ditalgebra, as in \cite[(11.1)]{hsb},
then  $M\in \underline{\cal A}\g\Mod$ is isomorphic to an object in the image of 
$F^r$ iff $\Ker\,\delta\cap W'_0$ annihilates $M$. In particular, if 
this intersection is zero, $F^r$ is an equivalence of categories.
\end{proposition}

\begin{remark}\label{R: F  y F0} We will follow the notation and terminology presented in \cite{bpsqh}, for
the reductions using admissible modules, but we need to recall some basic facts. 

Let $\underline{\cal A}=({\cal A},I)$ be a triangular interlaced weak ditalgebra with underlying weak ditalgebra ${\cal A}=(T,\delta)$ and 
layer $(R,W)$. 
Suppose 
that there is an $R$-$R$-bimodule decomposition $W_0=W'_0\oplus W''_0$ with 
$\delta(W'_0)=0$ and  that $X$ is a triangular admissible 
$B$-module, where $B=T_R(W'_0)$ and $\End_B(X)^{op}=S\oplus P$ is the corresponding  splitting.   These are the ingredients required to construct the triangular interlaced weak ditalgebra 
$\underline{\cal A}^X$ and the associated functor $F^X:\underline{\cal A}^X\g\Mod\rightmap{}\underline{\cal A}\g\Mod$, see \cite[(6.6)\&(6.7)]{bpsqh} for detailed definitions, notation, and basic properties. The statements of \cite[(6.6)\&(6.7)]{bpsqh} hold with the new definition of triangularity (\ref{D: triangular interlaced weak ditalg}), so we will refer to them with this in mind. 

 Let us describe a couple of useful commutative diagrams which relate various reduction functors associated to the same $B$-module $X$.

(1): We  have the triangular interlaced weak ditalgebra $\underline{\cal A}_0=({\cal A}_0,I_0)$, where ${\cal A}_0=(B,0)$ and $I_0=B\cap I$. So we get the corresponding interlaced weak ditalgebra $\underline{\cal A}_0^X$, with underlying ditalgebra ${\cal A}^X_0=(T_S(P^*),\delta_0)$, where $\delta_0$ is the differential determined by the comultiplication $\mu:P^*\rightmap{}P^*\otimes_SP^*$. So, we also have  its associated functor $F^X_0:\underline{\cal A}^X_0\g\Mod\rightmap{}\underline{\cal A}_0\g\Mod$. The inclusions of interlaced weak ditalgebras $j:\underline{\cal A}_0\rightmap{}\underline{\cal A}$ and $j^X:\underline{\cal A}_0^X\rightmap{}\underline{\cal A}^X$ determine a commutative square
of functors
$$\begin{matrix}\underline{\cal A}^X\g\Mod&\rightmap{F^X}&\underline{\cal A}\g\Mod\\
\shortlmapdown{F_{j^X}}&&\shortrmapdown{F_j}\\
\underline{\cal A}^X_0\g\Mod&\rightmap{F^X_0}&\underline{\cal A}_0\g\Mod,\\
\end{matrix}$$
where $F_j$ and $F_{j^X}$ denote the restriction functors induced by $j$ and $j^X$, respectively. 
From \cite[(6.10)\&(6.11)]{bpsqh}, we know that,  
 for any complete admissible $B$-module $X$, the functors $F_0^X$ and $F^X$ 
are full and faithful. This applies in the context of \cite[(6.9)]{bpsqh}. 

(2): Since ${\cal A}_0=(B,0)$ is a triangular  ditalgebra (that is a triangular weak interlaced ditalgebra with zero ideal-component) and $X$ is an admissible $B$-module, we also have an associated functor, which we denote here by $\hat{F}_0^X:{\cal A}_0^X\g\Mod\rightmap{}{\cal A}_0\g\Mod$, see \cite[(12.10)]{BSZ}. These functors are extensively studied in \cite{BSZ}. We have the following commutative diagram 
  $$\begin{matrix}
 {\cal A}_0^X\g\Mod &\rightmap{\widehat{F}_0^X}&{\cal A}_0\g\Mod\\
 \shortrmapup{}&&\shortrmapup{}\\
 \underline{\cal A}_0^X\g\Mod &\rightmap{F_0^X}&\underline{\cal A}_0\g\Mod\\
 \end{matrix}$$
 where  the vertical arrows are inclusions.
\end{remark}

 For the precise definition of the notion of wildness for an interlaced weak ditalgebra, present in the next remark, see \cite[(6.12)]{bpsqh}. 

\begin{remark}\label{P: Reducciones vs wildness} From \cite[(6.14)]{bpsqh}, which holds with our new definition of triangularity (\ref{D: triangular interlaced weak ditalg}), we get the following. 
 Assume that the triangular interlaced weak ditalgebra $\underline{\cal A}^z$ is obtained from a Roiter  interlaced weak ditalgebra $\underline{\cal A}$ by some reduction procedure of type $z\in \{d,r,q,a,X\}$. Then, 
 $\underline{\cal A}^z$ is a Roiter interlaced weak ditalgebra, which is not wild whenever $\underline{\cal A}$ is not wild.   
 \end{remark}

\begin{proposition}\label{P: gen tame => non wild} Any pregenerically tame triangular interlaced weak ditalgebra $\underline{\cal A}$ is not wild.
\end{proposition}

\begin{proof}  The proof of \cite[(2.9)]{BPS2} readily adapts to this situation. 
\end{proof}

\section{Bimodules and reduction functors}

We recall some basic notions of \cite{CB2}, adapted to the language used here.  

\begin{definition}\label{D: bimods} Given an interlaced weak ditalgebra $\underline{\cal A}=({\cal A},I)$ and   a $k$-algebra  $E$,
 an $\underline{\cal A}\g E$-{\sl bimodule} is an object $M\in\underline{\cal A}\g \Mod$, together
with a $k$-algebra morphism $\alpha_M:E\rightarrow \End_{\underline{\cal A}}(M)^{op}$. If $N$ is another $\underline{\cal A}$-$E$-bimodule, then
$$\Hom_{\underline{\cal A}\g E}(M,N)=\{f\in\Hom_{\underline{\cal A}}(M,N)\mid
f\alpha_M(e)=\alpha_{N}(e)f,\hbox{ for all }e\in E\}.$$
The category of $\underline{\cal A}$-$E$-bimodules, where the rule of composition is the same as for $\underline{\cal A}\g\Mod$, is denoted by $\underline{\cal A}\g E\g\Mod$. 

We shall denote by $\underline{A}$ the quotient algebra $A/I$, where $A$ is the subalgebra formed  by the degree zero elements in the underlying tensor algebra $T$ of ${\cal A}=(T,\delta)$.
An $\underline{\cal A}$-$E$-{\sl bimodule} $M$ is {\sl proper} iff
$\alpha_M$ factors through the canonical embedding map
$$\End_{\underline{A}}(M)^{op}\rightmap{} \End_{\underline{\cal A}}(M)^{op} \hbox{ such that }
f\mapsto (f,0).$$

The full subcategory of $\underline{\cal A}\g E\g\Mod$ formed by the proper
$\underline{\cal A}\g E$-bimodules will be denoted by $\underline{\cal A}\g
E\g\Mod_p$.  The canonical embedding functor $L_{\underline{\cal A}}:\underline{A}\g\Mod\rightarrow
\underline{\cal A}\g\Mod$, which maps each $\underline{A}$-morphism $f$ to $(f,0)$,
induces a functor $L^E_{\underline{\cal A}}:\underline{A}\g E\g\Mod\rightarrow\underline{\cal A}\g
E\g\Mod$. Clearly, the  proper $\underline{\cal A}$-$E$-bimodules coincide with the
images of the $\underline{A}$-$E$-bimodules under the faithful functor
$L^E_{\underline{\cal A}}$,
we will identify this class of objects with the $\underline{A}$-$E$-bimodules.
\end{definition}

The following  lemma is shown as in the classical case, see \cite[(21.3)]{BSZ}.

\begin{lemma}\label{L: F^E} Assume that $E$ is a $k$-algebra and $\underline{\cal A}$, $ \underline{\cal A}'$
are interlaced weak ditalgebras. Then any functor
$F:{\underline{\cal A}'}\g\Mod\rightmap{}\underline{\cal A}\g \Mod$
 induces a functor
$$F_E:{\underline{\cal A}'}\g E\g\Mod\rightmap{}\underline{\cal A}\g E\g\Mod,$$
which maps any object $M\in \underline{\cal A}'\g E\g \Mod$ onto
$F(M)=F_E(M)$, where the structure of
$\underline{\cal A}$-$E$-bimodule on $F(M)$ is given by the composition morphism
$$\alpha_{F(M)}:E\rightmap{\alpha_{M}} \End_{\underline{\cal A}'}(M)^{op}\rightmap{F}
\End_{\underline{\cal A}}(F(M))^{op}.$$
Moreover, if $F$ is full and faithful, then $F_E$ is so.
In this case, if an $\underline{\cal A}\g E$-bimodule $M$ is
isomorphic as an $\underline{\cal A}$-module  to $F(N)$, for some
$\underline{\cal A'}$-module $N$, then $N$ admits a natural structure of
$\underline{\cal A'}$-$E$-bimodule such that $M\cong F_E(N)$ in $\underline{\cal A}\g E\g\Mod$.
\end{lemma}

\begin{definition}\label{D: length} Let $\underline{\cal A}$ be a  interlaced weak ditalgebra with layer $(R,W)$ and $E$ any $k$-algebra. Given any $\underline{\cal A}\g E$-bimodule $M$, the composition of $\alpha_M: E\rightmap{}\End_{\underline{\cal A}}(M)^{op}$ with the projection $\pi_M: \End_{\underline{\cal A}}(M)^{op}\rightmap{}\End_R(M)^{op}$, which maps $(f^0,f^1)$ onto $f^0$, determines a structure of  $R\g E$-bimodule on $M$. We will denote by $\ell_E(M)$ the length of this right $E$-module.

The full
subcategory of $\underline{\cal A}\g E\g\Mod$ formed by the finite
$E$-length bimodules is denoted by  $\underline{\cal A}\g E\g\mod$
and its intersection with $\underline{\cal A}\g E\g\Mod_p$ by $\underline{\cal
A}\g E\g\mod_p$.

 Suppose that $\underline{\cal A}$ and $\underline{\cal A'}$ are
 layered interlaced weak ditalgebras and $E$ is a $k$-algebra.  Then we
say that a functor $F:\underline{\cal A}\g E\g\Mod\rightmap{}\underline{\cal A}'\g E\g\Mod$ is {\sl length-controlling} iff
$\ell_E(M)$ finite implies $\ell_E(F(M))$ finite  and, furthermore,
$\ell_E(M)\leq\ell_E(F(M))$, for all $M\in \underline{\cal A}\g E\g\Mod$.
If $\ell_E(M)=\ell_E(F(M))$, for all $M$, then $F$ is called
{\sl length-preserving}.
\end{definition}

\begin{lemma}\label{L: funtores restrinccion y longitudes} Let $F':\underline{\cal A}'\g\Mod\rightmap{}\underline{\cal A}\g\Mod$ be the functor constructed in (\ref{P: la nueva}), assume that $F'$ is full and faithful, and have in mind the notation used there. Then, for any $k$-algebra $E$, we have length-preserving full and faithful induced functors
$$F'_E:\underline{\cal A}'\g E\g\Mod\rightmap{}\underline{\cal A}\g E\g \Mod \hbox{ \ \ and \ \ }
F'_E:\underline{\cal A}'\g E\g\Mod_p\rightmap{}\underline{\cal A}\g E\g \Mod_p. 
$$
Moreover, their images consist of the objects annihilated by $K_0$. 
\end{lemma}

\begin{proof} Both functors, $F'_E$ and its restriction,  are full and faithful by (\ref{L: F^E}). 

Let $M\in \underline{\cal A}\g E\g \Mod$ such that $K_0M=0$. Then, $M$ is canonically an $A'$-module which we denote by $N$. So $N\in \underline{\cal A}'\g\Mod$ and the restriction of $N$  through the morphism $\phi_{\vert A}:A\rightmap{}A'$ is such that  $F'(N)= M$ in $\underline{\cal A}\g\Mod$. From (\ref{L: F^E}), we get  that $N$ admits a structure of $\underline{\cal A}'\g E$-bimodule such that $F'_E(N)= M$.  

By definition of $F'$, for any morphism  
$f=(f^0,f^1):N\rightmap{}N'$ in $\underline{\cal A}'\g\Mod$, we have $F'(f)=(f^0,f^1\phi_{\vert V})$. Hence, if $\alpha_N: E\rightmap{}\End_{{\cal A}'}(N)^{op}$ is the structure map of some $N\in \underline{\cal A}'\g E\g \Mod$, the structure map  $\alpha_{F'(N)}$ of its image $F'_E(N)$ in $\underline{\cal A}\g E\g\Mod$ is such that $\alpha_{F'(N)}(e)=F'[\alpha_N(e)]=(\alpha_N(e)^0,\alpha_N(e)^1\phi_{\vert V})$, for $e\in E$. From this we get that $\ell_E(N)=\ell_E(F'_E(N))$. 

Finally, if $ M\in \underline{\cal A}\g E\g \Mod_p$ is such that $K_0M=0$, we have constructed $N\in \underline{\cal A}'\g E\g\Mod$ such that $F'_E(N)= M$. Since $M$ is a proper bimodule, from the preceding formula for $\alpha_{F'(N)}(e)$, we derive that $\alpha_N(e)^1\phi_{\vert V}=0$. Since $\phi$ is a surjective morphism of graded algebras, we obtain that $\alpha_N(e)^1=0$ and $N\in \underline{\cal A}'\g E\g \Mod_p$.   
\end{proof}

\begin{definition}\label{D: regularidades 1} Let $F:\underline{\cal A}'\g\Mod\rightmap{}\underline{\cal A}\g\Mod$ be a functor, where $\underline{\cal A}$ and $\underline{\cal A}'$  are 
  interlaced weak ditalgebras. The functor $F$ is called a \emph{regular   
  length-controlling functor} iff for any $k$-algebra $E$, the functor $F_E:\underline{\cal A}'\g E\g\Mod\rightmap{}\underline{\cal A}\g E\g\Mod$ is length-controlling.

We say that $F$ is a \emph{regular  equivalence} if $F_E:\underline{\cal A}'\g E\g\Mod\rightmap{}\underline{\cal A}\g E\g\Mod$ is an equivalence of categories, for every $k$-algebra $E$, and  restricts to an equivalence $F_E:\underline{\cal A}'\g E\g\Mod_p\rightmap{}\underline{\cal A}\g E\g\Mod_p$, for every division $k$-algebra $E$.   
\end{definition}

\begin{remark}\label{R: reg equiv sii reg equiv for div alg} Notice that once we know that  that a full and faithful functor $F:\underline{\cal A}'\g\Mod\rightmap{}\underline{\cal A}\g\Mod$ induces an equivalence $F_E:\underline{\cal A}'\g E\g\Mod_p\rightmap{}\underline{\cal A}\g E\g\Mod_p$, for every division $k$-algebra $E$, then $F$ is a regular equivalence. 

Indeed, considering the division algebra $k$, we get that $F$ is an equivalence. Then,  we can apply (\ref{L: F^E}), to obtain the density of $F_E:\underline{\cal A}'\g E\g\Mod\rightmap{}\underline{\cal A}\g E\g\Mod$, for every $k$-algebra $E$. 
\end{remark}

\begin{remark}\label{L: FdE, FaE, FqE} The preceding lemma applies to the reduction functors $F^z:\underline{\cal A}^z\g\Mod\rightmap{}\underline{\cal A}\g\Mod$, where $z\in \{d,q,a,r\}$, see  
\cite[(6.2), (6.4), (6.5)]{bpsqh},  (\ref{P: regularization}) and (\ref{R: basic properties for Fd, Fq, Fa}).  It is easy to see that for $z\in \{q,a\}$, we have that $F^z$ is in fact a regular equivalence. The image of the functor  $F_E^d$  consists of  the objects annihilated by the deleted idempotent. The  density properties in the regularization case $z=r$ deserve a more careful attention, this is studied in the next statement. 
\end{remark}

\begin{proposition}\label{P: regularization and bimods}
 Let $\underline{\cal A}$ be a triangular interlaced weak ditalgebra with 
triangular layer $(R,W)$. Assume that $\underline{\cal A}^r$ is obtained from $\underline{\cal A}$ by regularization. Adopt the assumptions of (\ref{P: regularization}), including that $\underline{\cal A}$ is a Roiter interlaced weak ditalgebra and that $\Ker \delta\cap W'_0=0$. Then, 
 the associated functor $F^r: \underline{\cal A}^r\g \Mod\rightmap{}\underline{\cal A}\g \Mod$ is a regular length-preserving equivalence. 
\end{proposition}

\begin{proof} (1): The fact that $F^r_E:\underline{\cal A}^r\g E\g \Mod\rightmap{}\underline{\cal A}\g E\g \Mod$ is a 
length-preserving equivalence for any $k$-algebra $E$, follows from   (\ref{L: funtores restrinccion y longitudes}), (\ref{P: regularization}), and (\ref{L: F^E}). 

\medskip 
\noindent(2): Let us prove that $F_E^r:\underline{\cal A}^r\g E\g\Mod_p\rightmap{}\underline{\cal A}\g E\g\Mod_p$ is an equivalence, for any $k$-algebra $E$. 

Recall that  $W_0=W'_0\bigoplus W''_0$, $W_1=\delta(W'_0)\bigoplus W''_1$ and $\delta:W'_0\rightmap{}\delta(W'_0)$ is an isomorphism. Let us denote by $\psi$ the inverse of the latter. 
Take any proper bimodule $M\in \underline{\cal A}\g E\g\Mod_p$. Using its underlying $\underline{\cal A}$-module structure and the fact that $\underline{\cal A}$ is a Roiter interlaced weak ditalgebra, proceed as in the proof of \cite[(6.3)]{bpsqh} to construct $\underline{M}\in \underline{\cal A}\g\Mod$, with the same underlying $R$-module $M$ and such that $W'_0\underline{M}=0$, and an isomorphism $f=(f^0,f^1):\underline{M}\rightmap{}M$ in $\underline{\cal A}\g \Mod$.  
 Moreover, by construction we have that $f^0=id_M$,  $f^1(W''_1)=0$, and $f^1(w)[m]=\psi(w)m$, for $w\in \delta(W'_0)$.  
The $A$-module structure of $\underline{M}$  is given by 
 $$ a \star m = am - f^1(\delta(a))(m), \hbox{ \ for } a\in A \hbox{ and } m\in M.$$
Clearly,  $\underline{M}$ is  an  $R\g E$-bimodule with the same action of $E$ given on $M$. We will show the following.
\medskip

\noindent\emph{Claim:} $\underline{M}$ is an $A\g E$-bimodule.   
\medskip

First, observe that given $e\in  E$ and  
 $ w \in \delta(W_0')$, we have  $[f^1(w)(m)]e = [\psi(w)m]e =
\psi(w)(me) = f^1(w)(me)$. 
For $w \in W''_1$, we have  $[f^1(w)(m)]e =0= f^1(w)(me)$. 
Thus, for each $ w \in W_1$, we have the equation
\begin{equation}\tag{$*$} [f^1(w)(m)]e = f^1(w)(me).
\end{equation}

Then, for $w \in W_0$, we get 
$$(w \star m)e = (wm)e - [f^1(\delta(w))(m)]e = w(me) - f^1(\delta(w))(me) = w \star (me).$$
Since $R$ and $W_0$ generate $A$, we obtain that $(a\star m)e=a\star (me)$, for $a\in A$, $m\in \underline{M}$, and $e\in E$. So, our claim is proved. 

Now, since $W'_0\underline{M}=0$, we  have that $\underline{M}\in \underline{\cal A}^r\g E\g\Mod_p$ and, using equation $(*)$,  we obtain  that $f:F^r_E(\underline{M})\rightmap{} M$ is an isomorphism  in $\underline{\cal A}\g E\g\Mod_p$. 
\end{proof}

\begin{lemma}\label{L: densidad de F^X} Let $\underline{\cal A}=({\cal A},I)$ be a triangular interlaced weak ditalgebra and $X$ a complete  triangular admissible $B$-module, where $B=T_R(W'_0)$ as in \cite[(6.7)]{bpsqh}. Consider the triangular interlaced weak
ditalgebra $\underline{\cal A}^X$ and the associated reduction functor
$F^X:\underline{\cal A}^X\g\Mod\rightmap{}\underline{\cal A}\g\Mod$.

Given any $k$-algebra $E$, the functor $F^X_E$ maps proper bimodules onto proper bimodules. 
Moreover, if $M\in \underline{\cal A}\g E\g\Mod_p$, $N\in S\g E\g\Mod$, and $X\otimes_SN\cong M$ in $B\g E\g\Mod$, then, there is some $\overline{N}\in \underline{\cal A}^X\g E\g\Mod_p$  with underlying $S\g E$-bimodule
$N$ and such that $F^X_E(\overline{N})\cong M$ in 
$\underline{\cal A}\g E\g\Mod_p$.
 \end{lemma}
 
 \begin{proof} The fact that $F^X_E$ maps proper bimodules onto proper bimodules follows from the definition of $F_E^X$ on morphisms. Similar to the proof of \cite[(25.5)]{BSZ}, now using \cite[(6.7)]{bpsqh}(4).  
 \end{proof}

\begin{definition}\label{D: minimal algebra and source point} Recall that a \emph{minimal algebra} 
$R$ 
is a finite product of algebras $R=\prod_{u\in {\cal P}}R_u$, where each $R_u$ 
is either a rational $k$-algebra or is isomorphic to the field $k$. Thus the unit of $R$ is a sum of primitive orthogonal idempotents $1=\sum_{u\in {\cal P}}e_u$, where each $e_u$ corresponds to the unit element of $R_u$.  
\end{definition}

\begin{proposition}\label{P: X-reduccion y bimods} 
Assume that $\underline{\cal A}=({\cal A},I)$ is a triangular interlaced weak ditalgebra with layer $(R,W)$. 
Under the assumptions of \cite[(6.7)]{bpsqh}, 
consider the algebra $B=T_R(W_0')$. 
 Assume that $\underline{\cal A}^X$ is the triangular interlaced weak ditalgebra obtained from  $\underline{\cal A}$ by reduction, using a complete triangular admissible  $B$-module $X$, thus $\underline{\cal A}^X$ has layer $(S,W^X)$, where $\Gamma=\End_B(X)^{op}=S\oplus P$. Assume that $S$ is a minimal algebra.  
Consider the associated full and faithful functor $F^X:\underline{\cal A}^X\g\Mod\rightmap{}\underline{\cal A}\g\Mod$. Denote by $\mu(X)$ the number of generators in a set of generators of the right $S$-module $X$ with minimal cardinality. 
Then, the following holds: 
\begin{enumerate}
 \item For all  $N\in \underline{\cal A}^X\g\Mod$, we have that
$\Endol{F^X(N)}\leq \mu(X)\times\Endol{N}$. Then   $\underline{\cal A}^X$  is pregenerically tame, 
whenever $\underline{\cal A}$ is so.
\item For any $k$-algebra $E$, the associated  functor $F^X_E:\underline{\cal A}^X\g E\g\Mod\rightmap{}\underline{\cal A}\g
E\g\Mod$ is length-controlling, full and faithful. For any $N\in \underline{\cal A}^X\g E\g \Mod$, the inequality  $\ell_E(F^X_E(N))\leq \mu(X)\times \ell_E(N)$ holds. 
\end{enumerate}
\end{proposition}

\begin{proof} The argument is essentially the same as \cite[(4.7)]{BPS4}, an adaptation of the proof of \cite[(25.7)]{BSZ}, we sketch it here.

We first prove the second item. Take an  
  $\underline{\cal A}^X\g E$-bimodule $N$, with bimodule structure $\alpha_N:E\rightmap{}\End_{\underline{\cal A}^X}(N)^{op}$, so the $\underline{\cal A}\g E$-bimodule $F^X(N)$, has bimodule structure $\alpha_{F^X(N)}:E\rightmap{}\End_{\underline{\cal A}}(F^X(N))^{op}$, as in (\ref{L: F^E}). 

  As in the proof of \cite[(25.7)]{BSZ}, it can be shown that  $\ell_E(X\otimes_S N)$ coincides with the length of the right $E$-module $X\otimes_S N$, where $E$ acts as usual on the right  tensor factor $N$. 
  Since $S$ is a minimal algebra, from this it follows that $\ell_E(N)\leq \ell_E(F^X_E(N))$, see the argument in the proof of 
  \cite[(25.7)]{BSZ}.
 
 Moreover, from an 
epimorphism $S^{\mu(X)}\rightmap{}X$ of right $S$-modules,  we obtain an epimorphism of right $E$-modules
$S^{\mu(X)}\otimes_SN\rightmap{}X\otimes_SN$. 
Therefore, we obtain  
$\ell_E(F^X(N))=\ell_E(X\otimes_S N)\leq 
\ell_E(N^{\mu(X)})=\mu(X)\times \ell_E(N)$, as wanted. 
  
 For the first item, given $N\in \underline{\cal A}^X\g\Mod$, set  $E:=\End_{\underline{\cal A}^X}(N)^{op}$. Then,
 we have the $\underline{\cal A}^X\g E$-bimodule $N$, with bimodule structure $\alpha_N:E\rightmap{}\End_{\underline{\cal A}^X}(N)^{op}$ given by the identity map.  
Since $F^X$ is full and faithful, we have the isomorphism of algebras $E\cong \End_{\underline{\cal A}}(F^X(N))^{op}$ induced by $F^X$.  It 
 provides, by restriction, a structure of right $E$-module on $F^X(N)$. Clearly, $\Endol{N}=\ell_E(N)$ and $\Endol{F^X(N)}=\ell_E(F_E^X(N))$. So, the formula in 1 follows from 2.

Finally, if we assume that $\underline{\cal A}^X$ is not pregenerically tame, we have an infinite family of pairwise non-isomorphic pregeneric $\underline{\cal A}^X$-modules with bounded endolength. Then, applying the full and faithful functor $F^X$ to them, we obtain an infinite family of pairwise non-isomorphic pregeneric $\underline{\cal A}$-modules with bounded endolength. Hence, the interlaced weak ditalgebra $\underline{\cal A}$ is not pregenerically tame.
\end{proof}

\section{Reduction to minimal ditalgebras}

In this section we show that the study of bimodules
 with bounded length over any
non-wild ${\cal P}$-oriented triangular interlaced weak ditalgebra $\underline{\cal A}$, as in  (\ref{T: cal(A) pregen tame sii tame})\&(\ref{D: triangular interlaced weak ditalg}), 
can be reduced to the study of the bimodules over a minimal ditalgebra
obtained from $\underline{\cal A}$ by a finite number of reductions.

\begin{definition}\label{D: regularidades 2} Let $F:\underline{\cal A}'\g\Mod\rightmap{}\underline{\cal A}\g\Mod$ be a functor, where $\underline{\cal A}$ and $\underline{\cal A}'$  are 
  interlaced weak ditalgebras.  
    Given $d\in \hueca{N}$, the functor $F$ is called \emph{regular $d$-dense}  if is satisfies the following: for any $k$-algebra $E$, if $M\in \underline{\cal A}\g E\g\Mod$ with $\ell_E(M)\leq d$, there is $N\in \underline{\cal A}'\g E\g\Mod$ such that $F_E(N)\cong M$. Moreover, if $E$ is a division $k$-algebra and $M$ is a proper $\underline{\cal A}\g E$-bimodule, then $N$ can be chosen to be a proper $\underline{\cal A}'\g E$-bimodule. 
\end{definition}

\begin{lemma}\label{L: F0X ddenso -> FX ddenso} In the context of (\ref{R: F  y F0}), for  a complete admissible $B$-module $X$, we have the following: 
\begin{enumerate}
\item If the functor $\widehat{F}_0^X$ is regular $d$-dense (resp. regular equivalence), so is the functor $F_0^X$. 
\item If the functor $F_0^X$ is regular $d$-dense (resp. regular equivalence), so is the functor $F^X$.
\end{enumerate}
\end{lemma}

\begin{proof} From (\ref{R: F  y F0}), for any $k$-algebra $E$,   
we have the commutative diagram 
$$\begin{matrix}{\cal A}_0^X\g E\g\Mod&\rightmap{(\widehat{F}_0^X)_E}&{\cal A}_0\g E\g\Mod\\
\shortlmapup{}&&\shortrmapup{}\\
\underline{\cal A}^X_0\g E\g\Mod&\rightmap{(F^X_0)_E}&\underline{\cal A}_0\g E\g\Mod.\\
 \shortlmapup{F^E_{j^X}}&&\shortrmapup{F^E_j}\\
 \underline{\cal A}^X\g E\g\Mod &\rightmap{F^X_E}&\underline{\cal A}\g E\g\Mod.\\
 \end{matrix}$$

\noindent(1): Assume that $\widehat{F}_0^X$ is regular $d$-dense and 
take $M_0\in \underline{\cal A}_0\g E\g\Mod$ with $\ell_E(M_0)\leq d$. By assumption, there is some $N_0\in {\cal A}_0^X\g E\g\Mod$ with $(\widehat{F}_0^X)_E(N_0)\cong M_0$. So $\widehat{F}^X_E(N_0)$ is an ${\cal A}_0$-module with underlying $B$-module structure $X\otimes_SN_0$ and the structure of the ${\cal A}_0^X$-module  $N_0$ is uniquely determined  and satisfies $I_0^XN_0=0$, see the proof of \cite[(6.7)(4)]{bpsqh}. So, $N_0\in \underline{\cal A}_0^X\g E\g\Mod$ and $(F_0^X)_E(N_0)\cong M_0$.  
 
 The statement on proper bimodules in the definition of regular $d$-dense can be proved as in \cite[(25.5)]{BSZ}.
\medskip

\noindent(2): Assume that $F_0^X$ is regular $d$-dense and 
take $M\in \underline{\cal A}\g E\g\Mod$ with $\ell_E(M)\leq d$. 
Since $\ell_E(F_j^E(M))=\ell_E(M)\leq d$, by assumption, there is some $N_0\in \underline{\cal A}^X_0\g E\g \Mod$ such that $(F_0^X)_E(N_0)\cong F_j^E(M)$. Hence, $F_j^E(M)\cong X\otimes_S N_0$ as $B$-modules. This implies, by \cite[(6.7)]{bpsqh}, the existence of some $N\in \underline{\cal A}^X\g\Mod$ such that $F^X(N)\cong M$ in $\underline{\cal A}\g\Mod$ and $F_{j^X}(N)=N_0$. Then, from (\ref{L: F^E}), $N$ admits a structure of $\underline{\cal A}^X\g E$-bimodule such that $F^X_E(N)\cong M$ in $\underline{\cal A}\g E\g\Mod$, as we wanted to show. 

If $E$ is a division algebra and $M$ is a proper bimodule, by assumption, $N_0$ can be chosen to a proper $\underline{\cal A}_0\g E$-bimodule.  Then, proceeding as in \cite[(25.5)]{BSZ}, we can show that $N$ is a proper $\underline{\cal A}^X\g E$-bimodule too.

The density in the statement of our lemma, refering to regular equivalences, is proved similarly. 
\end{proof}

 \begin{lemma}[edge-reduction]\label{L: edge-reductions}
 Let $\underline{\cal A}=({\cal A},I)$ be a triangular interlaced weak ditalgebra with
 triangular layer $(R,W)$ and derivation $\delta$, where $R$ is a minimal algebra.
  We have $R=\prod_{u\in {\cal P}}R_ue_u$, where each $Re_u$ is either
isomorphic to $k$ or to some rational algebra.
We can assume that ${\cal P}=J\uplus J'$ where $Re_u=ke_u$, for $u\in J$, and
$Re_v=R_ve_v$ with $R_v=k[x]_{g_v}$, for $v\in J'$. Moreover, assume that  $W_0=W'_0\oplus W''_0$, and that the $R$-$R$-bimodule $W'_0$ is freely generated by some element $\alpha\in e_{v_0}W_0e_{u_0}$, where $u_0$ and $v_0$ are two  different elements in $J$ and $\delta(\alpha)=0$. Set $e=e_{u_0}+e_{v_0}$.

   Then,  there is complete triangular admissible $B$-module $X$ such that   $\underline{\cal A}^X:=({\cal A}^X,I^X)$ is a triangular interlaced weak ditalgebra
  with
   triangular layer $(S,W^X)$, where $S$ is a minimal algebra of the form  $$S=kf_{v_*}\times kf_z\times kf_{u_*}\times (1-e)R,$$
where $f_{v_*}, f_{z}, f_{u_*}$ are new  primitive idempotents of $S$.
   Moreover, the associated functor 
$F^X:\underline{\cal A}^X\g\Mod\rightmap{}\underline{\cal A}\g\Mod$ is a regular 
length-controlling equivalence.
\end{lemma}

\begin{proof}  Here, we consider the subalgebra $B:=eB\times (1-e)R$, where $eB$ is identified with the path algebra  of the quiver $u_0\rightmap{\alpha}v_0$. Thus, $eB$ admits
only three classes of indecomposables represented by the simple injective $Z_{u_0}$, the simple projective $Z_{v_0}$, and the injective-projective $Z_z$, which are naturally considered as $B$-modules. We consider the $B$-module
$X=Z_{u_0}\oplus Z_z\oplus Z_{v_0}\oplus (1-e)R$.
Then, we have the splitting
 $\End_B(X)^{op}=S\oplus P$, where
 $S$ has the form specified in the statement of this lemma, where
  $f_{v_*}, f_z, f_{u_*}\in \End_B(X)^{op}$ are the idempotents corresponding to the indecomposable direct summands $Z_{u_0}$, $Z_z$, and $Z_{v_0}$  of $X$,
  and $P=\rad \End_B(Z_{u_0}\oplus Z_z\oplus Z_{v_0})^{op}$.  The $B$-module $X$ is complete by \cite[(6.9)]{bpsqh}.

  From (\ref{R: F y F0})(1) and (\ref{P: X-reduccion y bimods}), we  know  that  $F^X$ is a regular length-controlling full and faithful functor.  We know from (\ref{L: F0X ddenso -> FX ddenso}), that, in order to show that $F^X$ is a regular equivalence, it is enough to show that the functor  $\widehat{F}^X_0:{\cal A}_0^X\g\Mod\rightmap{}{\cal A}_0\g\Mod$ is a regular equivalence. But, since ${\cal A}_0$ is a seminested ditalgebra, this is a consequence of \cite[(25.8)]{BSZ}. 
  See also \cite[(2.10)]{CB2}. 
\end{proof}

 \begin{lemma}[multiple $d$-unravelling]\label{L: d-unravellings}
 Let $\underline{\cal A}=({\cal A},I)$ be a triangular interlaced weak ditalgebra with
 triangular layer $(R,W)$, where $R$ is a minimal algebra.
  We have $R=\prod_{u\in {\cal P}}R_ue_u$, where each $Re_u$ is either
isomorphic to $k$ or to some rational algebra.
We can assume that ${\cal P}=J\uplus J'$ where $Re_u=ke_u$, for $u\in J$, and
$Re_v=R_ve_v$ with $R_v=k[x]_{g_v}$, for $v\in J'$.

 Consider $d\in \hueca{N}$ and non-zero elements $h_v\in R_v$, for $v\in J'$. Then, there is a complete triangular admissible $R$-module $X$ such that   $\underline{\cal A}^X:=({\cal A}^X,I^X)$ is a triangular interlaced weak ditalgebra
  with
   triangular layer $(S,W^X)$, where $S$ is a minimal algebra of the form  $$S=\left[\prod_{w\in J''}kf_w\right]\times \left[\prod_{v\in
J'}e_v(R_v)_{h_v}\right]\times \prod_{u\in J}ke_u,$$
where $\{f_w\}_{w\in J''}$ is a new finite family of primitive idempotents of $S$.
     The associated functor
$F^X:\underline{\cal A}^X\g\Mod\rightmap{}\underline{\cal A}\g\Mod$ is a regular full and faithful length-controlling $d$-dense functor.  
\end{lemma}

\begin{proof} Fix $d\in \hueca{N}$ and   set $e=\sum_{u\in J}e_u$. Consider the algebras 
$C'=\prod_{v\in J'}(Re_v)_{h_v}$ and $C''=\prod_{v\in J'}(Re_v)/\langle h_v^d\rangle$. The finite-dimensional $k$-algebra $C''$ 
admits only a finite number
 of isoclasses of indecomposable finite-dimensional modules
 represented by the $C''$-modules $\{Z_w\}_{w\in J''}$.
 Define $Z:=\bigoplus_{w\in J''}Z_w$ and consider the $R$-module
 $$X=Z\oplus C'\oplus Re.$$
 We have the splitting
 $\End_R(X)^{op}=S\oplus P$, with 
 $S=\left[\prod_{w\in J''}kf_w\right]\times C'\times Re$,  where 
  $f_w\in \End_R(X)^{op}$ is the idempotent corresponding to the indecomposable direct summand $Z_w$ of $X$,
  and $P=\rad \End_R(Z)^{op}$.
Apply \cite[(6.7)]{bpsqh} to $B=R$ and the triangular admissible $R$-module $X$ to obtain the triangular interlaced weak ditalgebra $\underline{\cal A}^X$, with layer $(S,W^X)$, and the functor $F^X:\underline{\cal A}^X\g\Mod\rightmap{}\underline{\cal A}\g\Mod$. From \cite[(6.9)]{bpsqh} and (\ref{R: F  y F0})(1), $X$ is complete and $F^X$ is full and faithful.

From  (\ref{P: X-reduccion y bimods}), the  functor  $F^X:\underline{\cal A}^X\g \Mod\rightmap{}\underline{\cal A}\g\Mod$ is a  regular length-controlling  full and faithful functor. From (\ref{L: F0X ddenso -> FX ddenso}), we know that in order to prove that $F^X$ is a regular $d$-dense functor, it will be enough to show that $\widehat{F}_0^X:{\cal A}_0^X\g \Mod\rightmap{}{\cal A}_0\g\Mod$ is a  regular $d$-dense functor. Here, ${\cal A}_0=(R,0)$. 

For this, we consider the following:  For each $v\in J'$, let $Z_v$ be the direct sum of a complete family $\{Z_w\}_{w\in J''_v}$ of representatives of the isomorphism classes of the indecomposable modules of the algebra 
$Re_v/\langle h_v^d\rangle$. Then,  
 consider the  $Re_v$-module  
$X_v := Z_v\oplus C'_v$, where  $C'_v:=(Re_v)_{h_v}$. 

Then   $C'=\prod_{v\in J'} C'_v$ and we may assume that $J''=\bigcup_{v\in J'}J''_v$; thus we have the  equalities of $R$-modules 
$Z=\bigoplus_{v\in J'}Z_v$ and  $X=[\bigoplus_{v\in J'}X_v]\oplus Re$. 

For each $v\in J'$, we have the splitting $\End_{Re_v}(X_v)^{op}=S_v\oplus P_v$, where $P_v=\rad\End_{Re_v}(Z_v)^{op}$, with  $S_v=\prod_{w\in J''_v}kf_w\times C'_v$ and $f_w$ is the idempotent in $\End_{Re_v}(X_v)^{op}$ corresponding to the direct summand $Z_w$ of $X_v$. We use the same notation for these new idempotents $f_w$ because, since $\End_R(X)^{op}=Re\times\prod_{v\in J'}\End_{Re_v}(X_v)^{op}$, we can identify them canonically with the corresponding $f_w$ used before.   
From \cite[(6.9)]{bpsqh}, each $Re_v$-module  $X_v$ is complete. For each $v\in J'$, we have the associated functor $F^{X_v}:(Re_v,0)^{X_v}\g\Mod\rightmap{}(Re_v,0)\g\Mod$. 

The isomorphism of algebras $T_S(P^*)\cong Re\times \prod_{v\in J'}T_{S_v}(P_v^*)$ determines an isomorphism of ditalgebras $(R,0)^X\cong (Re,0)\times\prod_{v\in J'}(Re_v,0)^{X_v}$. This last one induces an  isomorphism of categories in the upper horizontal arrow of the following commutative diagram 
$$\begin{matrix}
(R,0)^X\g \Mod&\rightmap{\cong}& (Re,0)\g \Mod \times \prod_{v\in J'}(Re_v,0)^{X_v}\g\Mod\\
\shortlmapdown{\widehat{F}_0^X}&&\shortrmapdown{id \times \prod_{v\in J'} F^{X_v}}\\
(R, 0)\g\Mod &\rightmap{\cong}& (Re, 0)\g\Mod \times \prod_{v\in J'}(Re_v , 0)\g\Mod,\\
\end{matrix}$$
see \cite[\S6]{BSZ}. 
Since each  $(Re_v,0)$ is a seminested 
 ditalgebra, from \cite[(25.9)]{BSZ}, we know that each functor  $F^{X_v} : (Re_v ,0)^{X_v}\g\Mod \rightmap{}
(Re_v , 0)\g\Mod$ is regular $d$-dense. See also \cite[(2.11)]{CB2}. It follows that the functor  $\widehat{F}_0^X$ is regular $d$-dense.  
\end{proof}

\begin{lemma}[ideal-reduction]\label{L: ideal-reduction}
 Let $\underline{\cal A}=({\cal A},I)$ be a triangular interlaced weak ditalgebra with
 triangular layer $(R,W)$, where $R$ is a minimal algebra.
  We have $R=\prod_{u\in {\cal P}}R_ue_u$, where each $Re_u$ is either
isomorphic to $k$ or to some rational algebra.
Assume that ${\cal P}=J\uplus J'$ where $Re_u=ke_u$, for $u\in J$, and
$Re_v=R_ve_v$ with $R_v=k[x]_{g_v}$, for $v\in J'$. Assume furthermore that $I$ contains no primitive idempotent $e_u$ of $R$, but that $I_0:=I\cap R\not=0$.  Set
${\cal V}=\{v\in {\cal P}\mid I_0e_v\not=0\}$.

 Then,   there is a complete triangular admissible $R$-module $X$ such that   $\underline{\cal A}^X:=({\cal A}^X,I^X)$ is a triangular interlaced weak ditalgebra
  with
   triangular layer $(S,W^X)$, where $S$ is a minimal algebra of the form  $$S=\left[\prod_{w\in J''}kf_w\right]\times \prod_{u\in {\cal P}\setminus {\cal V}}Re_u,$$
where $\{f_w\}_{w\in J''}$ is a new finite family of primitive idempotents of $S$. 
   Moreover, the associated functor 
$F^X:\underline{\cal A}^X\g\Mod\rightmap{}\underline{\cal A}\g\Mod$ is regular full and faithful length-controlling $d$-dense, for every $d\in \hueca{N}$.
\end{lemma}

\begin{proof} This proof is similar to the proof of
\cite[(7.11)(case 2)]{bpsqh}. 
 Set  $e=\sum_{v\in {\cal V}}e_v$. Then, we have  $I_0\subseteq Re$ and the finite-dimensional quotient $k$-algebra $Re/I_0\cong \prod_{v\in {\cal V}}Re_v/I_0e_v$ has finite representation type. Consider a complete family of representatives $\{Z_w\}_{w\in J''}$ of the isoclasses of indecomposable $Re/I_0$-modules and set $Z:=\bigoplus_{w\in J''}Z_w$. Then, we have the $R$-module $X:=Z\oplus R(1-e)$ and the splitting $\End_R(X)^{op}=S\oplus P$, with  $P=\rad\End_R(Z)^{op}$ and $S=\left[\prod_{w\in J''}kf_w\right]\times R(1-e)$, where $f_w\in \End_R(X)^{op}$ is the idempotent associated to the direct summand $Z_w$. 

As in the proof of the last lemma, apply \cite[(6.7)]{bpsqh} to the triangular admissible $R$-module $X$ to obtain the triangular interlaced weak ditalgebra $\underline{\cal A}^X$ and the functor $F^X:\underline{\cal A}^X\g\Mod\rightmap{}\underline{\cal A}\g\Mod$. 
From \cite[(6.9)]{bpsqh} and 
(\ref{R: F  y F0})(1), $X$ is complete and $F^X$ is full and faithful. 
From  (\ref{P: X-reduccion y bimods}), we already know  that $F^X:\underline{\cal A}^X\g\Mod\rightmap{}\underline{\cal A}\g\Mod$ is a regular 
length-controlling, full and faithful functor.
 
\medskip
\noindent(1): We will first show that $(F^X)_E:\underline{\cal A}^X\g E\g \Mod_p\rightmap{}\underline{\cal A}\g E\g \Mod_p$ is a $d$-dense functor for every division $k$-algebra $E$.
From (\ref{L: F0X ddenso -> FX ddenso})(2), we know that it will be enough to show that   
 $(F_0^X)_E:\underline{\cal A}_0^X\g E\g \Mod_p\rightmap{}\underline{\cal A}_0\g E\g \Mod_p$ is a $d$-dense functor for every division $k$-algebra $E$. Here, we have $\underline{\cal A}_0=((R,0),I_0)$ and $\underline{\cal A}_0^X$ has layer $(S,P^*)$.  
 
 Take any division $k$-algebra $E$ and $M\in \underline{\cal A}_0\g E\g \Mod_p$ with $\ell_E(M)=d$. So, $M$ is an $R\g E$-bimodule annihilated by $I_0$. Hence, we have a decomposition of $R\g E$-bimodules $M=eM\oplus (1-e)M$, where $eM$ is an $Re/I_0\g E$-bimodule. Following Crawley-Boevey's argument in \cite[(2.11)]{CB2}, see also \cite[(25.9)]{BSZ}, we obtain that  $eM\cong \oplus_{w\in J''}Z_w\otimes_kV_w$, for some finite-dimensional right $E$-vector spaces $V_w$.

Notice that, for $w\in J''$, we have the simple $S$-module $Sf_w$ and an isomorphism of $R$-modules $F^X(S_{f_w})=X\otimes_SSf_w\cong Xf_w\cong Z_w$. Then,  if we consider the
$S$-module $N:=\left[\bigoplus_{w\in J''}Sf_w\otimes_k V_w\right]\bigoplus (1-e)M$,   we have an isomorphism of $R$-modules $F_0^X(N)\cong eM\bigoplus (1-e)M=M$.  Using \cite[(6.7)]{bpsqh}, we know that $F_0^X(\overline{N})\cong M$, for some  $\overline{N}\in \underline{\cal A}_0\g\Mod$. 
Finally, we obtain what we wanted using (\ref{L: densidad de F^X}).
\medskip

\noindent(2): Now  we show that in fact $F^X_E:\underline{\cal A}^X\g E\g\Mod\rightmap{}\underline{\cal A}\g E\g\Mod$ is an equivalence, for any $k$-algebra $E$. For this, by the argument given in the proof of (\ref{L: F0X ddenso -> FX ddenso})(2), it will be enough to show that $(F_0^X)_E:\underline{\cal A}_0^X\g E\g\Mod\rightmap{}\underline{\cal A}_0\g E\g\Mod$ is an equivalence. 
Take any $M\in \underline{\cal A}_0\g E\g\Mod$.

 From (\ref{L: F^E}), it will be enough to show that there is some $N\in \underline{\cal A}_0^X\g\Mod$ such that $F_0^X(N)\cong M$ in $\underline{\cal A}_0\g \Mod$. 
  We have $M=eM\oplus (1-e)M$, where $eM$ is an $Re/I_0$-module. Hence, 
from \cite{Ausl}, 
 $eM\cong \oplus_{w\in J''}Z_w\otimes_k V_w$ for some (possibly infinite-dimensional) $k$-vector spaces $V_w$. Consider the $S$-module $N:=[\bigoplus_{w\in J''}Sf_w\otimes_k V_w]\oplus (1-e)M$ and notice that $F_0^X(N)\cong M$. 
\end{proof}

\begin{remark} In the context of the last lemma, if we futhermore assume that 
$I\subseteq I_0\oplus  \langle W_0\rangle$,
where $\langle W_0\rangle$ is the ideal of $A$ generated by $W_0$, then we obtain that   $I^X\cap S=0$.

Indeed,  since $I_0X=0$, we have $\sigma_{\nu,x}(r)=\nu(r x)=0$, for $r\in I_0$, $\nu\in X^*$, and $x\in X$. 
By assumption, any $h\in I$ has the form $h=r+h'$, where $r\in I_0$ and $h'\in \langle W_0\rangle$. Hence $\sigma_{\nu,x}(h)=\sigma_{\nu,x}(h')\in \langle W_0^X\rangle$. Thus, $I^X\cap S=0$. 

The functor $F^X$ constructed in the preceding lemma appears in the sequence of functors in the next Theorem (\ref{T: parametrización de módulos de dim acotada}), where the construction is applied to a special type of interlaced weak ditalgebra  satisfying the additional requirement of the preceding remark. 
\end{remark}

\begin{remark}\label{R: introd to minimal ditalgebras} The minimal ditalgebras 
 appearing in the next theorem are, by definition, ditalgebras ${\cal B}$ with triangular layer $(R,W)$, where $R$ is a minimal algebra, $W_0=0$, and $W_1$ is freely generated by a finite directed subset, see \cite[(23.2)]{BSZ}. Their module category is tame and well understood. A crucial fact in the proof of Drozd's theorem on the tame-wild dichotomy for finite-dimensional algebras in \cite{D} is that, for any seminested non-wild ditalgebra ${\cal A}$ and any dimension $d$, there is a composition of reduction functors $F:{\cal B}\g\Mod\rightmap{}{\cal A}\g\Mod$, where ${\cal B}$ is a minimal ditalgebra such that any $M\in {\cal A}\g\Mod$ with $\dim M\leq d$ is of the form $M\cong F(N)$, for some $N\in {\cal B}\g\Mod$. This is proved in \cite[\S8]{BDZZ} using bocses, and in \cite[\S28]{BSZ} using the equivalent language of ditalgebras. Using  bocses, almost the same statement was proved insightfully in \cite{CB1}, among other important results, but using a finite number of minimal ditalgebras. This statement was later generalized to the context of bimodules, see \cite[(2.13)]{CB2} and \cite[(28.22)]{BSZ}. The mentioned facts play a crucial role in the construction of the family of functors in the next theorem. 
\end{remark}

\begin{remark}\label{R: Fz and Fzominus} In the following, for the proof of (\ref{T: parametrización de módulos de dim acotada}), given  triangular interlaced weak ditalgebra $\underline{\cal A}=({\cal A},I)$ with a multiple source $\Omega$, as in \cite[(8.1)]{bpsqh}  we will consider the associated multiple source idempotent  $e_\Omega:=\sum_{w\in \Omega}e_w$, the triangular interlaced weak ditalgebra $\underline{\cal A}^{\ominus}$ given by the supression of $e_\Omega$, and the corresponding restriction functor $\Res:\underline{\cal A}\g\Mod\rightmap{}\underline{\cal A}^\ominus\g\Mod$, as in 
\cite[(8.2)]{bpsqh}. 

In this context, we studied in \cite[(8.5)--(8.9)]{bpsqh}     
 some special kind of reductions $\underline{\cal A}\longmapsto \underline{\cal A}^z$ of type $z\in \{d,q,a,r,X\}$.  
For each one of these cases, we have seen  that whenever we can perform a reduction on $\underline{\cal A}^{\ominus}$ of type $z$, then we can perform a corresponding reduction of the same type $z$  on $\underline{\cal A}$. Moreover,  in each one of these cases, we can identify  $\underline{\cal A}^{z\ominus}$ with $\underline{\cal A}^{\ominus z}$ and we get  a  commutative square  of functors 
$$\begin{matrix}
\underline{\cal A}^z\g\Mod&\rightmap{F^z}&\underline{\cal A}\g\Mod\\
\shortlmapdown{\Res}&&\shortrmapdown{\Res}\\
\underline{\cal A}^{z\ominus}\g\Mod&\rightmap{F^{\ominus z}}&\underline{\cal A}^{\ominus}\g\Mod,\\
\end{matrix}$$
where $F^z$ and $F^{\ominus z}$ denote the corresponding associated functors. Moreover, whenever $M\in \underline{\cal A}\g\Mod$ is such that $\Res(M)\cong F^{\ominus z}(N')$, for some $N'\in \underline{\cal A}^{z\ominus}\g\Mod$, we have that  $M\cong F^z(N)$, for some $N\in \underline{\cal A}^z\g\Mod$. 
\end{remark}

\begin{lemma}\label{L: FzE y Fzominus E} In the context of the last remark, for any type  $z\in\{d,q,a,r,X\}$ and 
any $k$-algebra $E$, we have an induced commutative diagram 
$$\begin{matrix}
\underline{\cal A}^z\g E\g\Mod&\rightmap{(F^z)_E}&\underline{\cal A}\g E\g\Mod\\
\shortlmapdown{\Res_E}&&\shortrmapdown{\Res_E}\\
\underline{\cal A}^{z\ominus}\g E\g\Mod&\rightmap{(F^{\ominus z})_E}&\underline{\cal A}^{\ominus}\g E\g\Mod.\\
\end{matrix}$$
Moreover, 
\begin{enumerate}
\item If $M\in \underline{\cal A}\g E\g\Mod$ is such that $\Res_E(M)\cong (F^{\ominus z})_E(N')$, for some $N'\in \underline{\cal A}^{z\ominus}\g E\g\Mod$, we have that  $M\cong (F^z)_E(N)$, for some $N\in \underline{\cal A}^z\g E\g\Mod$. 

\item If $E$ is a division $k$-algebra and $M\in \underline{\cal A}\g E\g\Mod_p$
is such that $\Res_E(M)\cong (F^{\ominus z})_E(N')$, for some $N'\in \underline{\cal A}^{z\ominus}\g E\g\Mod_p$, we have that  $M\cong (F^z)_E(N)$, for some $N\in \underline{\cal A}^z\g E\g\Mod_p$. 
\end{enumerate}
\end{lemma}

\begin{proof} In case $z=X$, according to \cite[(8.9)]{bpsqh}, we have that $F^{\ominus z}$ is the functor associated to the complete triangular admissible $B^{\ominus}$-module $X^{\ominus}$ and   $F^z$ is the functor associated to the complete triangular admissible $B$-module $X=Re_\Omega\oplus X^{\ominus}$. We have the corresponding splittings $\End_{B^{\ominus}}(X^\ominus)^{op}=S^\ominus\oplus P^\ominus$ and $\End_{B}(X)^{op}=S \oplus P$, where $S= Re_\Omega\times S^\ominus$. 
\medskip

\noindent (1): From (\ref{L: F^E}), we know that in order to find $N\in \underline{\cal A}\g E\g\Mod$ such that $(F^z)_E(N)\cong M$ in $\underline{\cal A}\g E\g\Mod$, it is enough to find $N\in \underline{\cal A}\g\Mod$ with $F^z(N)\cong M$ in $\underline{\cal A}\g\Mod$: this we have by (\ref{R: Fz and Fzominus}). 

\medskip
\noindent(2): Assume that $M\in \underline{\cal A}\g E\g\Mod_p$
is such that $\Res_E(M)\cong (F^{\ominus z})_E(N')$, for some $N'\in \underline{\cal A}^{z\ominus}\g E\g\Mod_p$. By (\ref{L: densidad de F^X}), we need to exhibit an $S\g E$-bimodule $N$ such that $X\otimes_SN\cong M$ in $B\g E\g \Mod$. Here, we can take $N=e_\Omega M\oplus N'$, which satisfies that 
$X\otimes_SN\cong [Re_\Omega\otimes_{Re_\Omega}
e_\Omega M]\oplus [X^\ominus\otimes_{S^\ominus} N']=e_\Omega M \oplus (1-e_\Omega)M=M$, as $B\g E$-bimodules.      
 
 In case $z\in \{a,r,q\}$, we already know from (\ref{L: FdE, FaE, FqE}) and (\ref{P: regularization and bimods}), that $F^z$ is a regular equivalence, so our claim is obvious.
Finally, in case $z=d$, from (\ref{L: funtores restrinccion y longitudes}) the induced functors
$$(F^d)_E : \underline{\cal A}^d\g E\g\Mod\rightmap{}\underline{\cal A}\g E\g\Mod \hbox{ \  and \ } (F^d)_E: \underline{\cal A}^d\g E\g \Mod_p\rightmap{}\underline{\cal A}\g E\Mod_p,$$
have image consisting of the objects  annihilated by the deleted indempotent. Clearly, for any  $\underline{\cal A}\g E$-bimodule $M$, it is annihilated by the deleted idempotent if  $\Res(M)$ is annihilated by the deleted idempotent. So, our claim in this case  follows immediately.
\end{proof}

\begin{definition} An  \emph{elementary reduction functor} $F^z:\underline{\cal A}^z\g\Mod\rightmap{}\underline{\cal A}\g\Mod$ of type $z\in \{d,q,a,r,X\}$ is one of the functors considered in (\ref{L: FdE, FaE, FqE}), (\ref{P: regularization and bimods}), or (\ref{P: X-reduccion y bimods}). In the following, we only use functors of type $X$ of the following special forms:  functors associated to edge-reductions as in (\ref{L: edge-reductions}), to  multiple $d$-unravellings as in (\ref{L: d-unravellings}),  and to   ideal-reductions as in (\ref{L: ideal-reduction}). By definition, a \emph{reduction functor} is any finite composition of elementary reduction functors. Notice that any composition of regular length-controlling $d$-dense reduction functors is a regular $d$-dense reduction functor. 

Whenever $\underline{\cal A}=({\cal A},I)$ is  an interlaced weak ditalgebra with $I=0$, the underlying weak ditalgebra ${\cal A}$ is in fact a ditalgebra and $\underline{\cal A}\g\Mod={\cal A}\g\Mod$, so we will often omit the underlining and refer to it indistinctly as $\underline{\cal A}$ or ${\cal A}$.    
\end{definition}

 The proof of the following theorem relies on the properties of elementary reduction functors studied before, which include the type of functors appearing in the reduction of non-wild seminested ditalgebras to minimal ditalgebras, see \cite[(28.22)]{BSZ} and \cite[(9.2)]{bpsqh}.  Let us note that the statement of  \cite[(9.2)]{bpsqh} must be corrected by removing the word ``almost''.

\begin{theorem}\label{T: parametrización de módulos de dim acotada}
  Let ${\cal P}$ be a finite preordered set.
 Assume $\underline{\cal A}=({\cal A},I)$ is a ${\cal P}$-oriented triangular interlaced weak ditalgebra as in
(\ref{D: biquiver P-orientado})\&(\ref{D: triangular interlaced weak ditalg}), where $I$ is an
 ideal of $A$ contained in the radical of $A$.
Suppose that $\underline{\cal A}$ is not wild and take $d\in \hueca{N}$. 
Then, there is a regular $d$-dense reduction functor $F:{\cal B}\g\Mod\rightmap{}\underline{\cal A}\g\Mod$, where ${\cal B}$ is minimal ditalgebra (equipped with the trivial ideal).  
\end{theorem}

\begin{proof} This proof is an adaptation of the proof of \cite[(9.2)]{bpsqh}  and is based on the construction given there. The proof proceeds by induction on the cardinality of 
${\cal P}$.  If  $\vert {\cal P}\vert=1$, then ${\cal A}$ is a minimal ditalgebra (with $I=0$), and we are done.  Consider the multiple source 
$\Omega\subseteq {\cal P}$ constructed in the mentioned proof, and the ${\cal P}^\ominus$-oriented non-wild triangular interlaced weak ditalgebra $\underline{\cal A}^{\ominus}=(\underline{\cal A}^{\ominus},I^{\ominus})$ obtained by supression of the multiple source idempotent $e_\Omega$, which satisfies that $I^\ominus\subseteq \rad A^{\ominus}$. Here, ${\cal P}^\ominus={\cal P}\setminus \Omega$, where we may assume $\Omega\subset {\cal P}$.  So, by the induction hypothesis, for a fixed $d\in \hueca{N}$, 
 to obtain a finite sequence
of reductions
$$\underline{\cal A}^\ominus \mapsto \underline{\cal A}^{\ominus  z_1}
\mapsto
\cdots\mapsto
\underline{\cal A}^{\ominus  z_1z_2\cdots z_t}$$
of type $z_i\in \{a,d,r,q,X\}$ such that
$\underline{\cal A}^{\ominus  z_1z_2\cdots z_t}$ is a
minimal ditalgebra (with  $I^{\ominus  z_1z_2\cdots z_t}=0$), and the reduction functor 
$ F^{\ominus z_1}\cdots F^{\ominus z_t}$ is regular $d$-dense.

As remarked in (\ref{R: Fz and Fzominus}), 
we can perform a corresponding finite sequence of reductions $\underline{\cal A}\mapsto \underline{\cal A}^{z_1}\mapsto
\cdots\mapsto
\underline{\cal A}^{z_1z_2\cdots z_t}$ 
of the same type $z_i\in \{a,d,r,q,X\}$ as those applied successively to $\underline{\cal A}^\ominus$ in order to obtain 
$\underline{\cal A}^{\ominus  z_1\cdots z_t}=\underline{\cal A}^{z_1\cdots z_t\ominus }
$,  and there is a commutative
diagram
$$\begin{matrix}\underline{\cal
A}^{z_1\cdots z_t}\g\Mod&\rightmap{F^{z_t}}
&\cdots&\rightmap{F^{z_1}}&\underline{\cal A}\g\Mod\\
\shortlmapdown{\Res}&&&&\shortlmapdown{\Res}\\
  \underline{\cal A}^{z_1\cdots z_t\ominus }\g\Mod&\rightmap{F^{\ominus
z_t}}&\cdots&\rightmap{F^{\ominus  z_1}}&
  \underline{\cal A}^\ominus\g\Mod
  \end{matrix}$$
where $\underline{\cal A}^{z_1\cdots z_t\ominus }$  is a minimal ditalgebra (with ideal 
$I^{z_1\cdots z_t\ominus }=0$). Moreover, an easy induction shows that, for any $k$-algebra $E$, the induced commutative diagram 
$$\begin{matrix}\underline{\cal
A}^{z_1\cdots z_t}\g E\g\Mod&\rightmap{(F^{z_t})_E}
&\cdots&\rightmap{(F^{z_1})_E}&\underline{\cal A}\g E\g\Mod\\
\shortlmapdown{\Res_E}&&&&\shortlmapdown{\Res_E}\\
  \underline{\cal A}^{z_1\cdots z_t\ominus }\g E\g\Mod&\rightmap{(F^{\ominus
z_t})_E}&\cdots&\rightmap{(F^{\ominus  z_1})_E}&
  \underline{\cal A}^\ominus\g E\g\Mod
  \end{matrix}$$
has the same properties of the diagram in (\ref{L: FzE y Fzominus E}), if we replace the elementary reduction functor $F^z$ appearing there by the reduction functor $F^{z_1}\cdots F^{z_t}$. 
Since $\ell_E(\Res(M))\leq \ell_E(M)$ for any $\underline{\cal A}\g E$-bimodule $M$, we derive from this that  the reduction functor 
$F^{z_1}\cdots F^{z_t}$ is regular $d$-dense. 

 Now, we concentrate our attention on the non-wild triangular interlaced weak ditalgebra $\underline{\cal A}'= \underline{\cal A}^{z_1z_2\cdots z_t}$, see (\ref{P: Reducciones vs wildness}) having in mind that $\underline{\cal A}$ is a Roiter interlaced weak ditalgebra, as remarked in \cite[(5.4)]{bpsqh}. We know that $\underline{\cal A}'$ is  a stellar triangular interlaced weak ditalgebra due to the fact that $\underline{\cal A}^{\ominus z_1z_2\cdots z_t}$ is a minimal ditalgebra, see again the proof of \cite[(9.2)]{bpsqh}. 
 So, we can apply the construction of the proof of   \cite[(7.11)]{bpsqh} to the same $d$, to obtain a
composition of elementary reduction functors
$$\begin{matrix}\underline{\cal
A}''\g\Mod&\rightmap{F^{y_s}}&\cdots&\rightmap{F^{y_1}}&\underline{\cal
A}'\g\Mod\\
  \end{matrix}$$
such that $\underline{\cal A}''=\underline{\cal A}^{\prime y_1\cdots y_s}$ is a non-wild seminested ditalgebra (with zero ideal $I''$). Here, the composition functor $F^{y_1}\cdots F^{y_s}$ is regular length-controlling $d$-dense due to the fact that each one of these elementary reduction functors $F^{y_i}$ is a regular length-controlling  $d$-dense functor, due to (\ref{L: FdE, FaE, FqE}), (\ref{L: d-unravellings}), and (\ref{L: ideal-reduction}).  

Finally,   from \cite[(28.22)]{BSZ} applied to the same $d$, there is a minimal ditalgebra ${\cal B}$ and a
composition of elementary reduction functors $G:{\cal B}\g\Mod\rightmap{}{\cal A}''\g\Mod$, which is a regular $d$-dense functor. The reduction functors of type $X$ used in  \cite{BSZ} are either edge reductions or unravellings, which are included in our list of elementary reduction functors. 
Then,  the composition 
 $$F:=F^{z_1}\cdots F^{z_t}F^{y_1}\cdots F^{y_s}G:{\cal B}\g\Mod\rightmap{}\underline{\cal A}\g\Mod$$
 is the wanted regular $d$-dense reduction functor. 
\end{proof}

\begin{remark}\label{D: los Qi's}
 If ${\cal B}$ is a minimal ditalgebra with layer $(R,W)$, we have the decomposition of the unit element of $R$ as a sum $1=\sum_{i=1}^se_i$ of primitive orthogonal idempotents. Here, $B=R$ is   the underlying algebra of degree zero elements of the tensor algebra of ${\cal B}$. If $Be_i\not=k$, we denote by $Q_i$ the field of fractions of the rational $k$-algebra $Be_i$, considered as a $B$-module. So, $Q_i$ has a natural structure of $B\g k(x)$-bimodule. We identify  this bimodule  with   $Be_i\otimes_{e_iB}k(x)$.  
 
 From \cite[(31.6)]{BSZ}, we know that $\End_{\cal B}(Q_i)^{op} = E^0_{Q_i} \bigoplus \rad \End_{\cal B}(Q_i)^{op}$, where $E^0_{Q_i}$ is the subalgebra formed by the elements 
$(\mu_r,0)$, where $\mu_r$ denotes multiplication by  $r\in k(x)$, and the elements of   $\rad\End_{\cal B}(Q_i)^{op}$ have the form $f= (0,f^1)$. So,  $Q_i$ is a proper ${\cal B}\g E^0_{Q_i}$-bimodule.  We will keep this notation in the following. 
\end{remark}

\begin{theorem}\label{T: pregens de underline(A)} Assume that 
  $\underline{\cal A}=({\cal A},I)$ is a ${\cal P}$-oriented  triangular interlaced weak ditalgebra, as in   (\ref{D: biquiver P-orientado})\&(\ref{D: triangular interlaced weak ditalg}), where $I\subseteq \rad A$.  Suppose that $\underline{\cal A}$ is not wild and take $d\in \hueca{N}$. Consider the reduction  functor
$F : {\cal B}\g \Mod \rightmap{} \underline{\cal A}\g\Mod$ of  (\ref{T: parametrización de módulos de dim acotada}), where ${\cal B}$ is a minimal ditalgebra.  The following holds:   
\begin{enumerate}
 \item The $\underline{A}\g B$-bimodule $\overline{T}:=F(B)$ is finitely generated by the right and  the composition 
$$B\g\Mod\rightmap{L_{\cal B}}{\cal B}\g\Mod\rightmap{F}\underline{\cal A}\g\Mod$$ is naturally isomorphic to $L_{\underline{\cal A}}(\overline{T}\otimes_B-)$.
 So, $F(N)\cong \overline{T}\otimes_B N$, for any $N\in {\cal B}\g\Mod$. If $N$ is a proper ${\cal B}\g E$-bimodule, for some $k$-algebra $E$, then we have an isomorphism of $\underline{A}\g E$-bimodules $F_E(N)\cong \overline{T}\otimes_B(N_E)$.
 \item The $\underline{\cal A}$-modules of the form $G\cong \overline{T}\otimes_BQ_i$, for $i\in [1,s]$, are pregeneric.

\item If $M$ is an indecomposable $\underline{\cal A}$-module with  $\Endol{M}\leq  d$, there is an indecomposable ${\cal B}$-module $N$ with finite endolength such that $\overline{T}\otimes_B N\cong M$.

 \item If the given $\underline{\cal A}$-module $M$ is pregeneric, then the $B$-module $N$ is pregeneric and  has a natural structure of a $B\g k(x)$-bimodule. Moreover, for the $\underline{\cal A}$-module $\underline{M}:=\overline{T}\otimes_BN$, we have a decomposition 
$$\End_{\underline{\cal A}}(\underline{M})^{op}=E^0_{\underline{M}}\bigoplus \rad\End_{\underline{\cal A}}(\underline{M})^{op},$$
for some subalgebra $E^0_{\underline{M}}$ and the ideal $\rad\End_{\underline{\cal A}}(\underline{M})^{op}$ is nilpotent.
   There is an isomorphism  of algebras 
$\rho  : k(x)\rightmap{}E_{\underline{M}}^0$ given by $\rho(r) = (id_{\overline{T}} \otimes \mu_r,0)$, where $\mu_r$ denotes multiplication by  $r$, which endows $\underline{M}$  with a structure of an  $\underline{A}\g k(x)$-bimodule such that $$\Endol{M} =\Endol{\underline{M}}= \dim_{k(x)}\underline{M}.$$
\end{enumerate}
\end{theorem}

\begin{proof} $(1)$ follows from 
\cite[(6.13)]{bpsqh}, where we can replace the functors $F_\phi$ by the functors $F'$ as in (\ref{P: la nueva}). For $N=Q_i$ and $E=E_{Q_i}^0$,  we have  isomorphisms $\psi:F(Q_i)\rightmap{}F(B\otimes_BQ_i)$ and $\phi:  F(B \otimes_BQ_i)\rightmap{}F(B) \otimes_B Q_i$ of $\underline{A}$-modules such that, for any $r \in k(x)$, the following diagram commutes in $\underline{\cal A}\g\Mod$
$$\begin{matrix}
F(Q_i)&\rightmap{(\psi,0)}&F(B\otimes_BQ_i)& \rightmap{(\phi,0)}& F(B)\otimes_BQ_i\\
\shortrmapdown{F(\mu_r,0)}&&\shortrmapdown{F(1_B \otimes \mu_r ,0)}&&\shortrmapdown{(1_{F(B)} \otimes \mu_r ,0)}\\
F(Q_i)&\rightmap{(\psi,0)}& F(B\otimes_BQ_i) &\rightmap{(\phi,0)}& F(B)\otimes_BQ_i
\end{matrix}$$ 
 The proof of (2) is similar to the proof of \cite[(4.3)]{CB2}: As remarked in (\ref{D: los Qi's}), we have the  decomposition
$$\End_{{\cal B}}(Q_i)^{op}=E^0_{Q_i}\bigoplus\rad\End_{{\cal B}}(Q_i)^{op},$$
and $Q_i$ is a proper ${\cal B}\g E^0_{Q_i}$-bimodule. So, $F(Q_i)$ is a proper $\underline{\cal A}\g E^0_{Q_i}$-bimodule, which is isomorphic to the proper $\underline{\cal A}\g E^0_{Q_i}$-bimodule $F(B)\otimes_BQ_i$ through $(\phi\psi,0)$. 
Since the functor $F$ is full and faithful, we have isomorphisms of algebras 
$$\End_{\cal B}(Q_i)^{op}\rightmap{F_\vert}\End_{\underline{\cal A}}(F(Q_i))^{op}\rightmap{}\End_{\underline{\cal A}}(F(B)\otimes_B Q_i)^{op},$$
where the second isomorphism is conjugation by $(\phi\psi,0)$. 
For $u=(\mu_r,0)\in E^0_{Q_i}$ and $t\otimes n\in F(B)\otimes_BQ_i$, we have $(t\otimes n)u=t\otimes \mu_r(n)$. Therefore, we have a decomposition 
$$\hbox{(*): \hbox{ \ \ \ \ \ \ }}\End_{\underline{\cal A}}(F(B)\otimes_B Q_i)^{op}=E^0_{F(B)\otimes_BQ_i}\bigoplus 
\rad\End_{\underline{\cal A}}(F(B)\otimes_B Q_i)^{op}$$
where $E^0_{F(B)\otimes_BQ_i}\cong E^0_{Q_i}\cong k(x)$. By definition, the endolength of $F(B)\otimes_BQ_i$ is its length as an $\End_{\underline{\cal A}}(F(B)\otimes_BQ_i)^{op}$-module where the action  is given by  $(t\otimes n)(u^0,u^1)=(t\otimes n)u^0$. Hence, we have that $$\Endol{F(B)\otimes_B Q_i}=\dim_{k(x)}F(B)\otimes_B Q_i,$$
 where $k(x)$ acts on $F(B)\otimes_BQ_i$ through the isomorphism $\rho:k(x)\rightmap{}E^0_{F(B)\otimes_BQ_i}$ given by $\rho(r)=(1_{F(B)}\otimes \mu_r,0)$, for $r\in k(x)$. 
Since $F_{k(x)}$ is length-controlling,  we obtain that $F(B)\otimes_BQ_i$ is pregeneric. 
\medskip

\noindent(3):  Here, $M$ is an $\underline{\cal A}\g E$-bimodule with  length $\ell_E(M)\leq d$, where 
$E = \End_{\underline{\cal A}}(M)^{op}$ and its bimodule structure map 
$\alpha_M : E \rightmap{}
\End_{\underline{\cal A}}(M)^{op}$ is the identity.
From (\ref{T: parametrización de módulos de dim acotada}), we know that there is a ${\cal B}\g E$-bimodule  $N$ with finite length $\ell_E(N)$ such that 
$F_E(N)\cong M$. 

Since the functor $F$ is full and faithful, we get 
 $E\cong \End_{\cal B}(N)^{op}$ and, hence, $\Endol{N}=\ell_E(N)$ is finite.  By \cite[(31.6)]{BSZ}, the module $N$ over the minimal algebra $B$ has finite endolength. 
If $M$ is finite-dimensional, so is $N$. 
\medskip

\noindent(4): If $M$ is infinite-dimensional, $N$ has the same property, so, $N$ is a pregeneric $B$-module.    
 It is known that the $B$-modules $Q_i$ provide a complete set of representatives of the isomorphism  classes of the  pregeneric
$B$-modules, see \cite[(31.3)]{BSZ}. So, we can assume that $N\cong Q_i$, for some $i$.  From the decomposition in (*), we get the description of the action of $k(x)$ on $\underline{M}$ and the wanted last equality.  
\end{proof}

\noindent{\bf Proof of (\ref{T: cal(A) pregen tame sii tame})} We know that $\underline{\cal A}$ is tame iff it is not wild, by \cite[(10.4)]{bpsqh}. From (\ref{P: gen tame => non wild}), we know that if $\underline{\cal A}$ is pregenerically tame, then it is not wild. So, it remains to show that if $\underline{\cal A}$ is not wild, then it is pregenerically tame.
So, take $d\in \hueca{N}$ and a pregeneric $\underline{\cal A}$-module $G$ with $\Endol{G}\leq d$. Apply (\ref{T: pregens de underline(A)}) to $d$ to obtain a reduction functor $F:{\cal B}\g \Mod\rightmap{}\underline{\cal A}\g\Mod$  and a pregeneric $B$-module $G'$ with $F(G')\cong G$.  Since there are only finitely many pregeneric $B$-modules $G'$, up to isomorphism, it follows that $\underline{\cal A}$ is pregenerically tame. 
$\hfill\square$

\section{On the category  of modules for a special $\underline{\cal A}$} 

In this section, we consider  a \emph{special triangular interlaced weak ditalgebra $\underline{\cal A} = ({\cal A}, I)$}, as defined in \cite[\S12]{hsb}. That is such that its layer $(S,W)$ satisfies: $S$ is a finite product of copies of the field $k$, $W$ is
finite-dimensional, and so is $\underline{A}=A/I$, where $A=T_S(W_0)$. Thus, $\underline{\cal A}$ is a Roiter interlaced weak ditalgebra. 

It is known that, in this case, the category $\underline{\cal A}\g\Mod$ admits an exact structure ${\cal E}$ where the conflations are the composable pairs of morphisms $M\rightmap{f}E\rightmap{g}N$ such that $gf=0$ and the sequence 
$0\rightmap{}M\rightmap{f^0}E\rightmap{g^0}N\rightmap{}0$
is exact in $S\g\Mod$, see \cite[\S11]{hsb}.

The \emph{right algebra of $\underline{\cal A}$} is the finite-dimensionl $k$-algebra
$\Gamma := \End_{\underline{\cal A}}(\underline{A})^{op}$.
The composition
$
\underline{A}\rightmap{}\End_{\underline{A}}(\underline{A})^{op}\rightmap{}\End_{\underline{\cal A}}(A)^{op}=\Gamma$ is an embedding of $k$-algebras,  
so $\Gamma$ is naturally an $\underline{A}$-algebra. The functor $H:\underline{\cal A}\g\Mod\rightmap{} \Gamma\g\Mod$, given by  $H=\Hom_{\underline{\cal A}}(\underline{A},-)$ is an exact functor, because $\underline{A}$ is ${\cal E}$-projective,  see \cite{BuBt} and \cite[\S12]{hsb}. In this section we explore some other properties of the functor $H$. For a conceptual approach to the proof of these properties and more,   the notion of compactness  defined below   is helpful.

\begin{remark}\label{R: compacts} Let ${\cal C}$ be an additive  category with arbitrary coproducts. Thus, any family $\{M_i\}_{i\in I}$ of objects in ${\cal C}$ has a coproduct $\coprod_{i\in I}M_i$ equipped with its family of  injections $\sigma_{M_j}:M_j\rightmap{}\coprod_{i\in I}M_i$, for $j\in I$. The universal property of the coproduct determines, for every $j\in I$, a morphism $\pi_{M_j}:\coprod_{i\in I}M_i\rightmap{}M_j$ such that $\pi_{M_j}\sigma_{M_j}=id_{M_j}$ and $\pi_{M_i}\sigma_{M_j}=0$ for any $i\not=j$. We call the morphisms $\pi_{M_j}$ \emph{the projections of the coproduct}.  Assume that the category ${\cal C}$ satisfies the following additional requirement: for any family of objects   $\{M_i\}_{i\in I}$ in ${\cal C}$ and any morphism $f:N\rightmap{}\coprod_{i\in I}M_i$ such  that $\pi_{M_i}f=0$, for all $i\in I$, we have that $f=0$. 

We say that an object $N$ of ${\cal C}$ is \emph{compact} iff for any family of objects $\{M_i\}_{i\in I}$ in ${\cal C}$ and any  morphism  $f:N\rightmap{}\coprod_{i\in I}M_i$ we have that $\pi_{M_i}f=0$, for almost every $i\in I$. In this case, we have 
$$f=\sum_{i\in I}\sigma_{M_i}\pi_{M_i}f.$$  
\end{remark}

\begin{remark}\label{R: coprod} The category $\underline{\cal A}\g \Mod$ has products and coproducts, see \cite[(3.7)]{bpsqh} and \cite[(2.9)]{BSZ}. Given a family $\{M_i\}_{i\in I}$ in $\underline{\cal A}\g\Mod$, their coproduct $\coprod_{i\in I}M_i$ in $\underline{\cal A}\g\Mod$ is the coproduct $\coprod_{i\in I}M_i$ in $\underline{A}\g\Mod$ with injections $\sigma_{M_j}:M_j\rightmap{}\coprod_{i\in I}M_i$ given by $\sigma_{M_j}=(\sigma^0_{M_j},0)$, where   $\sigma^0_{M_j}$ are the injections of the coproduct $\coprod_{i\in I}M_i$ in $\underline{A}\g\Mod$. 
The product $\prod_{i\in I}M_i$ of the given family  in $\underline{\cal A}\g\Mod$ is constructed similarly.  

The canonical morphism of $\underline{\cal A}$-modules  $\coprod_{i\in I}M_i\rightmap{}\prod_{i\in I}M_i$ is a monomorphism, which implies that the category ${\cal C}=\underline{\cal A}\g\Mod$ satisfies the properties required in (\ref{R: compacts}).

  Also, given a finite-dimensional algebra $\Lambda$, any  full and replete subcategory ${\cal C}$ of $\Lambda\g\Mod$ which is closed under coproducts has the properties required in (\ref{R: compacts}). This follows, from the fact that the canonical morphisms $\coprod_{i\in I}M_i\rightmap{}\prod_{i\in I}M_i$ in $\Lambda\g\Mod$ are monomorphisms. 
\end{remark}

\begin{lemma}\label{L: compacts in cal(A)-Mod} The compact objects of $\underline{\cal A}\g\Mod$ are the finite-dimensional objects. 
\end{lemma}

\begin{proof} We already know that the 
 finite-dimensional spaces are compact objects of $k\g\Mod$. Let $\{M_i\}_{i\in I}$ be a family of $\underline{\cal A}\g\Mod$ and $N$ a 
finite-dimensional $\underline{\cal A}$-module. Consider any morphism $f:N\rightmap{}\coprod_{i\in I}M_i$. So, $f$  has components 
$$f^0\in \Hom_k(N,\coprod_{i\in I}M_i)\hbox{ \  and \ } f^1\in \Hom_{A\g A}(V,\Hom_k(N,\coprod_{i\in I}M_i)).$$ 
In order to show that $\pi_{M_i}f=0$, for almost every $i\in I$, we will show that each one of 
$(\pi_{M_i}f)^0$
and $(\pi_{M_i}f)^1$ is zero for almost all $i$.  
Since $N$ is finite-dimensional,  
$(\pi_{M_i}f)^0=\pi^0_{M_i}f^0=0$, for almost all $i$.  
Since $V$ is also finite-dimensional, say with finite basis ${\cal V}$,  
we just have to show that for each basis element  $v\in {\cal V}$, we have that 
$(\pi_{M_i}f)^1
(v)$ is zero for almost all $i$. For such an element $v\in {\cal V}$, using again that  $N$ is finite-dimensional, we see that the linear map $f^1(v)$ satisfies that  
$(\pi_{M_i}f)^1
(v)=\pi^0_{M_i}f^1(v)=0$, for almost all $i$. 

For the converse, it is enough to notice that any compact object $N$ of $\underline{\cal A}\g\Mod$ is compact in $\underline{A}\g\Mod$, because, as is well known, the compact $\underline{A}$-modules are the finite-dimensional ones.  Indeed,  a morphism $f^0:N\rightmap{}\coprod_{i\in I}M_i$ in $\underline{A}\g\Mod$ determines the morphism $f=(f^0,0):N\rightmap{}\coprod_{i\in I}M_i$ in $\underline{\cal A}\g\Mod$, thus $\pi_{M_i}f=0$, for almost all $i$. Thus $\pi^0_{M_i}f^0=0$, for almost all $i$. 
\end{proof}

We will require the determination of the  compact objects for the following categories. 

\begin{definition}\label{R: compactos de filtrados}   
Let ${\cal S}$ be a collection of indecomposable modules in $\Lambda\g \mod$, where $\Lambda$ is any finite-dimensional algebra. Define $\widetilde{\cal F}({\cal S})$ as the full subcategory of $\Lambda\g\Mod$ formed by the $\Lambda$-modules $M$ with a finite ${\cal S}$-filtration, that is a filtration of the form  
  $$\hueca{F}:\hbox{ \ \ }
0=M_0\subset M_1\subset M_2\subset\cdots \subset M_\ell=M$$ 
such that each factor $M_i/M_{i-1}$ is isomorphic to some possibly infinite direct sum of objects in ${\cal S}$. 
\end{definition}

It follows from the definition that $\widetilde{\cal F}({\cal S})$ is closed under extensions.   

\begin{lemma}\label{L: compacts in cal(F)(Delta)} With the preceding notation, assume furthermore that $\widetilde{\cal F}({\cal S})$ is closed under coproducts. 
Then, the compact objects of $\widetilde{\cal F}(\cal S)$ are the finite-dimensional ones. 
\end{lemma}

\begin{proof} It is well known that if $M\in \widetilde{\cal F}({\cal S})$ is finite-dimensional, then $M$ is compact in $\Lambda\g\Mod$, hence it is so in $\widetilde{\cal F}(\cal S)$. For the converse, we notice first that, for any epimorphism $p:M\rightmap{}N$ in $\widetilde{\cal F}({\cal S})$ with $M$  a compact object of $\widetilde{\cal F}({\cal S})$, we also have that $N$ is  a compact object in $\widetilde{\cal F}({\cal S})$. It is immediate that any  object of $\widetilde{\cal F}({\cal S})$ of the form $\bigoplus_{i\in I}S_i$, with each $S_i\in {\cal S}$, is compact iff $I$ is finite. 

The proof of the proposition will follow by induction on the length $\ell$ of the filtration $\hueca{F}$ of the compact object  $M\in \widetilde{\cal F}({\cal S})$ described in the preceding definition. The base of the induction, when $\ell=1$, follows from the preceding paragraph. Now, assume that $\ell>1$ and suppose that any  compact object with an ${\cal S}$-filtration of length $\ell-1$ is finite-dimensional.  

From the filtration $\hueca{F}$ as above, we obtain the following filtration for $M/M_1$ with length $\ell-1$:  
$$0\subset M_2/M_1\subset\cdots\subset M/M_1$$
where each factor $(M_t/M_1)/(M_{t-1}/M_1)\cong M_t/M_{t-1}$ is a direct sum of  objects of ${\cal S}$. So, $M/M_1\in \widetilde{\cal F}({\cal S})$ and it is compact. Our induction hypothesis implies  that $M/M_1$ is finite-dimensional. Moreover, we have an isomorphism $\theta:M_1\rightmap{} \bigoplus_{j\in J}S_j$, with $S_j\in {\cal S}$. Consider the exact sequence 
$$\xi: \hbox{ \ \ \ }0\rightmap{}M_1\rightmap{s}M\rightmap{p}M/M_1\rightmap{}0,$$
and its equivalence class $[\xi]$  in $\Ext_\Lambda(M/M_1,M_1)$. The isomorphism $\theta$  induces an isomorphism 
$\Ext_\Lambda(M/M_1,M_1)\cong \Ext_\Lambda(M/M_1,\bigoplus_{j\in J}S_j)$ mapping $[\xi]$ on 
$[\theta\xi ]\in \Ext_\Lambda(M/M_1,\bigoplus_{j\in J}S_j)$. For each subset $J'$ of $J$, we have  a commutative diagram 
$$\begin{matrix}\Ext_\Lambda(M/M_1,\bigoplus_{j\in J}S_j)&\rightmap{\cong}&\bigoplus_{j\in J}\Ext(M/M_1,S_j)\\
\shortlmapdown{\pi_{J'}}&&\shortlmapdown{p_{J'}}\\
\Ext_\Lambda(M/M_1,\bigoplus_{j\in J'}S_j)&\rightmap{\cong}&\bigoplus_{j\in J'}\Ext(M/M_1,S_j),\\
\end{matrix}$$
 with $\pi_{J'}:=\Ext(M/M_1,q_{J'})$,  where $q_{J'}:\bigoplus_{j\in J}S_j\rightmap{}\bigoplus_{j\in J'}S_j$ and $p_{J'}$ are the corresponding canonical projections. 
For some cofinite subset $J'$ of $J$, we have that $[ q_{J'}\theta\xi]=\pi_{J'}([\theta\xi])=0$. 
 This means that the pushout sequence obtained from $\xi$ using $q_{J'}\theta $ splits. So, there is a morphism $t:M\rightmap{}\bigoplus_{j\in J'}S_j$ such that $q_{J'}\theta =ts$. Hence, $t$ is an epimorphism and $\bigoplus_{j\in J'}S_j$ is a compact object in $\widetilde{\cal F}({\cal S})$. Hence, $J'$ is finite, and cofinite in $J$, we get that $J$ is finite. So, $M_1$ is finite-dimensional, and so is $M$, as we wanted to show. 
\end{proof}

\begin{remark}\label{L: Hom(N,-) preserva coprods} Let ${\cal C}$ be an additive category with coproducts as required in (\ref{R: compacts}) and $N$ a compact object of ${\cal C}$. Define $\Gamma:=\End_{\cal C}(N)^{op}$. Then  
 the functor $\Hom_{\cal C}(N,-):{\cal C}\rightmap{}\Gamma\g\Mod$ preserves coproducts. More precisely, for any family $\{M_i\}_{i\in I}$ of objects in ${\cal C}$,  there is an isomorphism of $\Gamma$-modules 
$$\Hom_{\cal C}(N,\coprod_{i\in I}M_i)\rightmap{\Theta}\coprod_{i\in I}\Hom_{\cal C}(N,M_i)\hbox{ \ such that \ } 
f\longmapsto \sum_{i\in I}s_i(\pi_{M_i}f),$$ 
where $s_j:\Hom_{\cal C}(N,M_j)\rightmap{}\coprod_{i\in I}\Hom_{\cal C}(N,M_i)$ denotes the canonical injection in $\Gamma\g\Mod$. Indeed, the inverse of $\Theta$ is the map $\Theta'$ given by 
$\sum_{i\in I}s_i(f_i)\mapsto \sum_{i\in I}\sigma_{M_i}f_i$, where $f_i=0$, for almost all $i$.   From (\ref{R: compacts}), we immediately derive that $\Theta$ and $\Theta'$ are mutual inverses. It is  clear that $\Theta$ is a morphism of $\Gamma$-modules. 
\end{remark}

\begin{corollary}\label{C: H preserva coprods}\label{L: Hom(A,-) preserva sumas directas}  The functor $H=\Hom_{\underline{\cal A}}(\underline{A},-):\underline{\cal A}\g\Mod\rightmap{}\Gamma\g\Mod$ preserves coproducts, where again $\Gamma$ denotes  the right algebra of $\underline{\cal A}$.
\end{corollary}

\begin{proof} This follows from (\ref{R: coprod}) and (\ref{L: Hom(N,-) preserva coprods}), because $\underline{A}$ is finite-dimensional and, hence, compact by (\ref{L: compacts in cal(A)-Mod}).
\end{proof}

For the proof of the next proposition we work with the basic properties of the following  cokernel functor.

\begin{remark}\label{R: Coker}
Assume that $({\cal C},{\cal E})$ is an exact category. Let us denote by ${\cal P}({\cal C})$ the category of morphisms $f:P_2\rightmap{}P_1$ of ${\cal C}$ where $P_1$ and $P_2$ are ${\cal E}$-projective objects.   We denote by ${\cal P}^1({\cal C})$ the full subcategory of ${\cal P}({\cal C})$ formed by the morphisms $f:P_2\rightmap{}P_1$  such that $f=sp$, where $s:L\rightmap{}P_1$ is an inflation and $p:P_2\rightmap{}L$ is a deflation.

Observe that any $f:P_2\rightmap{}P_1$ in ${\cal P}^1({\cal C})$ has a cokernel in ${\cal C}$. Indeed, if $f=sp$ as before, there is a conflation $L\rightmap{s}P_1\rightmap{q}Q$ and it follows that $q$ is a cokernel for $f$.  So, we have the cokernel functor $\Cok:{\cal P}^1({\cal C})\rightmap{}{\cal C}$. 

Notice that $P\rightmap{}0$ or $P\rightmap{id}P$ belong to ${\cal P}^1({\cal C})$, for any ${\cal E}$-projective object $P$ in ${\cal C}$. We have the ideal ${\cal K}$ of the category ${\cal P}^1({\cal C})$ formed by the morphisms in ${\cal P}^1({\cal C})$ that are sums of morphisms which factor through objects of type $P\rightmap{}0$ or $id_P:P\rightmap{}P$ for some ${\cal E}$-projective object $P$. 
 Define the category $\underline{\cal P}^1({\cal C})$ as the quotient category of ${\cal P}^1({\cal C})$ modulo  the ideal ${\cal K}$. 
Then, the functor $\Cok$ induces a full and faithful functor $$\underline{\Cok}:\underline{\cal P}^1({\cal C})\rightmap{}{\cal C}.$$ Indeed, the ideal ${\cal K}$ is precisely the kernel of the functor $\Cok$, see the argument in the proof of \cite[(2.2)]{genbricks}.   If, moreover, the exact category $({\cal C},{\cal E})$ has enough projectives, then $\underline{\Cok}$ is an equivalence of categories because, in this case, any object $M$ in ${\cal C}$ admits an ${\cal E}$-projective presentation $P_2\rightmap{f}P_1\rightmap{}M$, with $f\in {\cal P}^1({\cal C})$.  
\end{remark}

\begin{definition}\label{D: admissible exact cats} Consider an exact category $({\cal C},{\cal E})$ with arbitrary coproducts and enough projectives. Assume furthermore that ${\cal C}$ satisfies the requirements of (\ref{R: compacts}) and that every ${\cal E}$-projective is a coproduct of compact objects. We  call such exact categories \emph{admissible}.  
We  denote by ${\cal C}_c$ the full subcategory of ${\cal C}$ formed by the compact objects of ${\cal C}$. 
\end{definition}

\begin{remark}\label{R: morfismos entre sumas de compactos} Let $M=\coprod_{i\in I}M_i$ and $N=\coprod_{j\in J}N_j$ be coproducts of familes of compact objects $\{M_i\}_{i\in I}$ and $\{N_j\}_{j\in J}$ in an additive category ${\cal C}$ as in (\ref{R: compacts}). Then, from (\ref{L: Hom(N,-) preserva coprods}) we obtain that any morphism $f:M\rightmap{}N$ in ${\cal C}$ is completely determined by a family of morphisms $\{f_{j,i}:M_i\rightmap{}N_j\}_{(j,i)\in J\times I}$ such that, for each $i$, we have $f_{j,i}=0$ for almost all $j$ and $f\sigma_{M_i}=\sum_j\sigma_{N_j}f_{j,i}$.
\end{remark}

\begin{proposition}\label{P: fiel y pleno en compactos implica fiel y pleno}
Let ${\cal C}$ and ${\cal C}'$ be admissible exact categories and suppose  that $G:{\cal C}\rightmap{}{\cal C}'$ is an exact functor that preserves coproducts and projectives. Assume, furthermore, that 
 it induces a full and faithful functor  $G_\vert:{\cal C}_c\rightmap{}{\cal C}'_c$ such that every compact ${\cal E}'$-projective object $Q$ is of the form $Q\cong G(P)$, for some compact ${\cal E}$-projective object $P$. Then, $G$ is full and faithful. 

If $G_1$ and $G_2$ are functors with the above properties and $\phi:G_1\rightmap{}G_2$ is a natural transformation such that its restriction $\phi_\vert:{G_1}_\vert\rightmap{}{G_2}_\vert$ is an isomorphism, then $\phi$ is an isomorphism.    
\end{proposition}
 
\begin{proof}  First observe that given $M=\coprod_{i\in I}M_i$ and $N=\coprod_{j\in J}N_j$, where every  $M_i$ and $N_j$ are compact objects in ${\cal C}_c$, we have that 
$$G:\Hom_{{\cal C}}(M,N)\rightmap{}\Hom_{\cal C'}(G(M),G(N))$$
 is an isomorphism. Indeed, from  (\ref{R: morfismos entre sumas de compactos}), we already know that a morphism $f:M\rightmap{}N$ in ${\cal C}$   is completely determined by a family of morphisms $\{f_{j,i}:M_i\rightmap{}N_j\}_{(j,i)\in J\times I}$  such that, for each $i$, we have $f_{j,i}=0$ for almost all $j$, with $f\sigma_{M_i}=\sum_{j\in J}\sigma_{N_j}f_{j,i}$, for each $i$. Then we have the formulas: 
 $$(*):\hbox{ \ \ \ \ }G(f)G(\sigma_{M_i})=\sum_{j\in J}G(\sigma_{N_j})G(f_{j,i}).$$
 Since $G$ preserves coproducts, we know that $G(M)$ is the coproduct of the family $\{G(M_i)\}_{i\in I}$ with injections $G(\sigma_{M_i}):G(M_i)\rightmap{}G(M)$ and $G(N)$ is the coproduct of the family $\{G(N_j)\}_{j\in J}$ with injections $G(\sigma_{N_j}):G(N_j)\rightmap{}G(N)$. A morphism $g:G(M)\rightmap{}G(N)$ is completely determined by a family of morphisms $\{g_{j,i}:G(M_i)\rightmap{}G(N_j)\}_{(j,i)\in J\times I}$ satisfying the formulas: 
  $$(**):\hbox{ \ \ \ \ }gG(\sigma_{M_i})=\sum_{j\in J}G(\sigma_{N_j})g_{j,i}.$$
  If we have 
$g=G(f)=0$ then, from $(*)$,  
we have $G(f_{j,i})=g_{j,i}=0$, for all $i,j$.  Since $G_\vert$ is a faithful functor, we obtain that $f_{j,i}=0$, for all $i,j$, thus $f=0$ and $G$ is injective. 
 
 Moreover, given any morphism $g:G(M)\rightmap{}G(N)$, since $G_\vert$ is full, there are morphisms $f_{j,i}:M_i\rightmap{}N_j$ with $G(f_{j,i})=g_{j,i}$, for all $i,j$. The family $\{f_{j,i}\}_{j,i}$ defines a morphism $f:M\rightmap{}N$ in ${\cal C}$ which satisfies the formulas preceding 
 $(*)$. Applying $G$, we obtain the formulas $(*)$.  It follows that $G(f)=g$, and that $G$ is surjective.  

Let us show now that $G$ is full and faithful. Since $({\cal C}, {\cal E})$ is admissible, if $P$ is ${\cal E}$-projective in ${\cal C}$, it is a coproduct of compact ${\cal E}$-projective objects. Hence, by assumption,  $G(P)$ is a coproduct of compact  ${\cal E}'$-projective objects. 

Since $G$ is exact, from the preceding paragraph, the functor $G$ induces a full and faithful functor $\widehat{G}:{\cal P}^1({\cal C})\rightmap{}{\cal P}^1({\cal C}')$. By assumption, every ${\cal E}'$-projective  object $Q$ is of the form $G(P)\cong Q$, for some ${\cal E}$-projective $P$. So the objects $Q\rightarrow 0$ and $id_Q:Q\rightmap{} Q$, where $Q$ is ${\cal E}'$-projective, are of the form $\widehat{G}(P\rightarrow 0)$ and $\widehat{G}(P\rightmap{id} P)$, for some ${\cal E}$-projective $P$. 
It follows that $\widehat{G}({\cal K})={\cal K}'$, where ${\cal K}$ and ${\cal K}'$ denote the kernels of the corresponding Cokernel functors.  
 Then we have an induced full and faithful functor $\underline{\widehat{G}}:\underline{\cal P}^1({\cal C})\rightmap{} \underline{\cal P}^1({\cal C}')$. It belongs to the following diagram 
$$\begin{matrix}
\underline{\cal P}^1({\cal C})&\rightmap{\underline{\Cok}}&{\cal C}\\
\shortlmapdown{\underline{\widehat{G}}}&&\shortlmapdown{G}\\
\underline{\cal P}^1({\cal C}')&\rightmap{\underline{\Cok}}&{\cal C}',\\
\end{matrix}$$
which commutes up to isomorphism. This natural isomorphism is induced by the isomorphism $\eta: \Cok \widehat{G}\rightmap{}G\Cok$, where each term of the family $\{\eta_X\}_X$ is defined, for $X=(P\rightmap{f}Q)$, by the commutative diagram 
$$\begin{matrix}
G(P)&\rightmap{G(f)}&G(Q)&\rightmap{}&\Cok \widehat{G}(X)&\rightmap{}&0\\
\shortlmapdown{id_P}&&\shortlmapdown{id_Q}&&\shortlmapdown{\eta_X}&\\
G(P)&\rightmap{G(f)}&G(Q)&\rightmap{}&G  \Cok (X)&\rightmap{}&0.\\
\end{matrix}$$ 
 It follows that $G$ is full and faithful. 

For the proof of the last statement of our Proposition, assume that $\phi_\vert:(G_1)_\vert\rightmap{}(G_2)_\vert$ is an isomorphism of functors from ${\cal C}_c$ to ${\cal C}'_c$. Notice that, whenever $M=\coprod_{i\in I}M_i$ with $M_i\in {\cal C}_c$, the  morphism $\phi_M$ in ${\cal C'}$ is unique such that the following diagram commutes for all $i\in I$ 
$$\begin{matrix}
 G_1(M)&\rightmap{ \ \ \phi_M \ \ }&G_2(M)\\
\shortrmapup{G_1(\sigma_{M_i})}&&\shortlmapup{G_2(\sigma_{M_i})}\\
G_1(M_i)&\rightmap{ \ \ \phi_{M_i} \ \ }&G_2(M_i).\\
\end{matrix}$$
Here $\phi_M$ is an isomorphism in ${\cal C}'$, because we can construct its inverse working with the inverses of the isomorphisms $\phi_{M_i}$. 

Given any object $X=(P\rightmap{f}Q)$ of ${\cal P}^1({\cal C})$, notice 
 that, from the naturality of $\phi$, the morphism $\phi_{\Cok X}$  is unique such that the following diagram commutes 
$$\begin{matrix}
G_1(P)&\rightmap{G_1(f)}&G_1(Q)&\rightmap{}&G_1(\Cok X)&\rightmap{}&0\\
\shortlmapdown{\phi_P}&&\shortlmapdown{\phi_Q}&&\shortlmapdown{\phi_{\Cok X}}&&\\
G_2(P)&\rightmap{G_2(f)}&G_2(Q)&\rightmap{}&G_2(\Cok X)&\rightmap{}&0.\\
\end{matrix}$$
Hence, $\phi_{\Cok X}$ is an isomorphism.

Since ${\cal C}$ has enough projectives, we can choose for any $M\in {\cal C}$ an object $X_M\in {\cal P}^1({\cal C})$ and an isomorphism $\theta_M:\Cok X_M\rightmap{}M$. Again, from the naturality of $\phi$, we have that $\phi_M=G_2(\theta_M)\phi_{\Cok X_M}G_1(\theta_M)^{-1}$ is an isomorphism.  Then, $\phi:G_1\rightmap{}G_2$ is a natural isomorphism.  
\end{proof}

\begin{proposition}\label{L: el funtor F fiel, pleno y exacto}
The functor $H:\underline{\cal A}\g\Mod\rightmap{}\Gamma\g\Mod$ is exact, full and faithful. 
\end{proposition}

\begin{proof} As remarked in (\ref{R: coprod}) and (\ref{L: compacts in cal(A)-Mod}),  $\underline{\cal A}\g\Mod$ has coproducts,  satisfies the properties of (\ref{R: compacts}), and $\underline{\cal A}\g\mod$ is its category of compact objects. 

 The  category $\underline{\cal A}\g\Mod$ with the given exact structure ${\cal E}$ has enough projectives and  every ${\cal E}$-projective  module is a direct sum of finite-dimensional $\underline{\cal A}$-modules,  see \cite[(6.20)]{BSZ}. Thus, $\underline{\cal A}\g\Mod$ is admissible. 
 
 The exact functor $H$ preserves coproducts by (\ref{C: H preserva coprods}). Hence, it preserves projectives because $H(\underline{A})=\Gamma$. Moreover, its restriction  $H_\vert:\underline{\cal A}\g\mod\rightmap{}\Gamma\g\mod$ to finite-dimensional modules is full and faithful by \cite[(12.4)]{hsb}.  Thus $H$ satisfies the requirements (\ref{P: fiel y pleno en compactos implica fiel y pleno}), and $H$ has the stated properties.  
\end{proof}

\begin{lemma}\label{L: F es Hom(A,L)} 
For each $M\in \underline{A}\g\Mod$, there is a natural isomorphism 
$$\sigma_M:\Gamma\otimes_{\underline{A}}M\rightmap{}\Hom_{\underline{\cal A}}(\underline{A},L_{\underline{\cal A}}(M))$$
of functors from $\underline{A}\g\Mod$ onto $\Gamma\g\Mod$.
 So the funtor $H$ restricts to an equivalence of categories $H:\underline{\cal A}\g\Mod\rightmap{}\widetilde{\cal I}$,
 where $\widetilde{\cal I}$ denotes the full subcategory of $\Gamma\g\Mod$ formed by the  modules induced from $\underline{A}\g\Mod$, that is by the class of $\Gamma$-modules isomorphic to some $\Gamma\otimes_{\underline{A}}N$, with $N\in \underline{A}\g\Mod$.
\end{lemma}

\begin{proof} By definition, $\sigma_M(f\otimes m)=f_m=(f^0_m,f^1_m)$, where $f^0_m(a)=f^0(a)m$ and $f^1_m(v)[a]=f^1(v)[a]m$, for $a\in \underline{A}$, $v\in V$, $m\in M$, and $f\in \Gamma$, see \cite[(12.5)]{hsb}. The functors   $HL_{\underline{\cal A}}$ and  $\Gamma\otimes_{\underline{A}}-$ are exact functors from $\underline{A}\g\Mod$ to $\Gamma\g\Mod$ which preserve coproducts and projectives. The admissible categories $\underline{A}\g\Mod$ and $\Gamma\g\Mod$ have as subcategories of compact objects $\underline{A}\g\mod$ and $\Gamma\g\mod$, respectively, and the natural transformation $\sigma:\Gamma\otimes_{\underline{A}}-\rightmap{}HL_{\underline{\cal A}}$ restricts to an isomorphism of the corresponding restricted functors to the subcategories of compact objects, as proved in \cite[(12.5)]{hsb}. Then, our statement follows from (\ref{P: fiel y pleno en compactos implica fiel y pleno}).  
\end{proof}

\begin{proposition}\label{L: tilde(I) cerrada bajo extensiones} The category $\widetilde{\cal I}$ is closed under extensions and direct summands.
\end{proposition}

\begin{proof} Consider $M\in \underline{\cal A}\g\Mod$ and a conflation $\xi: L\rightmap{f}P\rightmap{g}M$ in ${\cal E}$ where $H(P)$ is ${\cal E}$-projective. So, the sequence $$H(\xi):0\rightmap{}H(L)\rightmap{H(f)}H(P)\rightmap{H(g)}H(M)\rightmap{}0$$ is exact with $H(P)$ a projective $\Gamma$-module. Then, for every $N\in \underline{\cal A}\g\Mod$, we have the following diagram with exact rows 
$$\begin{matrix} 
\Hom_{\underline{\cal A}}(P,N)&\rightmap{f^*}&\Hom_{\underline{\cal A}}(L,N)&\rightmap{c}&\Ext_{\cal E}(M,N)&\rightarrow 0\\
\shortlmapdown{H}&&\shortlmapdown{H}&&\shortlmapdown{\widehat{H}}\\
\Hom_{\Gamma}(H(P),H(N))&\rightmap{H(f)^*}&\Hom_{\Gamma}(H(L),H(N))&\rightmap{\gamma}&\Ext_{\Gamma}(H(M),H(N))&\rightarrow 0\\
\end{matrix}$$
where $\widehat{H}$ is induced by the exact functor $H$, and the recipes of $c$ and $\gamma$ are given by $c(h)=[h\xi]$ and  $\gamma(h')=[h'H(\xi)]$, for $h\in \Hom_{\underline{\cal A}}(L,N)$ and $h'\in \Hom_\Gamma(H(L),H(N))$,  respectively. Here, $[h\xi]$ denotes the class of the conflation obtained  as the pushout of  $\xi$ using the morphism $h$. 
Since $H$ is exact, it preserves pushouts along inflations, so $\gamma(H(h))=[H(h)H(\xi)]=[H(h\xi)]=\widehat{H}([h\xi])$. Hence, the diagram is commutative and, since the first two vertical arrows are isomorphisms, so is the third one. This implies that $\widetilde{\cal I}$ is closed under extensions. 

The remaining part of the proof follows from the fact that, since idempotents split in $\underline{\cal A}\g\Mod$ which is equivalent to $\widetilde{I}$, then idempotents split in $\widetilde{\cal I}$. Indeed, take $M\in \widetilde{\cal I}$. 
We may assume that $M=H(M')$. Take a direct summand $N$ of $M$, consider the inclusion $i:N\rightmap{}M$, the projection $p:M\rightmap{}N$, and the idempotent $e=ip:M\rightmap{}M$, which satisfy $pi=id_N$. Then, there is an idempotent $e':M'\rightmap{}M'$ in $\underline{\cal A}\g\Mod$ with $H(e')=e$. Since $\underline{\cal A}$ is a Roiter interlaced weak ditalgebra, idempotents split in $\underline{\cal A}\g\Mod$, see \cite[(11.4)]{hsb}, so there is an $\underline{\cal A}$-module $N'$ and morphisms $q:M'\rightmap{}N'$ and $j:N'\rightmap{}M'$, such that $jq=e'$ and $qj=id_{N'}$ in $\underline{\cal A}\g\Mod$. The images $H(q):M\rightmap{}H(N')$ and $H(j):H(N')\rightmap{}M$ satisfy that
$H(j)H(q)=H(e')=e=ip$ and $H(q)H(j)=id_{H(N')}$. This implies that $H(N')\cong N$, indeed we have the isomorphism $pH(j):H(N')\rightmap{}N$ with inverse $H(q)i:N\rightmap{}H(N')$.
\end{proof}

For the proof of the following result, the use of the right algebra $\Gamma$ associated to the special triangular interlaced weak ditalgebra $\underline{\cal A}$ and the exact equivalence $H:\underline{\cal A}\g\mod\rightmap{}{\cal I}$,  where ${\cal I}$  denotes the full subcategory of $\Gamma\g\mod$ of modules induced from $\underline{A}\g\mod$, is crucial. The first ideas in this direction can be traced back to \cite{BK} and \cite{BuBt}, see also  \cite{Bu}.

\begin{theorem}\label{T: A-mod tiene sqcsd}
 If $\underline{\cal A}$ is a special triangular interlaced weak ditalgebra, then it is a Roiter interlaced weak ditalgebra and the category $\underline{\cal A}\g\mod$ has a natural exact structure ${\cal E}_{\underline{\cal A}}$, see \cite[(11.11)]{hsb}. The corresponding exact category $(\underline{\cal A}\g\mod,{\cal E}_{\underline{\cal A}})$ has almost split conflations.
\end{theorem}

\begin{proof} The proof of \cite[(7.18)]{BSZ} can be adapted to this situation, using the equivalence $H:\underline{\cal A}\g\mod\rightmap{}{\cal I}$. 
\end{proof}

The preceding statement and the following theorem were already used in \cite{bpsqh}, but we prefer to state them here explicitely, with a brief indication for their proof. 
Their formulation evolved from a well known result due to Crawley-Boevey in  \cite{CB1} for tame finite-dimensional algebras, which was revisited first in \cite{BSZ} for tame finite-dimensional seminested ditalgebras and then in \cite{bpsqh} for our present context.

\begin{lemma}\label{L: dim tau(M) leq c x dim M en A-mod}
Let $\underline{\cal A}$ be a special triangular interlaced weak ditalgebra.
  For any non 
${\cal E}_{\underline{\cal A}}$-projective indecomposable object $M\in \underline{\cal A}\g\mod$, there is an almost split conflation $N\rightmap{}E\rightmap{}M$. Thus, 
 the indecomposable object  $N$ is uniquely determined up to isomorphism. We call $N$ the \emph{translate} of $M$ and denote it by $\tau(M)$. There is a number $c_{\underline{\cal A}}$ such that, for any non ${\cal E}_{\underline{\cal A}}$-projective indecomposable $M\in \underline{\cal A}\g\mod$, the following  inequality holds:
$$\dim_k\tau(M)\leq c_{\underline{\cal A}}\dim_kM.$$
\end{lemma}

\begin{proof} The proof of \cite[(32.5)]{BSZ} can be adapted to this situation. See the proof of  \cite[(1.4)]{bpsqh}, in the admissible homological system case, Steps 1 and 2.
\end{proof}

\begin{remark}\label{R: almost  homogeneous}
Given an exact Krull-Schmidt category ${\cal C}$, recall that an indecomposable object $M\in {\cal C}$ is called \emph{homogeneus} if it admits an almost split conflation of the form $M\rightmap{}E\rightmap{}M$. 

The following lemma gives a useful property related to the preservation of homogeneity by certain type of  functors 
\end{remark}

\begin{lemma}\label{L: on homogeneity}
Let $F:{\cal C}'\rightmap{}{\cal C}$ be a full and faithful functor between exact Krull-Schmidt categories.  Assume that   ${\cal U}$ is a full replete subcategory of  ${\cal C}$ which is closed under direct summands. Moreover, suppose that,
 for every indecomposable object $M\in {\cal U}$ there is some $M'\in {\cal C}'$ such that $F(M')\cong M$.  Let $M \in {\cal U}$ be an indecomposable such that an almost split conflation of ${\cal C}$ ending at $M$ lies in ${\cal U}$ and 
 $F(M')\cong M$, for some homogeneous $M'\in {\cal C}'$.  Then, the object $M$ of ${\cal C}$ is also homogeneous.
\end{lemma}

\begin{proof}  Let $\zeta:N\rightmap{f}E\rightmap{g}M$  be an almost split conflation in ${\cal C}$ with $N,E\in {\cal U}$. Since ${\cal C}$ is a Krull-Schmidt category and $F$ is full and faithful, our assumptions imply that there is a sequence of morphisms $N'\rightmap{f'}E'\rightmap{g'}M'$ in ${\cal C}'$ which is mapped by the functor $F$ onto a conflation $$\widetilde{\zeta}:F(N')\rightmap{F(f')}F(E')\rightmap{F(g')}F(M')$$ of ${\cal C}$ isomorphic to $\zeta$. Since $g$ is a minimal right almost split morphism and $F$ is full and faithful, we obtain that $g':E'\rightmap{}M'$ is a right minimal almost split morphism. Since $M'$ is homogeneous, there is an  almost split conflation in ${\cal C}'$ of the form
$$\zeta':M'\rightmap{f''}E'\rightmap{g'}M'.$$ 
Hence, we have $F(g')F(f'')=0$.  
Since $\widetilde{\zeta}$ is an exact pair, there is some morphism $s:M'\rightmap{}N'$ in ${\cal C}'$ such that $F(f')F(s)=F(f'')$. So $f''=f's$. But $f''$ is irreducible and $f'$ is not a retraction, so $s$ is an isomorphism. Hence, 
$N\cong F(N')\cong F(M')\cong M$, and we are done. 
\end{proof}

\begin{remark}\label{R: F exact} Given a minimal ditalgebra ${\cal B}$, consider the class $\widehat{\cal E}_{\cal B}$ of composable pairs of morphisms $N\rightmap{f}E\rightmap{g}M$ in ${\cal B}\g\Mod$ 
such that $gf=0$ and 
$0\rightmap{}N\rightmap{f^0}E\rightmap{g^0}M\rightmap{}0$
is an exact sequence in $B\g\Mod$. It is known that $\widehat{\cal E}_{\cal B}$ is an exact structure on ${\cal B}\g\Mod$. We will consider ${\cal B}\g\Mod$ as an exact category equipped with this special exact structure. Moreover, we know that ${\cal B}\g\mod$ has almost split conflations, which are given by the image under the embedding functor $L_{\cal B}:B\g\Mod\rightmap{}{\cal B}\g\Mod$ of the well known almost split sequences of $B\g\mod$. See \cite[(32.3)]{BSZ}

As mentioned in (\ref{T: A-mod tiene sqcsd}), for a special triangular interlaced  weak ditalgebra $\underline{\cal A}$, the category $\underline{\cal A}\g\mod$ has a natural exact structure ${\cal E}_{\underline{\cal A}}$ with almost split conflations.

Notice that, with the specified exact structures on ${\cal B}\g\Mod$ and $\underline{\cal A}\g\Mod$, any reduction functor $F:{\cal B}\g\mod\rightmap{}\underline{\cal A}\g\mod$   as in (\ref{T: pregens de underline(A)}) is exact. 
Indeed, this can be proved using the argument of the proof of \cite[(32.4)]{BSZ} adapted to the present situation. 
This argument uses \cite[(6.13)]{bpsqh} and the fact that the underlying algebra of the layer of $\underline{\cal A}$ is semisimple. 
\end{remark}

\begin{theorem}\label{T: casi todo M en A-mod tq tau(M) cong M} Assume that 
  $\underline{\cal A}=({\cal A},I)$ is a ${\cal P}$-oriented  triangular interlaced weak ditalgebra, as in   (\ref{D: biquiver P-orientado})\&(\ref{D: triangular interlaced weak ditalg}), where $I\subseteq \rad A$.  Suppose that $\underline{\cal A}$ is not wild. 
  Thus,  $(\underline{\cal A}\g\mod,{\cal E}_{\underline{\cal A}})$  has almost split conflations and a translation $\tau$. Then, for every  $d\in \hueca{N}$, almost every  $d$-dimensional indecomposable module $M\in \underline{\cal A}\g\mod$ satisfies $\tau(M)\cong M$.
\end{theorem}

\begin{proof} Consider the full and faithful reduction functor $F:{\cal B}\g\mod\rightmap{}\underline{\cal A}\g\mod$ given by an application of (\ref{T: parametrización de módulos de dim acotada}) to $\underline{\cal A}$, using the number $d' := (1 + c_{\underline{\cal A}}) d$. Consider also the full subcategory ${\cal U}$ of $\underline{\cal A}\g\mod$ defined by the $\underline{\cal A}$-modules $N$ with $\dim_kN\leq d'$. From (\ref{R: F exact}), we know that almost every indecomposable object in ${\cal B}\g\mod$ is homogeneous. Given a  non-projective indecomposable $M\in \underline{\cal A}\g\mod$ with $\dim_kM\leq d$, from (\ref{L: dim tau(M) leq c x dim M en A-mod}), we know that there is an almost split conflation ending at $M$, which in fact lies in ${\cal U}$. The functor $F$ is such that every $N\in {\cal U}$ is of the form $N\cong F(N')$ for some $N'\in {\cal B}\g\mod$. Having in mind (\ref{R: F exact}), after eliminating all the possible $M$ with $\dim_kM\leq d$ with non-homogeneous $M'\in {\cal B}\g\mod$ with $F(M')\cong M$, we remain with almost every indecomposable $\underline{\cal A}$-module $M$ with $\dim_k M\leq d$. To any such $M$,  we can apply (\ref{L: on homogeneity}) and obtain that it is  homogeneous. 
\end{proof} 

Observe that the argument used in the preceding proof is essentially the same used in \cite[(32.6)]{BSZ}, which in turn is Crawley-Boevey's argument in \cite[Thm D]{CB1}. It underlies the argument of Step 3 in the proof of Theorem (1.4) in the admissible homological system case in  \cite{bpsqh}.

\section{${\cal P}$-presentations and admissible equivalences}\label{Coverage theorems ...}

In this section, we consider ${\cal P}$-oriented triangular interlaced weak ditalgebras $\underline{\cal A}=({\cal A},I)$ such that $I\subseteq \rad(A)^2$, where ${\cal P}$ is a finite preordered set.  As in  \cite[\S13]{hsb}, we will call such a  weak ditalgebra $\underline{\cal A}$, a 
\emph{${\cal P}$-strict interlaced weak ditalgebra}. Thus $\underline{\cal A}$ is a special triangular interlaced weak ditalgebra and \S5 applies to it. If, moreover, $\underline{\cal A}$ is not wild,  it satisfies the hypothesis of (\ref{T: pregens de underline(A)}). 

An example of this type of  interlaced weak ditalgebra is the ${\cal P}$-oriented interlaced weak ditalgebra $\underline{\cal A}=\underline{\cal A}(\Delta)$ associated to an admissible homological system $({\cal P},\leq,\{\Delta_v\}_{v\in {\cal P}})$ for a  finite-dimensional algebra $\Lambda$, see for instance \cite[(5.22) \& (13.2)]{hsb}. 

The fact that $\underline{\cal A}=({\cal A},I)$ is a ${\cal P}$-oriented interlaced weak ditalgebra entails certain properties on the algebra $\underline{A}=A/I$. Notably, it is a finite-dimensional basic $k$-algebra such that every indecomposable projective $\underline{A}$-module $P$ is a \emph{brick}, that is such that $\End_{\underline{A}}(P)\cong k$.  
 
 \begin{remark}\label{R: Sauter} Let $({\cal C},{\cal E})$ and $({\cal C}',{\cal E}')$ be exact categories. An \emph{exact functor}   
$\mathfrak{G}:({\cal C},{\cal E})\rightmap{}({\cal C}',{\cal E}')$ is a functor $\mathfrak{G}:{\cal C}\rightmap{}{\cal C}'$  that maps conflations of ${\cal E}$ on conflations of ${\cal E}'$.  We say that an exact functor  $\mathfrak{G}:({\cal C},{\cal E})\rightmap{}({\cal C}',{\cal E}')$ is an \emph{exact equivalence} if $\mathfrak{G}:{\cal C}\rightmap{}{\cal C}'$ is an equivalence of categories. We will say that $\mathfrak{G}$ is \emph{an equivalence of exact categories} iff $\mathfrak{G}:({\cal C},{\cal E})\rightmap{}({\cal C}',{\cal E}')$ is an exact equivalence and so is its quasi-inverse $\mathfrak{G}':({\cal C}',{\cal E}')\rightmap{}({\cal C},{\cal E})$. 
  
 J. Sauter has shown in 
 \cite[(2.8)]{Sauter} that an exact equivalence $\mathfrak{G}:({\cal C},{\cal E})\rightmap{}({\cal C}',{\cal E}')$ is an equivalence of exact categories iff the induced morphism 
 $$\Ext^1_{\cal E}(M,N)\rightmap{}\Ext^1_{\cal E'}(\mathfrak{G}(M),\mathfrak{G}(N))$$
 is an isomorphism for any $M,N\in {\cal C}$.   
\end{remark}

\begin{definition}\label{D: presentations of subcats} Let $\Sigma$ be a 
finite-dimensional $k$-algebra and ${\cal C}$ a full and replete subcategory of $\Sigma\g\Mod$ that is closed under arbitrary coproducts and extensions.  Then, a \emph{${\cal P}$-presentation of} ${\cal C}$ is an equivalence of categories   $\frak{G}:\underline{\cal A}\g\Mod\rightmap{}{\cal C}$, where ${\cal P}$ is a finite preordered set and 
 $\underline{\cal A}= ({\cal A}, I)$ is a ${\cal P}$-strict  interlaced weak ditalgebra. Furthermore, we require that the functor $\mathfrak{G}$  satisfies the following properties:
\begin{enumerate}
\item $\frak{G}$ is an equivalence of exact categories, where $\underline{\cal A}\g\Mod$ is endowed with its canonical exact structure ${\cal E}$ and ${\cal C}$ has the exact structure ${\cal E}'$  inherited from the usual  exact structure of $\Sigma\g\Mod$.  
\item The composition of functors $\underline{A}\g\Mod\rightmap{L_{\underline{\cal A}}}\underline{\cal A}\g\Mod\rightmap{\frak{G}}{\cal C}$ is isomorphic to the functor 
$\mathfrak{Z}\otimes_{\underline{A}}-$, for some finite-dimensional $\Sigma\g\underline{A}$-bimodule $\mathfrak{Z}$. 
 
\item Every generic $\Sigma$-module $M$ that belongs to ${\cal C}$ is of the form $M\cong \mathfrak{G}(N)$, for some pregeneric object $N\in \underline{\cal A}\g\Mod$. 
\end{enumerate}  
\end{definition}

In the following, unless we explicitely indicate the contrary, when we handle a ${\cal P}$-presentation of a category ${\cal C}$ as above, we adopt the preceding notations. 
  
\begin{remark}\label{R: general S-filtrations} Given an exact category $({\cal C},{\cal E})$ and a set of objects ${\cal S}$ of ${\cal C}$, we say that an object $M\in {\cal C}$ has an \emph{${\cal S}$-filtration} iff there is a sequence of inflations 
$$M_1\rightmap{s_1}M_2\rightmap{s_2}M_3\rightmap{}\cdots\rightmap{}M_{l-1}\rightmap{s_{l-1}}M_l=M$$
such that $M_1\cong \coprod_{i\in I}S_i$ with all $S_i\in {\cal S}$ and, for each $j\in [1,l-1]$, $\Coker(s_j)\cong \coprod_{i\in I_j}S_i$ with all $S_i\in{\cal S}$. 

Observe that if ${\cal D}$ is a full and replete  subcategory of ${\cal C}$ that is closed under coproducts and ${\cal E}$-extensions, and such that ${\cal S}\subseteq {\cal D}$, then any object $M\in {\cal C}$ with an ${\cal S}$-filtration belongs to ${\cal D}$. This follows by an easy induction. 

Notice also that whenever $G:({\cal C},{\cal E})\rightmap{}({\cal C}',{\cal E}')$ is an exact functor that preserves coproducts and $M\in {\cal C}$ admits an ${\cal S}$-filtration, then $G(M)$ admits a $G({\cal S})$-filtration. 
\end{remark}

\begin{proposition}\label{P: structure of presented subcats}
Let $\mathfrak{G}:\underline{\cal A}\g\Mod\rightmap{}{\cal C}$  be a ${\cal P}$-presentation for a subcategory ${\cal C}$ of $\Sigma\g\Mod$, as in (\ref{D: presentations of subcats}), and  let $\{S_v\}_{v\in {\cal P}}$ be a
complete family of representatives of the isoclasses of the simple $\underline{A}$-modules. Define $\Delta'_v:=\mathfrak{Z}\otimes_{\underline{A}}S_v$, for $v\in {\cal P}$, and $\Delta'=\{\Delta'_v\}_{v\in {\cal P}}$.  Then, we have:
\begin{enumerate}
\item  $({\cal P},\leq,\{\Delta'_v\}_{v\in {\cal P}})$ is   an  homological system for $\Sigma$ 
\item ${\cal C}=\widetilde{\cal F}(\Delta')$. Thus, by (\ref{R: CalF(Theta) es admisible}), ${\cal C}$ is an admissible exact subcategory of $\Sigma\g\Mod$. 
\end{enumerate}
\end{proposition}

\begin{proof} (1): 
We claim that  $\Ext^1_\Sigma(\mathfrak{Z}\otimes_{\underline{A}}M,\mathfrak{Z}\otimes_{\underline{A}}N)$ is an epimorphic image of $\Ext^1_{\underline{A}}(M,N)$, for any $\underline{A}$-modules $M$ and $N$. Indeed,  
consider the composition of induced morphisms by  $L_{\underline{\cal A}}$ y $\mathfrak{G}$, respectively,
$$\Ext^1_{\underline{A}}(M,N)\rightmap{}\Ext^1_{\underline{\cal A}}(M,N)\rightmap{}\Ext^1_{\cal C}(\mathfrak{Z}\otimes_{\underline{A}} M,\mathfrak{Z}\otimes_{\underline{A}} N).$$
The second one is an isomorphism by (\ref{R: Sauter}). The first one is well defined because $L_{\underline{\cal A}}:\underline{A}\g\Mod\rightmap{}\underline{\cal A}\g\Mod$ is exact, and is  surjective because $\underline{\cal A}$ is a Roiter interlaced weak ditalgebra and we can apply \cite[(11.5)]{hsb}, which means that any conflation $\zeta$ of $\underline{\cal A}\g\Mod$ is equivalent to one of the form $L_{\underline{\cal A}}(\zeta')$, for some conflation $\zeta'$ in $\underline{A}\g\Mod$. So, the composition is surjective and we can proceed, as in the proof of \cite[(13.7)]{hsb}, to prove (1).  
\medskip

(2): 
 Assume first that  $M\in {\cal C}$. Then,  we may suppose that  $M=\mathfrak{Z}\otimes_{\underline{A}}M'$, for some $\underline{A}$-module $M'$. Since $\underline{A}$ is finite-dimensional, we have a filtration of $M'$ of the form 
 $$0 = M'_\ell \subseteq  M'_{\ell-1} \subseteq \cdots \subseteq M'_1\subseteq M'_0=M',$$
with $M'_i=J^iM'$, for $i\in [0,\ell]$,  where $J$ denotes the Jacobson radical of the algebra $\underline{A}$. Since  each $\underline{A}/J$-module $M'_i/M'_{i+1}$ is semisimple, we see that $M'$ has an ${\cal S}$-filtration, where ${\cal S}$ denotes the set 
$\{S_v\}_{v\in {\cal P}}$ of simple $\underline{A}$-modules. Since $\mathfrak{G}L_{\underline{\cal A}}\cong\mathfrak{Z}\otimes_{\underline{A}}-$ is an exact functor and  $\Delta'_v=\mathfrak{Z}\otimes_{\underline{A}}S_v$, for $v\in {\cal P}$,  from the second observation in (\ref{R: general S-filtrations}),  we immediately obtain that $M$ admits a $\Delta'$-filtration. So, $M\in \widetilde{\cal F}(\Delta')$. 

  The proof  of the other inclusion follows from the first observation of (\ref{R: general S-filtrations}) applied to the subcategory ${\cal D}={\cal C}$ of $\Sigma\g\Mod$, because $\Delta'\subseteq {\cal C}$. 
\end{proof}

\begin{remark}\label{El caso cal(A)=cal(A)(Delta)} In the following, we only consider interlaced weak ditalgebras $\underline{\cal A}$ that are ${\cal P}$-strict interlaced weak ditalgebras. We denote by $\Gamma$  the right algebra of $\underline{\cal A}$ and consider the family of $\Gamma$-modules $\{\Delta'_v=\Gamma\otimes_{\underline{A}}S_v\}_{v\in {\cal P}}$. From \cite[(13.7)]{hsb}, we know that 
$({\cal P},\leq,\{\Delta'_v\}_{v\in {\cal P}})$ is an admissible homological system for $\Gamma$. Thus, we have the full subcategory $\widetilde{\cal F}(\Delta')$ of $\Gamma\g\Mod$. 
\end{remark}

\begin{proposition}\label{P: H(N) gen implies N gen}
 Consider the  functor $H:\underline{\cal A}\g\Mod\rightmap{}\Gamma\g\Mod$, as in (\ref{L: el funtor F fiel, pleno y exacto}), where $\Gamma$ is the right algebra of $\underline{\cal A}$. If $N\in \underline{\cal A}\g\Mod$ is such that $H(N)$  is a generic $\Gamma$-module, then $N$ is a pregeneric $\underline{\cal A}$-module with $\Endol{N}\leq \Endol{H(N)}$. 
\end{proposition}

\begin{proof} Set  $E:=\End_{\underline{\cal A}}(N)^{op}$.  So, $N$ is a right $E$-module with action  $ne=e^0(n)$, for $e\in E$ and $n\in N$. Then $M:=H(N)$ is a right $E$-module by restriction through the isomorphism $H:\End_{\underline{\cal A}}(N)^{op}\rightmap{}\End_\Gamma(M)^{op}$. This action is given by $he=e\circ h$, for $e\in E$ and $h\in \Hom_{\underline{\cal A}} (\underline{A},N)$. It is determined by  the canonical action of $\End_\Gamma(M)^{op}$ on $M$, so $\Endol{M}=\ell_E(M)$, which is finite by assumption.  

 Since $M$ is generic, the algebra $E$ is local with nilpotent radical $\frak{r}$, see \cite[(4.2 and 4.4)]{CB3}.   
Assume that $\frak{r}^l=0$ and define $N_i:=N\frak{r}^i$, for all $i$. Then, we have the filtration of $N$ by right $E$-modules 
$$0=N_l\subset N_{l-1}\subset\cdots\subset N_1\subset N_0=N,$$
such that each factor $N_i/N_{i+1}$ is a right $D$-module with $D:=E/\frak{r}$. 
 
  We will show that $\ell_E(N)$ is finite. Since $\Endol{N}=\ell_E(N)$, we will get from this that $N$ is a  pregeneric $\underline{\cal A}$-module. Define $M_i:=\{h\in \Hom_{\underline{\cal A}}(\underline{A},N)\mid h^0(1)\in N_i\}$, for all $i$. Then, we have the following filtration of $M$ by right $E$-modules 
$$0=M_l\subseteq M_{l-1}\subseteq\cdots\subseteq M_1\subseteq M_0=M,$$
which satisfy $M_i\frak{r}\subset M_{i+1}$. Since $\ell_E(M)$ is finite, each factor $M_i/M_{i+1}$ is a $D$-vector space with finite dimension $d_i$.  It will be enough to  show that $\dim_DN_i/N_{i+1}\leq d_i$, for each $i$. 

For a fixed $i$, given $n_1,\ldots,n_{d_i+1}\in N_i$, we consider the morphisms of $\underline{A}$-modules $h^0_s\in \Hom_{\underline{A}}(\underline{A},N)$ such that $h^0_s(1)=n_s$, for $s\in [1,d_i+1]$. Then, we have $h_s:=(h^0_s,0)\in \Hom_{\underline{\cal A}}(\underline{A},N)$. In fact, we have that every $h_s\in M_i$. Thus, there are elements in $e_1,\ldots,e_{d_i+1}\in E$ which are not all zero modulo $\frak{r}$, such 
$$\sum_{s=1}^{d_i+1}e_s\circ h_s=\sum_{s=1}^{d_i+1}h_se_s\in M_{i+1}.$$ 
Then, evaluating at 1 the first components of these morphisms, we have 
$$\sum_{s=1}^{d_i+1}n_se_s=\sum_{s=1}^{d_i+1}e^0_s(n_s)=\sum_{s=1}^{d_i+1}e^0_sh_s^0(1)\in N_{i+1}.$$
So, the elements $n_1,\ldots,,n_{d_i+1}$ are $D$-linearly dependent modulo $N_{i+1}$. From this our result follows. 
\end{proof}

\begin{proposition}\label{L: I tilde  presentada por cal(A)}\label{P: filtrados e inducidos}  Let $\underline{\cal A}$ be  any ${\cal P}$-strict interlaced weak ditalgebra and $\Gamma$  the right algebra of $\underline{\cal A}$. Consider the family $\Delta'=\{\Delta'_v\}_{v\in {\cal P}}$ of the $\Gamma$-modules $\Delta'_v=\Gamma\otimes_{\underline{A}}S_v$, for $v\in {\cal P}$. From \cite[(13.7)]{hsb}, we know that 
$({\cal P},\leq,\{\Delta'_v\}_{v\in {\cal P}})$ is an admissible homological system for $\Gamma$. With the notation of (\ref{L: F es Hom(A,L)}), the subcategory $\widetilde{\cal I}$ of $\Gamma\g\Mod$ admits the  ${\cal P}$-presentation $H:\underline{\cal A}\g\Mod\rightmap{}\widetilde{\cal I}$,   with $\widetilde{\cal F}(\Delta')=\widetilde{\cal I}$. 
\end{proposition}

\begin{proof} By (\ref{L: tilde(I) cerrada bajo extensiones}), $\widetilde{\cal I}$ is closed under extensions and direct summands. It is clearly closed under coproducts. By (\ref{L: F es Hom(A,L)}), we know that $\Gamma\otimes_{\underline{A}}-\cong HL_{\underline{\cal A}}$ and $H:\underline{\cal A}\g\Mod\rightmap{}\widetilde{\cal I}$ is an exact equivalence. From  \cite[(12.11)]{hsb}, the composition of the following morphisms, induced by $L_{\underline{\cal A}}$ and $H$, is surjective
$$\Ext^1_{\underline{A}}(M,N)\rightmap{}\Ext^1_{\underline{\cal A}}(M,N)\rightmap{}\Ext^1_{\cal C}(\Gamma\otimes_{\underline{A}} M,\Gamma\otimes_{\underline{A}} N).$$
Thus the second one is surjective. But it  was already injective, because $H$ is full and faithful. This implies by (\ref{R: Sauter}) that $H$ is an equivalence of exact categories.

    By (\ref{P: H(N) gen implies N gen}), we have  condition (3) in the definition of ${\cal P}$-presentation. The equality $\widetilde{\cal I}=\widetilde{\cal F}(\Delta')$ follows from (\ref{P: structure of presented subcats})(2). 
\end{proof}

\begin{definition}\label{D: admissible equivalences} 
 Assume that $({\cal C},{\cal E})$ and $({\cal C}',{\cal E}')$ are admissible exact subcategories of $\Sigma\g\Mod$ and $\Sigma'\g\Mod$, respectively, for some finite-dimensional algebras $\Sigma$ and $\Sigma'$. Then, an \emph{admissible equivalence} 
 $\mathfrak{U} : ({\cal C},{\cal E}) \rightmap{}({\cal C}',{\cal E}')$ is an equivalence of exact categories satisfying the following conditions: 
\begin{enumerate}
\item There is  a finite-dimensional $\Sigma'\g\Sigma$-bimodule $\mathfrak{X}$  such that the functor $\mathfrak{X}\otimes_{\Sigma}-:\Sigma\g\Mod\rightmap{}\Sigma'\g\Mod$ restricts to a functor ${\cal C}\rightmap{}{\cal C}'$ isomorphic to $\mathfrak{U}$.
\item  There is some $c\in \hueca{N}$ such that, for every $M \in  {\cal C}$, we have:  $$\Endol{M}\leq c\times \Endol{\mathfrak{U}(M)} \hbox{ \ and \ } 
\dim_kM\leq c\times \dim_k\mathfrak{U}(M).$$ 
\end{enumerate}
\end{definition}

With the preceding notation, we have: 

\begin{lemma}\label{L: endols y equivalencias admisibles} 
Assume that 
$\mathfrak{U} : ({\cal C},{\cal E}) \rightmap{}({\cal C}',{\cal E}')$ 
is an admissible equivalence and let 
 $M\in {\cal C}$. Then, 
$\Endol{\mathfrak{U}(M)}\leq \dim_k\mathfrak{X}\times \Endol{M}$. Moreover,  $M$ is a generic $\Sigma$-module in ${\cal C}$ iff $\mathfrak{U}(M)$ is a generic $\Sigma'$-module in ${\cal C}'$. 
\end{lemma}

\begin{proof} We have that $\Endol{\mathfrak{U}(M)}=\Endol{\mathfrak{X}\otimes_{\Sigma}M}\leq \dim_k \mathfrak{X}\times \Endol{M}$. By the second condition in the definition of admissible equivalence, we know that $M$ is infinite-dimensional iff $\mathfrak{U}(M)$ is so. So, if $M\in {\cal C}$ is a generic $\Sigma$-module, we get that $\mathfrak{U}(M)$ is a generic $\Sigma'$-module. Conversely, if $\mathfrak{U}(M)$ is a generic $\Sigma'$-module, for some $M\in {\cal C}$,  since $\mathfrak{U}$ is an admissible equivalence, we get $\Endol{M}\leq c\times \Endol{\mathfrak{U}(M)}$.  Hence,  $M$ is a  generic $\Sigma$-module. 
\end{proof}

\begin{remark}\label{R: Omega es equiv admis}  Let $\Omega:\Gamma\g\Mod\rightmap{}\Lambda\g\Mod$ be an equivalence of categories, where $\Gamma$ and $\Lambda$ are finite-dimensional algebras. Consider an  admissible exact subcategory ${\cal C}$ of $\Gamma\g\Mod$ as in (\ref{D: admissible exact cats}), and let ${\cal C}'$ be the minimal full and replete subcategory of $\Lambda\g\Mod$ containing all the objects of the form $\Omega(M)$ with $M\in {\cal C}$. Let ${\cal E}$ and ${\cal E}'$ be the inherited exact structures on ${\cal C}$ and ${\cal C}'$ from $\Gamma\g\Mod$ and $\Lambda\g\Mod$, respectively. Then, from Morita Theorem, $\Omega$  restricts to an equivalence of exact categories $\Omega_\vert:({\cal C},{\cal E})\rightmap{}({\cal C}',{\cal E}')$. This equivalence of exact categories  preserves coproducts and compact objects.  Moreover, every ${\cal E}'$-projective $M'\in {\cal C}'$ is of the form $\Omega(M)\cong M'$, for some ${\cal E}$-projective $M\in {\cal C}$. Since such an $M$ is a coproduct of compact objects, so is $\Omega(M)$ (and so is $M')$. Thus ${\cal C}'$ is an  admissible exact subcategory of $\Lambda\g\Mod$. Using again Morita Theorem, one shows that $\Omega_\vert:({\cal C},{\cal E})\rightmap{}({\cal C}',{\cal E}')$ is an admissible equivalence. 
\end{remark}

\begin{lemma}\label{L: P-present compos con admis equiv es P-present} 
Let $\Sigma$ and $\Sigma'$ be finite-dimensional algebras. 
Assume that an admissible exact subcategory $({\cal C}, {\cal E})$ of $\Sigma\g\Mod$ admits a ${\cal P}$-presentation $\mathfrak{G}:\underline{\cal A}\g\Mod\rightmap{}{\cal C}$ and that $\mathfrak{U}:({\cal C},{\cal E})\rightmap{}({\cal C}',{\cal E}')$ is an admissible equivalence onto the  exact admissible subcategory ${\cal C}'$ of $\Sigma'\g\Mod$.  Then, the composition $\mathfrak{U}\mathfrak{G}:\underline{\cal A}\g\Mod\rightmap{}{\cal C}'$ is a ${\cal P}$-presentation of ${\cal C}'$. 
\end{lemma}

\begin{proof} Clearly, any composition of equivalences of exact categories is an equivalence of exact categories. So $\mathfrak{U}\mathfrak{G}:\underline{\cal A}\g\Mod\rightmap{}{\cal C}'$ is an equivalence of exact categories. By assumption, $\mathfrak{G}L_{\underline{\cal A}}\cong \mathfrak{Z}\otimes_{\underline{A}}-$, for some finite-dimensional $\Sigma\g\underline{A}$-bimodule $\mathfrak{Z}$ and $\mathfrak{U}\cong \mathfrak{X}\otimes_{\Sigma}-$, for some finite-dimensional $\Sigma'\g \Sigma$-bimodule $\mathfrak{X}$. Therefore, 
$$\mathfrak{U}\mathfrak{G}L_{\underline{\cal A}}\cong \mathfrak{X}\otimes_{\Sigma}(\mathfrak{Z}\otimes_{\underline{A}}-)\cong (\mathfrak{X}\otimes_{\Sigma}\mathfrak{Z})\otimes_{\underline{A}}-$$
where $\mathfrak{X}\otimes_{\Sigma}\mathfrak{Z}$ is a finite-dimensional $\Sigma'\g \underline{A}$-bimodule.  

Suppose that $M'$ is generic $\Sigma'$-module in ${\cal C}'$. Then, $M'\cong \mathfrak{U}(M)$, for some $M\in {\cal C}$, which is a generic $\Sigma$-module in ${\cal C}$ by (\ref{L: endols y equivalencias admisibles}). 
Then, $M\cong \mathfrak{G}(N)$ for some 
pregeneric $N\in \underline{\cal A}\g\Mod$. Thus, $\mathfrak{U}\mathfrak{G}(N)\cong M$. 
\end{proof}

\begin{definition}\label{D: normal generic module} A pregeneric $\underline{\cal A}$-module $M$ will be called \emph{a normal pregeneric} $\underline{\cal A}$-module if it has a structure of an 
$\underline{A}\g k(x)$-bimodule such that
$$\End_{\underline{\cal A}}(M)^{op}=D\oplus {\cal R},$$
where $D$ is a subalgebra, ${\cal R}=\rad\End_{\underline{\cal A}}(M)^{op}$, and there is an isomorphism of algebras $k(x)\rightmap{}D$ which maps any $r\in k(x)$ on 
$L_{\underline{\cal A}}(\mu_r)$. Here,  $\mu_r\in \End_{\underline{A}}(M)^{op}$ denotes the right multiplication by $r$.
\end{definition}

We already know that,  if $\underline{\cal A}$ is not wild, any pregeneric $\underline{\cal A}$-module is isomorphic to a normal one, see (\ref{T: pregens de underline(A)}). 

\begin{proposition}\label{P: endols y pregen normales}\label{L: endol M leq endol P otimes M} Let $\mathfrak{G}:\underline{\cal A}\g\Mod\rightmap{}{\cal C}$ be a ${\cal P}$-presentation of an admissible exact subcategory ${\cal C}$ of $\Sigma\g\Mod$, with $\mathfrak{G}L_{\underline{\cal A}}\cong \mathfrak{Z}\otimes_{\underline{A}}-$. Here, again, $\Delta'_u=\mathfrak{Z}\otimes_{\underline{A}}S_u$, for $u\in {\cal P}$. If   $M\in \underline{\cal A}\g\Mod$ is a normal pregeneric module,  we have:
\begin{enumerate}
\item $\Endol{M}=\dim_{k(x)}M$ \hbox{ \ } and \hbox{ \ } 
$\Endol{
\mathfrak{Z}\otimes_{\underline{A}}M
}=
\dim_{k(x)}(\mathfrak{Z}\otimes_{\underline{A}}M)$; 
\item $\dim_{k(x)}(\mathfrak{Z}\otimes_{\underline{A}}M)=\sum_{u\in {\cal P}}\dim_{k(x)}e_uM\times\dim_k\Delta'_u$.
\end{enumerate} 
Hence,  
$\Endol{M}\leq   
  \Endol{\mathfrak{G}(M)}$. Moreover, 
for $N\in \underline{\cal A}\g\mod$, we have 
\begin{enumerate}
\item[{\it 3.}]$\dim_k\mathfrak{G}(N)=\dim_k(\mathfrak{Z}\otimes_{\underline{A}}N)=\sum_{u\in {\cal P}}\dim_ke_uN\times\dim_k\Delta'_u.$
\end{enumerate}
Hence, $\dim_kN\leq \dim_k\mathfrak{G}(N)$. 
\end{proposition}

\begin{proof} (1): The first equality in (1) follows immediately from the given decomposition of  $\End_{\underline{\cal A}}(M)^{op}$ in (\ref{D: normal generic module}). Now, consider the isomorphism of algebras 
$$\End_{\underline{\cal A}}(M)^{op}\rightmap{}\End_\Sigma(\mathfrak{Z}\otimes_{\underline{A}}M)^{op}$$ given by the composition of $\mathfrak{G}$ with the conjugation $c_\sigma$ associated the natural isomorphism $\sigma:\mathfrak{Z}\otimes_{\underline{A}}M\rightmap{}\mathfrak{G}L_{\underline{\cal A}}(M)$. Then, we have the decomposition
$$\End_{\Sigma}(\mathfrak{Z}\otimes_{\underline{A}}M)^{op}=c_\sigma \mathfrak{G}(D)\oplus c_\sigma \mathfrak{G}({\cal R}),$$
where $c_\sigma \mathfrak{G}({\cal R})=\rad\End_\Sigma(\mathfrak{Z}\otimes_{\underline{A}} M)^{op}$ and we have an isomorphism of algebras $k(x)\rightmap{}c_\sigma \mathfrak{G}(D)$ mapping each $r\in k(x)$ onto $c_\sigma \mathfrak{G} L_{\underline{\cal A}}(\mu_r)=id\otimes \mu_r$.  The second equality of (1) follows from this. 

 We know that the functor $\mathfrak{Z}\otimes_{\underline{A}}-:\underline{A}\g\Mod\rightmap{}\Sigma\g\Mod$ is isomorphic to $\mathfrak{G}L_{\underline{\cal A}}:\underline{A}\g\Mod\rightmap{}\Sigma\g\Mod$, which is exact. So, the first functor is also exact. Consider the $k(x)$-algebra $\underline{A}^{k(x)}:=\underline{A}\otimes_kk(x)$. Since   the $\underline{A}\g k(x)$-bimodules can  be considered as $\underline{A}^{k(x)}$-modules, the preceding tensor product functor 
 restricts to an exact functor $\mathfrak{Z}\otimes_{\underline{A}}-:\underline{A}^{k(x)}\g\Mod\rightmap{}\Sigma^{k(x)}\g\Mod$. Moreover, there is an isomorphism  between the latter and the functor  $\mathfrak{Z}^{k(x)}\otimes_{\underline{A}^{k(x)}}-:\underline{A}^{k(x)}\g\Mod\rightmap{}\Sigma^{k(x)}\g\Mod$.

 Since $k(x):k$ is a MacLane separable  field extension, $\rad(\underline{A}^{k(x)})=\rad(\underline{A})\otimes_kk(x)$, see  
\cite[(3.3)]{Ka} and \cite{Baut-Yingbo}, so the simple $\underline{A}^{k(x)}$-modules
have the form  $S_u\otimes_k k(x)$, where $\{S_u\}_{u\in {\cal P}}$ are the  simple $\underline{A}$-modules.  

Since the $\underline{A}^{k(x)}$-module $M$ has finite $k(x)$-dimension, it has a finite-composition series. This series is transformed by the exact functor $\mathfrak{Z}^{k(x)}\otimes_{\underline{A}^{k(x)}}-$ onto a $\Delta'_{k(x)}$-filtration of $\mathfrak{Z}^{k(x)}\otimes_{\underline{A}^{k(x)}}M\cong \mathfrak{Z}\otimes_{\underline{A}}M$, where  
$\Delta'_{k(x)}=\{\Delta'_u\otimes_kk(x)\}_{u\in {\cal P}}$. 

 As we remarked before, the indecomposable projective $\underline{A}$-modules are bricks, hence the indecomposable projective modules over the $k(x)$-algebra $\underline{A}^{k(x)}$  are bricks too. Then, 
the multiplicity of the simple $\underline{A}^{k(x)}$-module $S_u\otimes_kk(x)$ in the given composition series for $M$ is
$ \dim_{k(x)} e_uM$. So,  the multiplicity of the $\Sigma^{k(x)}$-module $\Delta'_u\otimes_kk(x)$ in the considered tensored filtration for $\mathfrak{Z}^{k(x)}\otimes_{\underline{A}^{k(x)}}M$ is the same number. Hence,
$$\dim_{k(x)}\mathfrak{Z}^{k(x)}\otimes_{\underline{A}^{k(x)}}M=\sum_u\dim_{k(x)}e_uM\times \dim_{k(x)}(\Delta'_u\otimes_kk(x)).$$ 
From this and (1), we get (2). As a consequence, we have 
$$\Endol{M}=
  \sum_u\dim_{k(x)}e_uM
  \leq
  \sum_u\dim_{k(x)}e_uM\times \dim_k\Delta'_u=
  \Endol{\mathfrak{G}(M)}.
$$
Similarly, working with the exact functor 
$\mathfrak{Z}\otimes_{\underline{A}}-:\underline{A}\g\mod\rightmap{}\Sigma\g\mod$ and $N\in \underline{\cal A}\g\mod$, instead of $\mathfrak{Z}^{k(x)}\otimes_{\underline{A}^{k(x)}}-$ and $M$, we obtain the last statement of our Proposition.  
\end{proof}

\begin{corollary}\label{C: cal(A) pregen tame-> cal(F)(Delta') gen tame} Let $\mathfrak{G}:\underline{\cal A}\g\Mod\rightmap{}{\cal C}$ be a ${\cal P}$-presentation of an admissible exact subcategory ${\cal C}$ of $\Sigma\g\Mod$. Assume that $\underline{\cal A}$ is pregenerically tame, then ${\cal C}=\widetilde{\cal F}(\Delta')$ and ${\cal F}(\Delta')$ is generically tame. 
\end{corollary}

\begin{proof} Let $d\in \hueca{N}$ and $M\in {\cal C}$ a generic $\Sigma$-module with $\Endol{M}\leq d$. Since $\mathfrak{G}$ is a ${\cal P}$-presentation, we get $M\cong \mathfrak{G}(N)$ for some pregeneric $N\in \underline{\cal A}\g\Mod$.  From (\ref{P: endols y pregen normales}), we have 
$\Endol{N}\leq \Endol{\mathfrak{G}(N)}=\Endol{M}\leq d$. By assumption, $\underline{\cal A}\g\Mod$ admits only finitely many non-isomorphic  pregeneric $\underline{\cal A}$-modules $N$ with $\Endol{N}\leq d$, so $\widetilde{\cal F}(\Delta')$ admits only 
finitely many non-isomorphic  generic 
$\Sigma$-modules $M$ with endolength bounded by $d$. 
\end{proof}

\begin{remark}\label{R: wildness characterizations} Assume that $\Lambda$ is a finite-dimensional algebra and ${\cal C}$ a full subcategory of $\Lambda\g\mod$, which is closed under direct summands and direct sums. Then,  ${\cal C}$ is wild, as in \cite[(1.2)]{bpsqh}, if and only if 
there is a $\Lambda\g k\langle x,y\rangle$-bimodule  $C$ such that 
$C_{k\langle x,y\rangle}$ is finitely generated and the functor 
 $C\otimes_{k\langle x,y\rangle}-  : k\langle x, y\rangle\g\mod\rightmap{}{\cal C} $ preserves indecomposables and isomorphism classes. 
Indeed, the equivalence of these statements  is  well known  for the particular case ${\cal C}=\Lambda\g\mod$, see \cite{D} and \cite{CB1}. For a general ${\cal C}$ the proof is the same. 
\end{remark}

\begin{proposition}\label{R: underline(A) wild --> F(Delta) wild} Assume that $\mathfrak{G}:\underline{\cal A}\g\Mod\rightmap{}{\cal C}$ is a ${\cal P}$-presentation for an  admissible exact  subcategory ${\cal C}$ of $\Sigma\g\Mod$. With the notation of  (\ref{P: structure of presented subcats}), we have ${\cal C}=\widetilde{\cal F}(\Delta')$. Then, 
 $\underline{\cal A}$ is wild iff ${\cal F}(\Delta')$ is wild. 
 \end{proposition}
 
 \begin{proof} Assume that $\underline{\cal A}$ is wild.  
 Consider a composition functor 
$$k\langle x,y \rangle\g\Mod\rightmap{ \ \ Z\otimes_{k\langle x,y \rangle}- \ \ }\underline{A}
 \g\Mod\rightmap{ \ \ L_{\underline{\cal A}} \ \  }\underline{\cal A}\g\Mod$$
 that preserves isoclasses of indecomposables as in \cite[(6.12)]{bpsqh}. Its restriction to finite-dimensional modules,  composed with the equivalence $\mathfrak{G}:\underline{\cal A}\g\mod\rightmap{}{\cal F}(\Delta')$, determines a functor 
$$\mathfrak{G}L_{\underline{\cal A}}(Z\otimes_{k\langle x,y \rangle}-):k\langle x,y \rangle\g\mod\rightmap{}{\cal F}(\Delta')$$
which preserves isoclasses and indecomposables, see \cite[(1.2)]{bpsqh}. Moreover,   we have  
$\mathfrak{G}L_{\underline{\cal A}}(Z\otimes_{k\langle x,y \rangle}-)\cong (\mathfrak{Z}\otimes_{\underline{A}} Z)\otimes_{k\langle x,y\rangle}-$, 
where 
$\mathfrak{Z}\otimes_{\underline{A}} Z$ is a finitely generated  $k\langle x,y\rangle$-module. From 
(\ref{R: wildness characterizations}), we obtain that  ${\cal F}(\Delta')$ is wild. 

In order to prove the converse, by contradiction, assume that ${\cal F}(\Delta')$ is wild but $\underline{\cal A}$ is not wild. 
 From \cite[(10.4)]{bpsqh}, we know that $\underline{\cal A}$ is strictly tame, as in \cite[(10.3)]{bpsqh}. This implies that ${\cal F}(\Delta')$ is stricty tame, see the argument in the second part of the proof of \cite[(1.3)]{bpsqh}. This contradicts \cite[(1.3)]{bpsqh}. 

\end{proof}

\begin{lemma}\label{P: F(Delta') wild --> F(Delta') not pregen tame} Assume that $\mathfrak{G}:\underline{\cal A}\g\Mod\rightmap{}{\cal C}$ is a 
 ${\cal P}$-presentation of an  exact admissible subcategory ${\cal C}$ of $\Sigma\g\Mod$. As in (\ref{P: structure of presented subcats}), we have ${\cal C}=\widetilde{\cal F}(\Delta')$. If $\underline{\cal A}$ is wild, then ${\cal F}(\Delta')$ is not generically tame. 
\end{lemma}

\begin{proof} By assumption, there is an $\underline{A}\g k\langle x,y\rangle$-bimodule $Z$, which is free finitely generated as  $k\langle x,y\rangle$-module,
such that the composition of functors
$$k\langle x,y\rangle\g\Mod\rightmap{ \ \ Z\otimes_{k\langle x,y\rangle}- \ \ }\underline{A}\g\Mod\rightmap{L_{\underline{\cal A}}}\underline{\cal A}\g\Mod$$
preserves isoclasses of indecomposables.
We also know that the composition  
$$\underline{A}\g\Mod\rightmap{L_{\underline{\cal A}}}\underline{\cal A}\g\Mod\rightmap{\mathfrak{G}}\widetilde{\cal F}(\Delta')$$
 is isomorphic to $\mathfrak{Z}\otimes_{\underline{A}}-$, for a finite-dmensional $\Sigma\g \underline{A}$-bimodule $\mathfrak{Z}$. Then, the functor 
$$W:=\mathfrak{Z}\otimes_{\underline{A}}Z\otimes_{k\langle x,y\rangle}-:
k\langle x,y\rangle\g\Mod\rightmap{}\widetilde{\cal F}(\Delta'),$$
 preserves isoclasses of  indecomposables, see \cite[(6.12)]{bpsqh}. As in 
\cite[(31.3) and (31.5)]{BSZ}, we  consider the $k\langle x,y\rangle$-modules $H_\lambda=k(x)[y]/(y-\lambda)$, for $\lambda\in k$, where $k\langle x,y\rangle$ acts by restriction through the epimorphism of algebras
$$k\langle x,y\rangle\rightmap{}k[x,y]\rightmap{}k[x,y]_{k[x]^*}\cong k(x)[y].$$
The family $\{H_\lambda\}_{\lambda\in k}$ is an infinite family of pairwise non-isomorphic pregeneric $k\langle x,y\rangle$-modules with endolength 1. Set $G_\lambda:=  W(H_\lambda)$, for $\lambda\in k$. So $\{G_\lambda\}_{\lambda\in k}$, is an infinite family of pairwise non-isomorphic indecomposable  $\Sigma$-modules in $\widetilde{\cal F}(\Delta')$. The $\Sigma\g k\langle x,y\rangle$-bimodule $\mathfrak{Z}\otimes_{\underline{A}}Z$ is a finitely generated right $k\langle x,y\rangle$-module. So there is an epimorphism of $k\langle x,y\rangle$-modules $k\langle x, y\rangle^c\rightmap{}\mathfrak{Z}\otimes_{\underline{A}}Z$, which determines an epimorphism  $(k\langle x,y\rangle \otimes_{k\langle x,y\rangle}H_\lambda)^c\rightmap{}W(H_\lambda)$ of right  $k(x)$-modules. Therefore, 
$\Endol{W(H_\lambda)}\leq \dim_{k(x)}W(H_\lambda)\leq c$. Thus, $\{G_\lambda\}_{\lambda\in k}$ is an infinite family of indecomposables with endolength bounded by $c$, they all are infinite-dimensional over $k$ because they are non-trivial 
 $k(x)$-vector spaces. So ${\cal F}(\Delta')$ is not generically tame.
\end{proof}

\begin{corollary}\label{C: F(Delta') pregen tame iff cal(A) pregen tame} With the preceding notations, $\underline{\cal A}$ is pregenerically tame iff ${\cal F}(\Delta')$ is generically tame. 
\end{corollary}

\begin{proof} If $\underline{\cal A}$ is pregenerically tame, by (\ref{C: cal(A) pregen tame-> cal(F)(Delta') gen tame}), ${\cal F}(\Delta')$ is generically tame. Conversely, if ${\cal F}(\Delta')$ is generically tame, by (\ref{P: F(Delta') wild --> F(Delta') not pregen tame}),  $\underline{\cal A}$ is not wild. Then,  by (\ref{T: pregens de underline(A)}), we get that $\underline{\cal A}$ is pregenerically tame.  
\end{proof}

\begin{proposition}\label{P: F(Delta) wild --> F(Delta) not pregen tame} Assume that $\mathfrak{G}:\underline{\cal A}\g\Mod\rightmap{}{\cal C}$ is a 
 ${\cal P}$-presentation of an  exact admissible subcategory ${\cal C}$ of $\Sigma\g\Mod$. As in (\ref{P: structure of presented subcats}), we have ${\cal C}=\widetilde{\cal F}(\Delta')$. 
The following statements are equivalent:
\begin{enumerate}
\item $\underline{\cal A}$ is wild.
\item ${\cal F}(\Delta')$ is wild.
\item ${\cal F}(\Delta')$ is not generically tame.
\end{enumerate}
In particular, ${\cal F}(\Delta')$ is tame iff ${\cal F}(\Delta')$ is generically tame. 
\end{proposition}

\begin{proof} The equivalence of the first two items is just (\ref{R: underline(A) wild --> F(Delta) wild}).  

If $\underline{\cal A}$ is wild, by (\ref{P: F(Delta') wild --> F(Delta') not pregen tame}), ${\cal F}(\Delta')$ is not generically tame. So $1$ implies $3$. Finally, assume that ${\cal F}(\Delta')$ is not generically tame. From (\ref{C: F(Delta') pregen tame iff cal(A) pregen tame}), we get that $\underline{\cal A}$ is not pregenerically tame. This implies that $\underline{\cal A}$ is wild, because if this was not the case, by (\ref{T: pregens de underline(A)}), it would be pregenerically tame. So $3$ implies $1$.  
\end{proof}

\begin{corollary}\label{C: filtered subcats simultaneously wild} 
Assume that the exact categories $({\cal C},{\cal E})$ of $\Sigma\g\Mod$, and $({\cal C}',{\cal E}')$  of $\Sigma'\g\Mod$  admit ${\cal P}$-presentations $\mathfrak{G}:\underline{\cal A}\g\Mod\rightmap{}{\cal C}$ and 
 $\mathfrak{G}':\underline{\cal A}\g\Mod\rightmap{}{\cal C}'$, respectively. 
 Denote by ${\cal F}(\Delta')$ and ${\cal F}(\Delta'')$ the subcategories of ${\cal C}$ and ${\cal C}'$ determined by the ${\cal P}$-presentations $\mathfrak{G}$ and $\mathfrak{G}'$, respectively. Then,  ${\cal F}(\Delta')$ is  wild iff ${\cal F}(\Delta'')$ is wild. 
\end{corollary}

\section{Covering $\Theta$-filtered modules}\label{S: covering Theta filtered modules}

Now, we proceed to prove 
an adaptation (\ref{T: el teorema de MSX}) of a theorem due to Mendoza, S\'aenz, and Xi, see \cite[\S3]{MSX} and \cite[\S11]{bpsqh}. This is an important step in the proof of our main results. 

\begin{remark}\label{R: TildeF(Delta) cerr bajo sumandos dir}
In this section,
$({\cal P},\leq,\{\Theta_v\}_{v\in {\cal P}})$ denotes a fixed  general homological system for a finite-dimensional algebra $\Upsilon$. It is well known that the subcategory ${\cal F}(\Theta)$ of $\Upsilon\g\mod$ is closed under direct summands, see \cite{MSX}. A short proof of this statement was given in \cite{P}. The argument used in \cite{P} can be adapted to show that the subcategory $\widetilde{\cal F}(\Theta)$ of $\Upsilon\g\Mod$ is closed under direct summands, see \cite{Tesis:Luis Fernando Tzec}. 
As a consequence of this, the indecomposable objects in $\widetilde{\cal F}(\Theta)$ are indecomposable $\Upsilon$-modules. Moreover, from \cite{Tesis:Luis Fernando Tzec}, we also know that $\widetilde{\cal F}(\Theta)$ is closed under coproducts. 
\end{remark}

The following statement is crucial in the proof of (\ref{T: el teorema de MSX}), its dual is verified in the first part of the proof of
\cite[(3.12)]{MSX}.

\begin{lemma}\label{L: de las sucesiones especiales}
For each $ v \in {\cal P}$, there is an exact sequence $$0\rightmap{}V_v\rightmap{}U_v\rightmap{}\Theta_v\rightmap{}0$$ such that $U_v$ is an indecomposable
${\cal F}(\Theta)$-projective and $V_v\in {\cal F}(\Theta)$ has a $\Theta$-filtration with factors of the form $\Theta_{u}$ with $u>v$.
\end{lemma}

We fix a family of special exact sequences, as provided by the last lemma, for the rest of this section.

\begin{lemma}\label{L: suficientes proyectivos en F(Delta)}
For each $M\in \widetilde{\cal F}(\Theta)$, there is an exact sequence
$$0\rightmap{}K\rightmap{}W\rightmap{}M\rightmap{}0$$
in $\widetilde{\cal F}(\Theta)$,
such that $W$ is a direct sum of copies of the modules in $\{U_v\mid v\in {\cal P}\}$.

The family $\{U_v\}_{v\in {\cal P}}$ is a complete set of representatives of the isomorphism classes of the indecomposable $\widetilde{\cal F}(\Theta)$-projective modules.
\end{lemma}

\begin{proof} Since each $U_v$ is an ${\cal F}(\Theta)$-projective module, we know that $\Ext^1_\Upsilon(U_v,\Theta_u)=0$, for all $u\in {\cal P}$.
Since $U_v$ is a finite-dimensional $\Upsilon$-module, the functor
$\Ext^1_\Upsilon(U_v,-)$ commutes with direct sums, and we get that
$\Ext^1_\Upsilon(U_v,V)=0$
for any direct sum $V$ of copies of the modules in $\{\Theta_v\mid v\in {\cal P}\}$. A simple induction argument shows that $\Ext^1_\Upsilon(U_v,M)=0$, for any $M\in \widetilde{\cal F}(\Theta)$. So,  any such $U_v$ is an $\widetilde{\cal F}(\Theta)$-projective.

Notice that $\widetilde{\cal F}(\Theta)$ is  closed under (possibly infinite) direct sums of copies of any finite family of objects in $\widetilde{\cal F}(\Theta)$. Thus, 
if $W=\bigoplus_{j\in J}W_j$, where each $W_j$ is isomorphic to some $U_v$ with $v\in {\cal P}$,  for instance from \cite[(5.2.53)]{Rowen}, we have $\Ext^1_\Upsilon(W,M)=
\prod_j
\Ext^1_\Upsilon(W_j,M)=0$. So $W$ is an $\widetilde{\cal F}(\Theta)$-projective module. Notice also that $\widetilde{\cal F}(\Theta)$ is closed under extensions. 
Now, we can easily adapt the argument in the proof of \cite[(11.2)]{bpsqh} to show the existence of an exact sequence as described in the statement of this lemma, for each $M\in \widetilde{\cal F}(\Theta)$.

If $M$ is any indecomposable $\widetilde{\cal F}(\Theta)$-projective, we have an exact sequence as we have just mentioned, so $M$ is an indecomposable  direct summand of $W$. As a consequence of
Crawley-J\o nsson-Warfield and Azumaya Theorems, see \cite[(12.6; 26.5; and 26.6)]{AF}, we obtain that $M\cong U_v$, for some $v\in {\cal P}$.

 Finally, the fact that the family $\{U_v\}_{v\in {\cal P}}$ consists of pairwise non-isomorphic $\Upsilon$-modules is proved in \cite[(3.7)]{MSX}.
\end{proof}

\begin{lemma}\label{R: CalF(Theta) es admisible} Given any general homological system  $({\cal P},\leq, \{\Theta_v\}_{v\in {\cal P}})$, the  category $\widetilde{\cal F}(\Theta)$ is an admissible exact subcategory of $\Upsilon\g\Mod$. 
\end{lemma}

\begin{proof} If ${\cal E}'$ denotes the exact structure of $\widetilde{\cal F}(\Theta)$ inherited from the standard exact structure of $\Upsilon\g\Mod$, then, from (\ref{L: suficientes proyectivos en F(Delta)}), we know that $\widetilde{\cal F}(\Theta)$ has enough ${\cal E}'$-projectives and every indecomposable ${\cal E}'$-projective object of $\widetilde{\cal F}(\Theta)$ is a direct sum of finite-dimensional ones (thus of compact objects, by (\ref{L: compacts in cal(F)(Delta)})).  From (\ref{R: TildeF(Delta) cerr bajo sumandos dir}), we know that $\widetilde{\cal F}(\Theta)$ is closed under coproducts and direct sums.
\end{proof}

\begin{definition}\label{D: Gamma y Theta}
Define $U:=\bigoplus_{v\in {\cal P}}U_v$ and $\Lambda:=\End_\Upsilon(U)^{op}$. Moreover, consider the family
$\{\Delta_v\}_{v\in {\cal P}}$ of $\Lambda$-modules given by $\Delta_v:=\Hom_\Upsilon(U,\Theta_v)$, for $v\in {\cal P}$.
\end{definition}

It was shown in \cite{MSX} that $({\cal P},\leq,\{\Delta_v\}_{v\in {\cal P}})$ is an admissible homological system for $\Lambda$ such that ${\cal F}(\Theta)$ is equivalent to ${\cal F}(\Delta)$ as exact categories, see also \cite[\S11]{bpsqh}.
In the following, we extend a little their results to show that  $\widetilde{\cal F}(\Theta)$ is equivalent to $\widetilde{\cal F}(\Delta)$ as exact categories, see \cite[(10.20)]{Buhler}.
The quasi-inverse equivalences will be realized by the appropriate  restrictions of the  functors $\mathfrak{H}:=\Hom_\Upsilon(U,-):\Upsilon\g\Mod\rightmap{}\Lambda\g\Mod$ and $\mathfrak{T}:=U\otimes_\Lambda-:\Lambda\g\Mod\rightmap{}\Upsilon\g\Mod$.
We start with the following.

\begin{lemma}\label{L: Hom(U,-) restringe a fiel-pleno-exact de F(Delta)} The functor $\mathfrak{H}$ restricts to a full and faithful exact functor
$$\mathfrak{H}:\widetilde{\cal F}(\Theta)\rightmap{}\Lambda\g\Mod.$$
\end{lemma}

\begin{proof}  The functor $\mathfrak{H}:\widetilde{\cal F}(\Theta)\rightmap{}\Lambda\g\Mod$ is exact because $U$ is  $\widetilde{\cal F}(\Theta)$-projective.  It preserves coproducts because $U$ is finite-dimensional; so, by (\ref{L: suficientes proyectivos en F(Delta)}), it also  preserves projectives.  We know, from (\ref{R: CalF(Theta) es admisible}), that  $\widetilde{\cal F}(\Theta)$ is an admissible exact subcategory of $\Upsilon\g\Mod$   and its  category of compact objects is ${\cal F}(\Theta)$ by (\ref{L: compacts in cal(F)(Delta)}). Moreover, the restriction $\frak{H}_\vert:{\cal F}(\Theta)\rightmap{}\Lambda\g\mod$ is full and faithful by \cite[(11.4)]{bpsqh}. 

Consider the full subcategory $\add(U)$ of
$\Upsilon\g\mod$  formed by the direct summands of finite direct sums of $U$. Using 
\cite[(II.2.1)]{ARS}, we know that
 $\mathfrak{H}$ restricts to an equivalence of categories $\add(U)\rightmap{}\add(\Lambda)$. Since    $\mathfrak{H}$ commutes with direct sums, we have that any projective $\Lambda$-module $P$ is of the form $\frak{H}(Q)\cong P$, for some $\widetilde{\cal F}(\Theta)$-projective $Q$. Thus,  our statement follows from (\ref{P: fiel y pleno en compactos implica fiel y pleno}).  
\end{proof}

\begin{theorem}\label{T: el teorema de MSX}
Given any homological system $({{\cal P}},\leq,\{\Theta_v\}_{v\in {\cal P}})$ for $\Upsilon$,
consider the algebra $\Lambda=\End_\Upsilon(U)^{op}$ and the $\Lambda$-modules $\Delta_v=\Hom_\Upsilon(U,\Theta_v)$, for all $v\in {\cal P}$, as before. 
Then, $({{\cal P}},\leq,\{\Delta_v\}_{v\in {\cal P}})$ is an admissible homological system for $\Lambda$ and  the functors $\mathfrak{H}=\Hom_\Upsilon(U,-)$ and $\mathfrak{T}=U\otimes_\Lambda-$ induce   quasi-inverse equivalences of  exact categories between $\widetilde{\cal F}(\Theta)$ and $\widetilde{\cal F}(\Delta)$.
\end{theorem}

\begin{proof} We already know  that $({{\cal P}},\leq,\{\Delta_v\}_{v\in {\cal P}})$ is an   admissible homological system, see 
\cite[(11.5)]{bpsqh}. Consider the  natural transformations
 $\alpha:\mathfrak{T}\mathfrak{H}\rightmap{}id_{\Upsilon\g\Mod}$
and $\beta:id_{\Lambda\g\Mod}\rightmap{}\mathfrak{H}\mathfrak{T}$ such that 
$\alpha_M:U\otimes_\Lambda \mathfrak{H}(M)\rightmap{}M$ has the recipe  $\alpha_M(u\otimes g)=g(u)$,  for $M\in \Upsilon\g\Mod$; and 
 $\beta_N:N\rightmap{}\mathfrak{H}(U\otimes_\Lambda N)$ has the recipe  $\beta_N(n)[u]=u\otimes n$, for $N\in \Lambda\g\Mod$.

We will show that we have quasi-inverse exact equivalences
 $\mathfrak{H}:\widetilde{\cal F}(\Theta)\rightmap{} \widetilde{\cal F}(\Delta)$  and $\mathfrak{T}: \widetilde{\cal F}(\Delta)\rightmap{}\widetilde{\cal F}(\Theta)$. For this we will show that 
 we have isomorphisms of functors $\alpha:\mathfrak{T}\mathfrak{H}\rightmap{}id_{\widetilde{\cal F}(\Theta)}$ and $\beta:id_{\widetilde{\cal F}(\Delta)}\rightmap{}\mathfrak{H}\mathfrak{T}$.
  In \cite[(11.5)]{bpsqh}, we proved that the restriction  $\alpha_\vert:\mathfrak{T}\mathfrak{H}_\vert \rightmap{}id_{{\cal F}(\Theta)}$ and $\beta_\vert: id_{{\cal F}(\Delta)}\rightmap{}\mathfrak{H}\mathfrak{T}_\vert$ are  natural isomorphisms.    

The argument in the proof of \cite[(11.5) Step 2]{bpsqh} shows that $\Tor_1^{\Lambda}(U,\Delta_u)=0$, for any $v\in {\cal P}$. 
 Since $\Tor_1^\Lambda(U,-)$ commutes with direct sums, see \cite[(5.2.50)]{Rowen}, we get $\Tor^{\Lambda}_1(U,V)=0$, for any direct sum $V$ of objects in $\Delta$.    A simple induction argument shows that $\Tor^{\Lambda}_1(U,N)=0$, for any $N\in \widetilde{\cal F}(\Delta)$. As a consequence of this, we get that  the functor 
 $U\otimes_\Lambda-:\widetilde{\cal F}(\Delta)\rightmap{}\Upsilon\g\Mod$ is  exact. From (\ref{R: general S-filtrations}), we know that  this exact functor restricts to an exact functor  $\frak{T}=U\otimes_\Lambda-:\widetilde{\cal F}(\Delta)\rightmap{}\widetilde{\cal F}(\frak{T}(\Delta))$. But, we have that $\frak{T}(\Delta_u)\cong \frak{T}\frak{H}(\Theta_u)\cong \Theta_u$, as proved in \cite[(11.5)]{bpsqh}, so $\widetilde{\cal F}(\frak{T}(\Delta))=\widetilde{\cal F}(\Theta)$.  
 
From (\ref{R: CalF(Theta) es admisible}), we know that $\widetilde{\cal F}(\Delta)$ and $\widetilde{\cal F}(\Theta)$ are admissible exact subcategories  and their subcategories of compact objects are ${\cal F}(\Delta)$ and ${\cal F}(\Theta)$ by (\ref{L: compacts in cal(F)(Delta)}). 
 Then, we can apply (\ref{P: fiel y pleno en compactos implica fiel y pleno}) to $\alpha$ and $\beta$ to obtain that $\alpha:\mathfrak{T}\mathfrak{H}\rightmap{}id_{\widetilde{\cal F}(\Theta)}$ and $\beta:id_{\widetilde{\cal F}(\Delta)}\rightmap{}\mathfrak{H}\mathfrak{T}$ are isomorphisms. 
\end{proof}

\begin{lemma}\label{L: endolongitudes van bien con H}
 With the preceding notation, for any $N\in \widetilde{\cal F}(\Theta)$, we have
 $$\Endol{\mathfrak{H}(N)}\leq \dim_kU\times \Endol{N}$$
 and, for any $M\in \widetilde{\cal F}(\Delta)$, we have
 $\Endol{U\otimes_\Lambda M}\leq \dim_kU\times \Endol{M}$.
 Therefore, the functors $\mathfrak{H}:\widetilde{\cal F}(\Theta)\rightmap{}\widetilde{\cal F}(\Delta)$ and  $\mathfrak{T}:\widetilde{\cal F}(\Delta)\rightmap{}\widetilde{\cal F}(\Theta)$ preserve generic modules.
\end{lemma}

\begin{proof} Given  $N\in \widetilde{\cal F}(\Theta)$, set  $E:=\End_{\Lambda}(\mathfrak{H}(N))^{op}$, and consider the $\Upsilon\g E$-bimodule
 structure on $N$, where $E$ acts by the right through the algebra isomorphism
 $\alpha_N:=\mathfrak{H}^{-1}:E\rightmap{}\End_{\Upsilon}(N)^{op}$.

If $b=\dim_k U$, there is an epimorphism $\Upsilon^b\rightmap{}U$ of left $\Upsilon$-modules. So, there is a monomorphism
$\mathfrak{H}(N)=\Hom_\Upsilon(U,N)\rightmap{}\Hom_\Upsilon(\Upsilon^b,N)\cong \Hom_\Upsilon(\Upsilon,N)^b\cong N^b$
 of right $E$-modules. Hence,
 $$\Endol{\mathfrak{H}(N)}=\ell_E(\mathfrak{H}(N))\leq b\times \ell_E(N)=b\times \Endol{N}.$$
 The other claim in the statement of this lemma is similar (and well known).
\end{proof}

\begin{remark}\label{R: frak(T) es equiv admis} 
 With the notation of (\ref{T: el teorema de MSX}), the functor $\mathfrak{T}:\widetilde{\cal F}(\Delta)\rightmap{}\widetilde{\cal F}(\Theta)$ is an admissible equivalence as in (\ref{D: admissible exact cats}). Indeed, by (\ref{R: CalF(Theta) es admisible}), $\widetilde{\cal F}(\Delta)$ and $\widetilde{\cal F}(\Theta)$ are admissible exact subcategories of $\Lambda\g\Mod$
 and $\Upsilon\g\Mod$, respectively. The condition $(2)$ in the definition of admissible equivalence for $\mathfrak{T}$ follows from (\ref{L: endolongitudes van bien con H}). 
\end{remark}

\section{Generic modules for ${\cal F}(\Theta)$}

This section is devoted to the proofs of the main results (\ref{T:propiedades de gens de F(Delta)}) and (\ref{T: main thm}). We keep the notation of the preceding section. So, we assume that $\Upsilon$ is a finite-dimensional $k$-algebra  and $({\cal P},\leq,\{\Theta_v\}_{v\in {\cal P}})$ is any homological system for $\Upsilon$, and we will develop the proofs with this notation.

\begin{remark}\label{R: context for proof of thm (1.5)} 
Consider the finite-dimensional algebra $\Lambda=\End_\Upsilon(U)^{op}$,  the admissible homological system ${\cal H}:=({\cal P},\leq,\{\Delta_v\}_{v\in {\cal P}})$ for $\Lambda$ as in (\ref{D: Gamma y Theta}), 
and the strict ${\cal P}$-oriented interlaced weak ditalgebra 
  $\underline{\cal A}=\underline{\cal A}(\Delta)$ associated to ${\cal H}$, see \cite[(5.22)\&(13.2)]{hsb}. 

 We have the  functor   $U\otimes_{\Lambda}- : \Lambda\g\Mod \rightmap{} \Upsilon\g\Mod$, which restricts to the admissible equivalence $\mathfrak{T}:\widetilde{\cal F}(\Delta)\rightmap{}\widetilde{\cal F}(\Theta)$ such that  
$\mathfrak{T}(\Delta_u)\cong  \Theta_u$ 
for each $u\in {\cal P}$, see (\ref{R: frak(T) es equiv admis}). From (\ref{P: filtrados e inducidos}), we have the homological system $\Delta'=\{\Delta'_u\}_{u\in {\cal P}}$ for the right algebra $\Gamma$ of $\underline{\cal A}$, and the ${\cal P}$-presentation  
$H : \underline{\cal A}\g \Mod \rightmap{} \widetilde{\cal F} (\Delta')$. Here, 
 $\Delta'_u=\Gamma\otimes_{\underline{A}}S_u$, for $u\in {\cal P}$.  
From \cite[(13.9)]{hsb}, we know the existence of an equivalence of categories $ \Omega : \Gamma\g\Mod \rightmap{} \Lambda\g\Mod$ such that  $\Omega(\Delta'_u) \cong \Delta_u$, for each $u\in {\cal P}$. Let $P$ be a progenerator in $\mod\g\Gamma$ such that $\Omega\cong P\otimes_{\Gamma}-$. 

From  (\ref{R: Omega es equiv admis}) and (\ref{R: frak(T) es equiv admis}), the exact admissible subcategory $\widetilde{\cal F}(\Theta)$ of $\Upsilon\g\Mod$ admits the ${\cal P}$-presentation  $\mathfrak{G}:=\mathfrak{T}\Omega_\vert H : \underline{\cal A}\g\Mod \rightmap{} \widetilde{\cal F}(\Theta)$, where  $\mathfrak{G}L_{\underline{\cal A}}\cong \mathfrak{Z}\otimes_{\underline{A}}-$, with $\mathfrak{Z}=U\otimes_{\Lambda}P\otimes_\Gamma\Gamma$ and $\mathfrak{Z}\otimes_{\underline{A}}S_u\cong \Theta_u$, for all $u\in {\cal P}$. 

 Theorem (\ref{T:propiedades de gens de F(Delta)}) is a particular case of the following one, when applied to the preceding ${\cal P}$-presentation of $\widetilde{\cal F}(\Theta)$, after  some trivial notational variation. 
 
 Notice that, from (\ref{C: filtered subcats simultaneously wild}), we get that $\widetilde{\cal F}(\Theta)$, $\widetilde{\cal F}(\Delta)$, and $\widetilde{\cal F}(\Delta')$ are simultaneously wild (resp. tame or  generically tame).  
\end{remark}

\begin{theorem}\label{T: (1.3) gral} Consider a fixed finite-dimensional algebra $\Upsilon$ and a general homological system $({\cal P},\leq, \{\Theta_v\}_{v\in {\cal P}})$ for $\Upsilon$. Suppose that $\mathfrak{G} : \underline{\cal A}\g\Mod \rightmap{} \widetilde{\cal F}(\Theta)$ is a ${\cal P}$-presentation for $\widetilde{\cal F}(\Theta)$ such that $\mathfrak{G}(S_v)\cong \Theta_v$, for $v\in {\cal P}$.  
Then, if  ${\cal F}(\Theta)$ is not wild and  $G$ is a generic $\Upsilon$-module for  ${\cal F}(\Theta)$, the following statements hold:
\begin{enumerate}
 \item There is a rational algebra $\Gamma_G$ and a $\Upsilon\g \Gamma_G$-bimodule $Z_G$ such that as a right $\Gamma_G$-module $Z_G$ is free of rank equal to the endolength of $G$. If $Q_G$ is the field of fractions of $\Gamma_G$, then $G\cong Z_G\otimes_{\Gamma_G}Q_G$.
 \item The functor $Z_G\otimes_{\Gamma_G}-:\Gamma_G\g\Mod\rightmap{}\widetilde{\cal F}(\Theta)$ preserves isoclasses and indecomposability.
 \item  The functor $Z_G\otimes_{\Gamma_G}-:\Gamma_G\g\mod\rightmap{}{\cal F}(\Theta)$ preserves Auslander-Reiten sequences. 
\end{enumerate}  
\end{theorem}

\begin{proof} Since ${\cal F}(\Theta)$ is not wild, from (\ref{R: underline(A) wild --> F(Delta) wild}), we have that  
 $\underline{\cal A}$ is not wild. 
 From (\ref{T: pregens de underline(A)}),  we know that, given $d\in \hueca{N}$, there is a minimal ditalgebra  ${\cal B}$ and a full and faithful functor  $F : {\cal B}\g\Mod\rightmap{} \underline{\cal A}\g \Mod$. The functor $F$ is such that, for any indecomposable 
 $M \in  \underline{\cal A}\g\Mod$ with 
 $\Endol{M}\leq d$, there is some 
 $N \in  {\cal B}\g\Mod$ with $F(N) \cong M$. If  $M$ is pregeneric, it is isomorphic to a normal pregeneric $\underline{\cal A}$-module $\underline{M}$.  With the notation of (\ref{T: pregens de underline(A)}), we have: 
$$\mathfrak{G}FL_{\cal B}
\cong 
\mathfrak{G}L_{\underline{\cal A}}(\overline{T}\otimes_B-)\cong \mathfrak{Z}\otimes_{\underline{A}} (\overline{T}\otimes_B -)\cong (\mathfrak{Z}\otimes_{\underline{A}} \overline{T})\otimes_B-.$$

Given a generic $\Upsilon$-module $G$ in $\widetilde{\cal F}(\Theta)$, since $\mathfrak{G}$ is a ${\cal P}$-presentation of $\widetilde{\cal F}(\Theta)$, there is a pregeneric $\underline{\cal A}$-module  $M$ 
such that  $\mathfrak{G}(M)\cong  G$. Hence, $\mathfrak{G}(\underline{M})\cong G$.  

Let $d:=\Endol{G}$, then, by (\ref{P: endols y pregen normales}), we have 
$$\Endol{M}=\Endol{\underline{M}}\leq \Endol{\mathfrak{G}(\underline{M})}=\Endol{G}=d.$$ 
By (\ref{T: pregens de underline(A)}), there is a pregeneric $N\in {\cal B}\g\Mod$ with $F(N)\cong M$. Thus,    
$\mathfrak{G}F(N) \cong G$. 
In fact, 
$N = L_{\cal B}(Be_i \otimes_{Be_i} Q_i)$,  where  $Q_i$ is the field of fractions of  $Be_i$ for some $i$, see (\ref{D: los Qi's}). Consider the $\Upsilon\g B$-bimodule  $T:=\mathfrak{Z}\otimes_{\underline{A}}\overline{T}$ and the $\Upsilon\g k(x)$-bimodule $\underline{G}:=T\otimes_BN$. Then, we have 
an isomorphism of $\Upsilon$-modules  $G\cong \mathfrak{G}FL_{\cal B}(N)\cong T\otimes_BN=\underline{G}$, so $G$ and $\underline{G}$ have the same endolength. Moreover, we have an isomorphism of $\Upsilon\g k(x)$-bimodules 
$$\underline{G}=T\otimes_BN= (\mathfrak{Z} \otimes_{\underline{A}}\overline{T})\otimes_B Be_i \otimes_{Be_i}Q_i\cong Te_i \otimes_{Be_i} Q_i,$$ 
where $ Te_i = 
(\mathfrak{Z} \otimes_{\underline{A}} \overline{T})e_i$ is a  finitely generated right $Be_i$-module. So, there is some $g\in Be_i$ such that $Z_G:=Te_i\otimes_{Be_i}(Be_i)_g$ is a  free  finitely generated right $(Be_i)_g$-module. Moreover, the field of fractions $Q_G$ of the rational algebra  $\Gamma_G:=(Be_i)_g$ coincides with $k(x)=Q_i$, and we have an isomorphism of $\Upsilon\g k(x)$-bimodules 
$$Z_G\otimes_{\Gamma_G}Q_G=Te_i\otimes_{Be_i}(Be_i)_g\otimes_{(Be_i)_g}Q_i\cong Te_i\otimes_{Be_i}Q_i\cong \underline{G}.$$
We have that $\dim_{k(x)}\underline{G}=\Endol{\underline{G}}=\Endol{G}$ and, from the last displayed formula, the rank of $Z_G$ is $\dim_{k(x)}\underline{G}$, so (1) is proved.
\medskip

\noindent(2):  Having in mind (\ref{R: F exact}), we have that  $L_{\cal B}:B\g \Mod\rightmap{}{\cal B}\g\Mod$ preserves isoclasses, indecomposability and maps almost split sequences onto almost split conflations. Since $\mathfrak{G}$ and $F$ are full and faithful functors,  they all preserve isoclasses and indecomposability. Then,   the functor $T\otimes_B-\cong\mathfrak{G} F L_{\cal B}$ preserves isoclasses and indecomposability. The localization morphism $\phi:Be_i\rightmap{}(Be_i)_g=\Gamma_G$ is an epimorphism of algebras, so it induces a full and faithful functor $F_\phi: \Gamma_G\g\Mod\rightmap{}Be_i\g\Mod$ which can be composed with the embedding functor $F_\pi : Be_i\g\Mod \rightmap{}B\g\Mod$, induced by the projection $\pi:B\rightmap{}Be_i$. Then the functor $Z_G\otimes_{\Gamma_G}-\cong  \mathfrak{G} F L_{\cal B}F_\pi F_\phi$ preserves isoclasses and indecomposables.
\medskip

\noindent(3): It follows from (\ref{R: F exact}), that the composition $\mathfrak{G}F:{\cal B}\g\mod\rightmap{}\widetilde{\cal F}(\Theta)\g\mod$ is a full and faithful exact functor between Krull-Schmidt categories with exact structures with almost split conflations. Set $E_n^\lambda:=L_{\cal B}F_\pi F_\phi(\Gamma_G/\langle x-\lambda\rangle^n)$, where $x-\lambda$ is irreducible in $\Gamma_G$ and $n\in \hueca{N}$. Then, $\{E_n^\lambda\}_{n,\lambda}$ are non-isomorphic indecomposable ${\cal B}$-modules and the   following are almost split conflations in ${\cal B}\g\mod$ 
$$\begin{matrix} E_1^\lambda\rightmap{}E_2^\lambda\rightmap{}E_1^\lambda\\
\\ 
 E_n^\lambda\rightmap{}E_{n+1}^\lambda\oplus E_{n-1}^\lambda\rightmap{}E_n^\lambda\hbox{ \ for } n>1.
\end{matrix}$$

Since $\mathfrak{G} F(E_1^\lambda)\cong Z_G\otimes_{\Gamma_G}\Gamma_G/\langle x-\lambda\rangle$, we have that $\dim_k \mathfrak{G} F(E_1^\lambda)$ coincides with the rank of $Z_G$, for all $\lambda$. Hence, from (\ref{T: casi todo M en A-mod tq tau(M) cong M}), we know that there  are only finitely many indecomposable ${\cal B}$-modules of the form $E_1^\lambda$ such that $\mathfrak{G}F(E_1^\lambda)\not\cong \tau \mathfrak{G}F(E_1^\lambda)$.   If this occurs for $\lambda_1,\ldots,\lambda_t\in k$, we can replace the polynomial $g$ used at the beginning  of this proof by $g':=g(x-\lambda_1)\cdots(x-\lambda_t)$ to obtain a rational algebra $\Gamma'_G=(Be_i)_{g'}$ such that $\mathfrak{G}F(E_1^\lambda)\cong \tau \mathfrak{G}F
(E_1^\lambda)$, for all such indecomposables $E_1^\lambda$ of $\Gamma'_G$, and we adjust $Z'_G:=Te_i\otimes_{Be_i}(Be_i)_{g'}$ accordingly. Moreover, (1) and (2), still hold for this $Z'_G$. So we can assume that $\Gamma_G$ and $Z_G$ already satisfy that 
$\mathfrak{G}F(E_1^\lambda)\cong \tau \mathfrak{G}F(E_1^\lambda)$, for all irreducible $x-\lambda$ of $\Gamma_G$, 
  and we can apply \cite[(32.7)]{BSZ}, to obtain that 
$\mathfrak{G}F$ preserves almost split conflations. Then,   so does the composition  $Z_G\otimes_{\Gamma_G}-\cong \mathfrak{G}FL_{\cal B}F_\pi F_\phi$.
\end{proof}

In the following, we keep the notation of $\S7$ and of the preceding proof.

\begin{proposition}\label{P: F(Delta) not wild --> F(Delta) gen tame}
Assume that we have a ${\cal P}$-presentation $\mathfrak{G}:\underline{\cal A}\g\Mod\rightmap{} \widetilde{\cal F}(\Theta)$ such that $\mathfrak{G}(S_v)\cong \Theta_v$, for $v\in {\cal P}$, and  ${\cal F}(\Theta)$ is not wild. Then,   for any $d\in \hueca{N}$,
almost every indecomposable $\Upsilon$-module  $M\in {\cal F}(\Theta)$ with  $\dim_kM\leq d$ is  of the form $M\cong Z_G\otimes_{\Gamma_G}N$, for some generic $\Upsilon$-module  $G$ for ${\cal F}(\Theta)$ with $\Endol{G}\leq d$ and some indecomposable $N\in \Gamma_G\g\mod$ with $\dim_k N\leq d/\Endol{G}$. 
\end{proposition}

\begin{proof} Fix some $d\in \hueca{N}$. Since $\underline{\cal A}$ is not wild, we can apply (\ref{T: pregens de underline(A)}) to $d$ to obtain a   minimal ditalgebra ${\cal B}$ and a  composition of functors
$$B\g\Mod\rightmap{L_{\cal B}}{\cal B}\g\Mod\rightmap{F}\underline{\cal A}\g\Mod\rightmap{\mathfrak{G}}\widetilde{\cal F}(\Theta),$$
isomorphic to the tensor functor $T\otimes_B-$, where $T=\mathfrak{Z}\otimes_{\underline{A}}\overline{T}$. Moreover, 
 we know that, if $G_1,\ldots,G_n$ are, up to isomorphism, all the generic $\Upsilon$-modules in $\widetilde{\cal F}(\Theta)$  with $\Endol{G_i} \leq d$, then there are  $M_1,\ldots,M_n$ pregenerics in $\underline{\cal A}\g\Mod$ with $\mathfrak{G}(M_i)\cong G_i$. Moreover, from (\ref{P: endols y pregen normales})(2), we have $\Endol{M_i}\leq \Endol{G_i}\leq d$. So,  as in the last proof, there is some $N_i\in {\cal B}\g\Mod$ with $FL_{\cal B}(N_i)\cong M_i$. Here, $N_i\cong Q_i$, when we consider the field of fractions $Q_i$ of the rational algebra $\Gamma_{G_i}=(Be_i)_{g_i}$ as a $B$-module, so  $G_i\cong\mathfrak{G}FL_{\cal B}(Q_i)\cong Z_{G_i}\otimes_{\Gamma_{G_i}}Q_i$. 
 
Now, if $M\in {\cal F}(\Theta)$ with $\dim_k M\leq d$, there is some $M'\in \underline{\cal A}\g\Mod$ with $M\cong \mathfrak{G}(M')$ and, by (\ref{P: endols y pregen normales})(3), we have that $\dim_kM'\leq \dim_k\mathfrak{G}(M')=\dim_kM\leq d$. From (\ref{T: pregens de underline(A)})(3), $M'$ 
 is isomorphic to $\overline{T}\otimes_{B}N$, for some indecomposable $Be_i$-module $N$ with finite dimension over $k$.   In fact, since every reduction functor is dimension-controlling, we know that $\dim_kN\leq \dim_kM'\leq d$. Now, recall from the preceding proof that each bimodule  $Z_{G_i}=Te_i\otimes_{Be_i}(Be_i)_{g_i}$, associated to a given generic module $G_i$ for ${\cal F}(\Theta)$,  involves a localization of $k$-algebras $\phi:Be_i\rightmap{}(Be_i)_{g_i}$ whose restriction functor $F_\phi:\Gamma_G\g\mod=(Be_i)_g\g\mod\rightmap{}Be_i\g\mod$ hits all isoclasses of indecomposable $Be_i$-modules with  dimension bounded by $d$, with only finitely many exceptions. 

From the construction of the bimodules $Z_{G_1},\ldots, Z_{G_n}$ in the last proof,   we get that for almost every  indecomposable $\Upsilon$-module $M\in {\cal F}(\Theta)$ with $\dim_k M\leq d$ there are some  $G\in \{G_1,\ldots,G_n\}$  and $N\in \Gamma_G\g\mod$ such that 
$M\cong Z_G\otimes_{\Gamma_G}N$. Since $Z_G$ is free by the right with rank equal to the endolength of $G$, we get 
$\dim_kM=\dim_k N\times \Endol{G}$. 
Our claim follows from this. 
\end{proof}

\begin{remark}\label{R: final remark proof of (1.6)} Finally, we come to the proof of 
Theorem (\ref{T: main thm}). In the context of (\ref{R: context for proof of thm (1.5)}), we remarked that $\widetilde{\cal F}(\Theta)$ admits a ${\cal P}$-presentation $\mathfrak{G}:\underline{\cal A}\g\Mod\rightmap{}\widetilde{\cal F}(\Theta)$ and we derived from this that ${\cal F}(\Theta)$ is generically tame iff ${\cal F}(\Theta)$ is tame. Then, Theorem (\ref{T: main thm}), follows from  a particular case of (\ref{P: F(Delta) not wild --> F(Delta) gen tame}), again after some trivial notational variation.  
\end{remark}

\noindent{\bf Data availability}

No data was used for the research described in this article.

\hskip2cm

\vbox{\noindent R. Bautista\\
Centro de Ciencias Matem\'aticas\\
Universidad Nacional Aut\'onoma de M\'exico\\
Morelia, M\'exico\\
raymundo@matmor.unam.mx\\}

\vbox{\noindent E. P\'erez\\
Facultad de Matem\'aticas\\
Universidad Aut\'onoma de Yucat\'an\\
M\'erida, M\'exico\\
jperezt@correo.uady.mx\\}

\vbox{\noindent L. Salmer\'on\\
Centro de Ciencias  Matem\'aticas\\
Universidad Nacional Aut\'onoma de M\'exico\\
Morelia, M\'exico\\
salmeron@matmor.unam.mx\\}

\end{document}